\documentclass[11pt]{article}

\usepackage[a4paper,textwidth=15cm,textheight=23cm,centering]{geometry}
\usepackage{amsmath,amssymb,amsfonts,amsthm}
\usepackage{url}
\newcommand{\Gr}{\mbox{gr}\,}
\newcommand{\PX}{2^X \setminus \{\emptyset \}}
\newcommand{\varchi}{\raisebox{2.5pt}{$\chi$}}

\newcommand{\Q}{\mathbb{Q}}

\newcommand{\R}{\mathbb{R}}
\newcommand{\TAK}{T^{\lower-0.05cm \hbox{$\scriptstyle A$}}_{\lower 0.05cm \hbox{$\scriptstyle K$}}}
\newcommand{\TAKadj}{\widetilde T^{\lower-0.05cm \hbox{$\scriptstyle A$}}_{\lower 0.05cm \hbox{$\scriptstyle K$}}}
\newcommand{\TK}{T_{\lower 0.05cm \hbox{$\scriptstyle K$}}}
\newcommand{\keywords}[1]{%
  \par\smallskip\noindent\begingroup
  \def\and{\unskip; }\textbf{Keywords. }#1\par\endgroup}
\newcommand{\subclass}[1]{%
  \par\smallskip\noindent\begingroup
  \def\and{\unskip; }\textbf{Mathematics Subject Classification (2020). }#1\par\endgroup}

\theoremstyle{plain}
\newtheorem{theorem}{Theorem}[section]
\newtheorem{lemma}[theorem]{Lemma}
\newtheorem{proposition}[theorem]{Proposition}
\newtheorem{corollary}[theorem]{Corollary}
\theoremstyle{definition}
\newtheorem{definition}[theorem]{Definition}
\theoremstyle{remark}
\newtheorem{remark}[theorem]{Remark}
\numberwithin{equation}{section}

\begin{document}

\title{Existence and qualitative theory for nonlinear accretive evolutions with Carath\'eodory forcing}
\author{Dieter Bothe}
\date{\small Mathematical Modeling and Analysis, Department of Mathematics,\\
Technische Universit\"at Darmstadt, Peter-Gr\"unberg-Str. 32,
64287 Darmstadt, Germany\\
\texttt{bothe@mma.tu-darmstadt.de}}
\maketitle

\begin{abstract}
Let $A$ be $m$-accretive in a real Banach space $X$.
Given closed sets $K(t)\subset X$, we develop an existence
theory for mild solutions of the constrained problem
\[
 u'(t)+Au(t)\ni f(t,u(t)),
 \qquad u(t)\in K_A(t):=K(t)\cap\overline{D(A)},
\]
where the forcing is Carath\'eodory on the effective tube $K_A$.
We also assume joint measurability of $f$ on the moving graph or a
Carath\'eodory extension to a fixed cylinder.
Under a linear-growth hypothesis and the corresponding regular- and exceptional-time subtangential conditions,
we reduce the problem to a
bounded effective tube with a common integrable bound.
A first central result constructs a closed separable resolvent-invariant
subspace preserving distances to the moving sets and the subtangential
conditions, permitting use of the Scorza--Dragoni property for the reduced forcing. A second
central ingredient is a two-level approximation. Mild solutions
$u_\varepsilon$ driven by $w_\varepsilon$ are accompanied by auxiliary
paths $v_\varepsilon$ and current-time selectors $x_\varepsilon$ satisfying
$x_\varepsilon(t)\in K_A(t)$ and
$w_\varepsilon(t)=f(t,x_\varepsilon(t))$ for almost every $t$ in a closed
regularity set of almost full measure. For every $t$ there is also
$\sigma_\varepsilon(t)\in[(t-\varepsilon)^+,t]$ with
$v_\varepsilon(\sigma_\varepsilon(t))\in
K_A(\sigma_\varepsilon(t))$. The current-time selectors identify the
limiting forcing, while the lagged nodes guarantee viability.
This yields
well-posedness under time-dependent constraints for such forcings $f$ when
they are locally Lipschitz in the state variable.
Further viability results follow under various compactness conditions.
Application of the abstract viability theory yields comparison principles, nonautonomous Lyapunov
pairs, periodic solutions, and time-dependent bounds for abstract
reaction--diffusion systems. 

\subclass{Primary 47J35 \and Secondary 34G25, 35B51, 35K57, 47H20}
\end{abstract}

\section{Introduction}\label{s:intro}
Let $X$ be a real Banach space and $A:D(A)\subset X\to2^X$ be
$m$-accretive. Let $J=[0,a]\subset \R$, and let $K:J\to2^X$ have closed values. Set $K_A(t):=K(t)\cap\overline{D(A)}$.  
We call $K_A$ the \emph{effective tube} associated with $K$ and
$\Gr(K_A)$ denotes its graph.  We consider Carath\'eodory forcings $f:\Gr(K_A)\to X$ and assume in
addition either joint measurability on the moving effective graph or the
existence of a Carath\'eodory extension to a fixed cylinder.  Given
$(t_0,x_0)\in\Gr(K_A)$ with $t_0<a$, we study mild solutions
$u:[t_0,a]\to\overline{D(A)}$ of
\begin{equation}\label{sivp4}
 u'(t)+Au(t)\ni f(t,u(t)),\qquad u(t_0)=x_0.
\end{equation}
Since $f$ is defined only on $\Gr(K_A)$, such a solution is understood to
satisfy the time-dependent constraints $u(t)\in K_A(t)$ for every $t\in[t_0,a]$.

Existence of viable solutions under time-dependent constraints is the
organizing principle of the present paper.  The qualitative results follow by
choosing appropriate tubes that encode order, dissipation, return conditions, or
pointwise bounds.  Solution-dependent upper and lower tubes then yield comparison, moving
epigraphs yield Lyapunov inequalities, viable return tubes yield periodic
solutions under well-posedness and fixed-point geometry, and moving
rectangles yield bounds for reaction--diffusion systems.

Before the existence construction, the linear-growth and subtangential hypotheses
are used to reduce the problem, for each prescribed bounded set of initial values, to
a bounded effective tube on which a single finite-valued $L^1$ function controls both
the forcing and the exceptional-time estimate.  
The existence theory is then developed
entirely for the bounded effective tube and transferred back to the original
tube at the end of Section~\ref{s:exist}.

The underlying viability theorem has two principal features.  First,
it constructs from the bounded data a closed separable subspace $Y\subset X$ such
that $J_\lambda Y\subset Y$ for every $\lambda>0$, the forcing is $Y$-valued
on the reduced graph outside one null set, and the regular and exceptional
subtangential conditions pass to the part $A\cap(Y\times Y)$.  Since the
construction also preserves the distances to the constraint sets $K_A(t)$ at every time,
the argument can be carried out in the separable subspace and provides viable solutions
in the original space.  
Second, the approximate solutions contain lagged constrained nodes, which
recover membership in the moving tube, together with current-time selectors
\[
 x_\varepsilon(t)\in K_A(t),\qquad
 w_\varepsilon(t)=f(t,x_\varepsilon(t))
 \quad\hbox{for a.e. }t\in E_\varepsilon.
\]
The current-time selectors permit forcing values to be compared at the same time and thereby
reconcile moving constraints with strong measurability in time and local (one-sided)
Lipschitz continuity in the state variable.  The corresponding result for static
constraints was obtained in \cite{Bo-JEE03}, while moving constraints with
continuous forcing were treated in \cite{Bo9}.
Without a local Lipschitz assumption, compact-semigroup, forced-trajectory compactness,
and measure-of-noncompactness hypotheses provide alternative existence mechanisms for
mild solutions under time-dependent constraints.

\medskip
\noindent
\textit{Nonlinear semigroups and flow invariance.}
The analytic background is the generation theory for accretive operators;
see the Crandall--Liggett theorem \cite{CrLi} and the nonlinear-semigroup
treatments \cite{Barbu,Barbu2010,BCP,ItoKa,Mi}.  In ordered spaces,
$T$-accretivity expresses contraction of the positive part, while complete
accretivity gives contraction for a broad class of convex integral
functionals \cite{BenCr}.  Quantitative results for implicit Euler
approximations are given in \cite{BeSh24}.

Flow-invariance and viability criteria originate with Nagumo's tangency
theorem \cite{Nagumo}; in the terminology introduced in Section~\ref{s:subt},
this is expressed by a subtangential condition at the boundary of the constraint set.
The classical tangency principle subsequently developed into an extensive theory of
invariance and viability for ordinary differential equations and differential
inclusions, in finite- and infinite-dimensional state spaces; see, for example,
Amann \cite{AmannODE}, Aubin and Cellina \cite{AubinCellina}, Pavel
\cite{PavelBook}, and Deimling \cite{MDE}.
Semigroup-adapted conditions in Banach spaces were
developed by Pavel \cite{Pavel77}, Pierre \cite{Pierre78}, and Vrabie
\cite{Vra1}.  Finite-dimensional differential inclusions on graphs of
time-dependent constraints, including the distinction between regular and
exceptional times and measurability obstructions for graph-defined right-hand
sides, were studied in \cite{Bo92}.  For $m$-accretive evolutions with
continuous forcing, static and moving constraints were treated in \cite{Bo9},
with extensions to range conditions in \cite{Bo-JEE05}.  Carath\'eodory
forcing was studied for nonlinear $m$-accretive evolutions in
\cite{CV,Bo-JEE03}.  Related flow-invariance results for nonlinear partial
differential delay equations with nonlinear multivalued accretive operators
and history-dependent perturbations defined on restricted subsets of the
history space were obtained by Ruess \cite{Ruess09} and, in the
nonautonomous case, by Ghavidel and Ruess \cite{GhavidelRuess12}.

Complementary criteria and solution concepts appear in
\cite{CaDaFr18,CaDaFr20,CGK22,CGMS22,Tol23,Tol24,JiKe25}. 

\par\medskip
\noindent
\textit{Comparison, quasimonotonicity, and monotone dynamics.}
The classical finite-dimensional comparison theory is based on the
quasimonotonicity conditions of M.~M\"uller and Kamke
\cite{Mueller,Kamke}.  Extensions to ordered topological vector spaces and
Banach spaces are found in \cite{Volkmann72,DeLak79,Herzog01}; monotone
dynamical systems and their cooperative or competitive variants are
surveyed in \cite{Smith95,Hirsch}.  Infinite-dimensional semilinear
comparison and invariant-set methods, including reaction--diffusion and
delay equations, were developed in \cite{MartinSmith90,MartinSmith91}.

Here the diffusion may be nonlinear and multivalued, so order preservation
is imposed through $m$-$T$-accretivity, and the forcing has
Carath\'eodory time--state regularity and is defined only on a moving tube.  Quasimonotonicity makes the
solution-dependent upper tube subtangential and hence yields ordered
companion solutions even without uniqueness.  Under well-posedness this gives order
preservation.  A direct positive-part argument gives comparison for two
given locally Lipschitz solutions.

\medskip
\noindent
\textit{Lyapunov couples and dissipation.}
Lyapunov methods for nonlinear contraction semigroups go back to Pazy
\cite{Pazy81}.  Lyapunov couples and their iteration into Lyapunov sequences
are developed systematically in \cite[Chapter~19]{BCP}; see also
\cite{Barbu2010}.  Characterizations of Lyapunov
pairs by Hamilton--Jacobi, contingent-derivative, and nonsmooth criteria
were obtained in \cite{KoSo,CaMot,AHT12,AHN18}.  See also \cite{STD26} for
differential inclusions.  The nonautonomous result below
applies the viability theorem to a moving
epigraph and separates regular-time subtangentiality from exceptional-time
control of the epigraph motion.

\medskip
\noindent
\textit{Abstract reaction--diffusion systems.}
Invariant order intervals provide positivity, bounds, continuation, and
periodic solutions.  Semilinear cooperative and competitive systems are
treated in \cite{MartinSmith90,MartinSmith91,Smith95}.  For nonlinear
componentwise diffusion, flow invariance and viability methods were used in
\cite{Bo9,Bo10}; further background and measurable-reaction results are
contained in \cite{BoWi12}, while the periodic theory for static constraints
without uniqueness was developed in \cite{Bo12}.  Section~\ref{s:react} considers the moving
rectangles
$K^y(t)=\{u\in L^p(\Omega;\mathbb R^m):0\leq u\leq y(t)\}$, where
quasipositivity controls the lower faces, and an upper-envelope inequality
controls the upper faces.
Its periodic application uses a moving return rectangle, i.e.\ a tube for which $K^y(T)\subset K^y(0)$.

\section{Preliminaries}\label{s:prel}
This section fixes notation and collects auxiliary results used in the
subsequent analysis.  For the
Banach space material we refer to \cite{MDE}; for Banach lattices to
\cite{peressini,Schaefer}; and for accretive operators, nonlinear semigroups,
and mild solutions to \cite{Barbu,Barbu2010,BCP,ItoKa,Mi}.
Let $\mathcal L(J)$ denote the Lebesgue $\sigma$-algebra on $J$.  Unless
stated otherwise, time maps are $\mathcal L(J)$-measurable, state spaces carry
their Borel $\sigma$-algebras, and product measurability refers to
$\mathcal L(J)\otimes\mathcal B(X)$, with the trace understood on subsets.
Every $L^1_+$ bound has a finite-valued measurable representative;
changes on a null set do not affect the arguments.

\subsection{Banach-space notation and accretive operators}
Throughout, $X$ is a real Banach space with norm $\|\cdot\|$, and $X^*$ is
its dual.  For $x\in X$ and $K\subset X$, put
$d(x,K):=\inf_{y\in K}\|x-y\|$ with $d(x,\emptyset)=\infty$.  
The open and closed balls with centre $x$ and radius $r>0$ are denoted by
$B_r(x)$ and $\overline B_r(x)$, respectively.

Recall that $X$ is called \emph{uniformly convex} if, for every
$\varepsilon>0$, there is $\delta>0$ such that
$\|x\|,\|y\|\leq1$ and $\|x+y\|>2-\delta$ implies $\|x-y\|<\varepsilon$.
Every uniformly convex Banach space is reflexive.
The normalized duality mapping is
\[
 \mathcal F:X\rightarrow 2^{X^*}\setminus\{\emptyset\},
 \qquad
 \mathcal F(x):=\{x^*\in X^*:x^*(x)=\|x\|^2,
                         \ \|x^*\|=\|x\|\}.
\]
The upper and lower semi-inner products, defined by
\[
 (x,y)_+:=\max_{y^*\in\mathcal F(y)}y^*(x),
 \qquad
 (x,y)_-:=\min_{y^*\in\mathcal F(y)}y^*(x),
\]
satisfy $|(x,y)_\pm|\leq\|x\|\,\|y\|$.  If $X^*$ is strictly convex,
then $\mathcal F$ is single-valued; if $X^*$ is uniformly convex, then this
single-valued map is uniformly continuous on bounded subsets of $X$.  In the
single-valued case the two semi-inner products coincide.  The resulting
semi-inner product is linear in its first argument, but generally not in its
second argument.

Let $A:D(A)\subset X\to2^X\setminus\{\emptyset\}$ be a possibly
multivalued operator.  Its domain, range, and graph are
$D(A):=\{x\in X:Ax\neq\emptyset\}$,
$R(A):=\bigcup_{x\in D(A)}Ax$, and
$\Gr(A):=\{(x,y)\in X\times X:x\in D(A),\ y\in Ax\}$.
The operator $A$ is \emph{accretive} if
\begin{equation}\label{accretive-semi-inner}
 (y-\bar y,x-\bar x)_+\geq0
 \quad\text{for all }x,\bar x\in D(A),\ y\in Ax,\ \bar y\in A\bar x.
\end{equation}
For $x,y\in X$ define the bracket
\begin{equation}\label{ordinary-bracket}
 [x,y]:=\lim_{h\to0+}\frac{\|x+hy\|-\|x\|}{h}.
\end{equation}
Then accretivity is
equivalent to $[x-\bar x,y-\bar y]\geq0$ for all
$x,\bar x\in D(A)$, $y\in Ax$, and $\bar y\in A\bar x$.  An accretive
operator is called \emph{$m$-accretive} if $R(I+\lambda A)=X$ for every
$\lambda>0$.  Its resolvent is denoted by $J_\lambda:=(I+\lambda A)^{-1}$.

For a bounded set $B\subset X$, the Hausdorff measure of noncompactness is
\[
 \beta(B):=\inf\{r>0:B\text{ admits a finite cover by balls of radius }r\},
\]
and the Kuratowski measure of noncompactness is
\[
 \alpha(B):=\inf\{d>0:\, B\text{ admits a finite cover by sets of diameter at most }d\}.
\]
Both vanish exactly on relatively compact sets.  They are monotone under
inclusion, homogeneous, subadditive under standard set addition, and invariant
under closure and convex hull.  For either measure $\gamma\in\{\alpha,\beta\}$,
$\gamma(B_1\cup B_2)=\max\{\gamma(B_1),\gamma(B_2)\}$.  In particular, adding or deleting finitely many points does not change the
measure of noncompactness.  Moreover, $\beta(B)\leq\alpha(B)\leq2\beta(B)$ and
$\beta(B+r\overline B_1(0))\leq\beta(B)+r$.  Accordingly, any statement formulated with $\beta$ may be reformulated
with $\alpha$, with the usual possible factor $2$ in numerical estimates.

\subsection{Semigroups and mild solutions}
Let $J=[0,a]$, let $A$ be $m$-accretive in $X$, let $w\in L^1(J;X)$, and
consider
\begin{equation}\label{QA}
 u'(t)+Au(t)\ni w(t)\quad\text{on }J,
 \qquad u(0)=u_0\in\overline{D(A)}.
\end{equation}
Given $\varepsilon>0$, an $\varepsilon$-discretization of \eqref{QA}
consists of a partition
\[
 0=t_0<t_1<\cdots<t_N=a,
 \qquad \max_{1\leq k\leq N}(t_k-t_{k-1})<\varepsilon,
\]
and vectors $w_1,\ldots,w_N\in X$ such that
$\sum_{k=1}^N\int_{t_{k-1}}^{t_k}\|w(t)-w_k\|\,dt<\varepsilon$.
For such a discretization, let $v_0=u_0$ and define $v_1,\ldots,v_N$ by the backward Euler scheme \[ \frac{v_k-v_{k-1}}{t_k-t_{k-1}}+Av_k\ni w_k, \qquad k=1,\ldots,N. \] Since $A$ is $m$-accretive, these vectors are uniquely determined. The step function associated with the discretization is defined by
$v_\varepsilon(0)=u_0$ and
$v_\varepsilon(t)=v_k$ for $t\in(t_{k-1},t_k]$.

A function $u\in C(J;X)$ is a \emph{mild solution} of \eqref{QA} if, for
every $\varepsilon>0$, an $\varepsilon$-discretization can be chosen such that
$\sup_{t\in J}\|u(t)-v_\varepsilon(t)\|<\varepsilon$.  We write $u(\cdot;s,x,w)$ for the solution starting from $x$ at time $s$.
The restriction of $w$ to the relevant interval is understood.  Uniqueness
implies $u(t;s,x,w)=u(t;r,u(r;s,x,w),w)$ for
$0\leq s\leq r\leq t\leq a$.

Every mild solution of \eqref{QA} is also an \emph{integral solution}; that
is,
\begin{equation}\label{integral-solution-inequality}
 \|u(t)-x\|\leq\|u(s)-x\|
 +\int_s^t[u(\tau)-x,w(\tau)-y]\,d\tau
\end{equation}
for all $0\leq s\leq t\leq a$ and $(x,y)\in\Gr(A)$.

\begin{proposition}\label{mild-difference}
Let $A$ be $m$-accretive in the real Banach space $X$.  For every
$u_0\in\overline{D(A)}$ and $w\in L^1(J;X)$, problem \eqref{QA} has a unique mild solution $u$, and $u\in C(J;\overline{D(A)})$.  If $u$ and $v$
are mild solutions with right-hand sides $f,g\in L^1(J;X)$, respectively,
then
\begin{align*}
 \|u(t)-v(t)\|
 &\leq\|u(s)-v(s)\|+\int_s^t\|f(\tau)-g(\tau)\|\,d\tau,\\
 \|u(t)-v(t)\|
 &\leq\|u(s)-v(s)\|+\int_s^t
 [u(\tau)-v(\tau),f(\tau)-g(\tau)]\,d\tau,\\
 \|u(t)-v(t)\|^2
 &\leq\|u(s)-v(s)\|^2
 +2\int_s^t(f(\tau)-g(\tau),u(\tau)-v(\tau))_+\,d\tau
\end{align*}
for $0\leq s\leq t\leq a$.
In particular, the solution map $w\mapsto u(\cdot;s,x,w)$ is continuous
from $L^1([s,a];X)$ to $C([s,a];X)$.
\end{proposition}

A family $S(t):D\to D$, $t\geq0$, on a closed set $D\subset X$ is a
\emph{strongly continuous semigroup of contractions} if
$S(0)=I$, $S(t+s)=S(t)S(s)$, $\lim_{t\to0+}S(t)x=x$, and
$\|S(t)x-S(t)y\|\leq\|x-y\|$ for all $t\geq0$ and $x,y\in D$.
It is \emph{compact} if $S(t)$ maps bounded subsets of $D$ into relatively
compact subsets of $X$ for every $t>0$.  It is \emph{equicontinuous} if, for
every bounded $B\subset D$, the family
$\{S(\cdot)x:x\in B\}$ is equicontinuous at every positive time.

The homogeneous problem associated with an $m$-accretive operator defines
such a semigroup on $D=\overline{D(A)}$ by $S(t)x:=u(t;0,x,0)$.  We say that $-A$ generates $S$.  A theorem of Br\'ezis states, in this
setting, that compactness of the semigroup is equivalent to compactness of
the resolvents together with the above equicontinuity property; see
\cite{Barbu}.

\subsection{Banach lattices and $T$-accretivity}
An ordered vector space is a real vector space $X$ endowed with a partial
order compatible with addition and multiplication by nonnegative scalars.
Its positive cone is $X_+:=\{x\in X:x\geq0\}$, so that
$x\leq y$ if and only if $y-x\in X_+$.  For $a\leq b$, the order interval
is $[a,b]:=\{x\in X:a\leq x\leq b\}$.  An ordered vector space is a \emph{vector lattice} if every pair $x,y$ has a
least upper bound $x\vee y$ and a greatest lower bound $x\wedge y$.  Put
$x^+:=x\vee0$, $x^-:=(-x)\vee0=(-x)^+$, and $|x|:=x^++x^-$.  Then $x=x^+-x^-$.  A vector lattice which is a Banach space is a
\emph{Banach lattice} if
\begin{equation}\label{lattice-cond}
 |x|\leq|y|\quad\Rightarrow\quad\|x\|\leq\|y\|.
\end{equation}
In particular, $\||x|\|=\|x\|$, and the lattice operations are continuous.
Throughout the paper $|x|$ denotes the lattice modulus, whereas $\|x\|$
denotes the norm.

Standard examples are $C(\Omega;\mathbb R^n)$ and
$L^p(\Omega;\mathbb R^n)$, $1\leq p\leq\infty$, with componentwise order
and a lattice norm.  The order in $L^p$ is understood almost everywhere.

\begin{proposition}\label{lattice-facts}
Let $x,y$ belong to a Banach lattice $X$.  Then:
\begin{itemize}
 \item[(i)] $|x+y|\leq|x|+|y|$;
 \item[(ii)] $(x+y)^+\leq x^++y^+$;
 \item[(iii)] $\|x^+\|\leq\|x\|$ and $\|x^-\|\leq\|x\|$;
 \item[(iv)] if $y\geq0$, then $\|(x-y)^+\|\leq\|x^+\|$;
 \item[(v)] $d(x,y+X_+)=\|(x-y)^-\|=\|(y-x)^+\|$;
 \item[(vi)] $(x+y)^+=(x\vee(-y))+y$;
 \item[(vii)] $|x^+-y^+|\leq|x-y|$ and, consequently,
 $\|x^+-y^+\|\leq\|x-y\|$.
\end{itemize}
\end{proposition}
\begin{proof}
Items (i), (ii), (vi), and (vii) are standard lattice identities; see
\cite[Proposition~1.2]{peressini}.  Items (iii) and (iv) follow from
$0\leq x^+,x^-\leq|x|$, $(x-y)^+\leq x^+$ for $y\geq0$, and
\eqref{lattice-cond}.  For (v), put $q=x-y$.  If $z\in X_+$, then
$q^-\leq(q-z)^-\leq|q-z|$, whereas $z=q^+$ gives $q-z=-q^-$.  Taking the
infimum over $z$ proves the assertion.
\end{proof}

An operator $A$ in a Banach lattice is \emph{$T$-accretive} if
$\|(x-\bar x)^+\|\leq\|(x-\bar x+\lambda(y-\bar y))^+\|$
for all $\lambda>0$, $x,\bar x\in D(A)$, $y\in Ax$, and
$\bar y\in A\bar x$.  Define the positive bracket by
\begin{equation}\label{positive-bracket-definition}
 \begin{aligned}
 {}[x,y]_+
 &:=\lim_{h\to0+}
 \frac{\|(x+hy)^+\|-\|x^+\|}{h}\\
 &=\inf_{h>0}
 \frac{\|(x+hy)^+\|-\|x^+\|}{h}.
 \end{aligned}
\end{equation}
The positive bracket has the standard properties
\[
 |[x,y]_+|\leq\|y\|,\qquad
 [x,y+z]_+\leq[x,y]_++\|z\|.
\]
For every fixed $y\in X$, the map $x\mapsto[x,y]_+$ is upper
semicontinuous; for every fixed $x\in X$, the map $y\mapsto[x,y]_+$ is
$1$-Lipschitz.  These elementary properties are standard; see, e.g.,
\cite{BCP}.

Then $A$ is $T$-accretive if and only if
$[x-\bar x,y-\bar y]_+\geq0$ for all admissible
$x,\bar x,y,\bar y$.  
For an $m$-accretive operator, $T$-accretivity is equivalent to order
preservation of all resolvents,
$x\leq\bar x\Rightarrow J_\lambda x\leq J_\lambda\bar x$.
An operator which is both
$m$-accretive and $T$-accretive is called \emph{$m$-$T$-accretive}.  
The generated semigroup is then order preserving as well; see \cite{calvert}.

\begin{proposition}\label{positive-difference}
Let $A$ be an $m$-$T$-accretive operator on the Banach lattice $X$.
Let $f,g\in L^1([0,a];X)$, and let $u,v$ be the corresponding mild
solutions.  Then, for $0\leq s\leq t\leq a$,
\begin{equation}\label{positive-difference-estimate}
 \|(u(t)-v(t))^+\|
 \leq\|(u(s)-v(s))^+\|
 +\int_s^t[u(\tau)-v(\tau),f(\tau)-g(\tau)]_+\,d\tau.
\end{equation}
\end{proposition}
The integrand is measurable and bounded in absolute value by
$\|f(\tau)-g(\tau)\|$.  Proposition~\ref{positive-difference} is the
$T$-accretive analogue of Proposition~\ref{mild-difference}.  It follows by
applying the defining $T$-accretive inequality to common backward-Euler
discretizations and then passing to the mild limit; see, for example,
\cite{CalvertT,Barbu2010}.

\subsection{Carath\'eodory maps and joint measurability}
The classical Carath\'eodory setting is a cylindrical domain $J\times C$,
where $C$ is a nonempty metric space and $Y$ is a Banach space.  For a map
$F:J\times C\to Y$, the usual separate hypotheses are
\begin{equation}\label{fixed-product-carath-cond}
 \begin{aligned}
 &F(\cdot,x)\text{ is strongly measurable for every }x\in C,\\
 &F(t,\cdot)\text{ is continuous on $C$ for almost every }t\in J.
 \end{aligned}
\end{equation}
We use the same terminology on general domains.  For
$G\subset J\times X$, write
$G^x:=\{t\in J:(t,x)\in G\}$ and
$G_t:=\{x\in X:(t,x)\in G\}$ for its horizontal and vertical sections.
A map $f:G\to Y$ is called \emph{Carath\'eodory} if
\begin{itemize}
 \item for every $x\in X$, the map $f(\cdot,x):G^x\to Y$ is the
 restriction to $G^x$ of a strongly measurable map from $J$ to $Y$;
 \item there is a null set $N_f\subset J$ such that $f(t,\cdot)$ is
 continuous on $G_t$ for every $t\in J\setminus N_f$.
\end{itemize}
Thus on a fixed cylinder this is precisely
\eqref{fixed-product-carath-cond}.  The restriction formulation is convenient
on a general domain because a horizontal section $G^x$ need not be known
measurable a priori.  If $G^x\in\mathcal L(J)$, the first condition is
equivalent to strong measurability of the zero extension of $f(\cdot,x)$
from $G^x$ to $J$.

The Carath\'eodory property does not, on a general domain, imply
measurable dependence on $(t,x)$.  We call $f$ \emph{jointly measurable
on $G$} if $G\in\mathcal L(J)\otimes\mathcal B(X)$ and $f$ is measurable
with respect to the trace product $\sigma$-algebra.  Joint measurability is
therefore an additional property on a general domain.

On a fixed cylinder $J\times K$ with $K$ separable and $Y$ a separable
Banach space, separate Carath\'eodory regularity is sufficient up to
modification on one null set.  The next proposition makes this precise.

\begin{proposition}\label{cross-meas}
Let $K$ be a nonempty separable metric space and let $Y$ be a separable
Banach space.  Suppose that $f:J\times K\to Y$ is such that
$f(\cdot,x)$ is strongly measurable for every $x\in K$ and
$f(t,\cdot)$ is continuous on $K$ for almost every $t\in J$.  Then there
are a null set $N_f\subset J$ and an
$\mathcal L(J)\otimes\mathcal B(K)$-measurable map
$\widetilde f:J\times K\to Y$ such that
$\widetilde f=f$ on $(J\setminus N_f)\times K$ and
$\widetilde f(t,\cdot)$ is continuous on $K$ for every $t\in J$.
If $f(t,\cdot)$ is continuous on $K$ for every $t\in J$, then $f$
itself is product measurable.
\end{proposition}

\begin{proof}
Choose a null set $N_f\subset J$ such that $f(t,\cdot)$ is continuous
on $K$ for every $t\in J\setminus N_f$, and define
$\widetilde f=f$ on $(J\setminus N_f)\times K$ and
$\widetilde f=0$ on $N_f\times K$. Then
$\widetilde f(t,\cdot)$ is continuous for every $t\in J$, while
$\widetilde f(\cdot,x)$ is strongly measurable for every $x\in K$.
Choose a dense sequence $(x_n)$ in $K$ and, for $p\geq1$, let $x_p(x)$
be the first $x_n$ in $B_{1/p}(x)$. Then $x_p:K\to K$ is Borel
measurable and $x_p(x)\to x$. Hence the maps
$(t,x)\mapsto\widetilde f(t,x_p(x))$ are product measurable and,
by continuity of $\widetilde f(t,\cdot)$, converge pointwise to
$\widetilde f(t,x)$. Thus $\widetilde f$ is product measurable.
If state continuity of $f$ holds for every $t$, one may take
$N_f=\emptyset$ and $\widetilde f=f$.
\end{proof}

We also need measurability of compositions of $f$ with measurable graph selections.
\begin{proposition}\label{superposition}
Let $X$ and $Y$ be separable Banach spaces and let
$f:G\to Y$ be jointly measurable on $G$.  If $v:J\to X$ is measurable and
$(t,v(t))\in G$ for almost every $t\in J$, then the composition
$f(\cdot,v(\cdot))$, defined arbitrarily on the exceptional set, is strongly
measurable as a $Y$-valued map.
\end{proposition}
\begin{proof}
The map $z_v:J\to J\times X$, $z_v(t)=(t,v(t))$, is measurable, and
$E:=z_v^{-1}(G)$ belongs to $\mathcal L(J)$.  On $E$, the composition
$f\circ z_v$ is measurable with respect to the trace $\sigma$-algebra.
Extending it by zero to $J\setminus E$ gives a measurable
$Y$-valued map.  Since $Y$ is separable, it is strongly measurable.
\end{proof}

A closed set $I\subset J$ is called a \emph{Scorza--Dragoni set for
$f:G\to X$} if the restriction of $f$ to $(I\times X)\cap G$ is continuous
in the relative topology.  The map $f$ is \emph{almost continuous}, or has
the \emph{Scorza--Dragoni property}, if, for every $\varepsilon>0$, it has a
Scorza--Dragoni set $I_\varepsilon$ such that
$\lambda_1(J\setminus I_\varepsilon)\leq\varepsilon$.  Thus joint
continuity is recovered after deleting an arbitrarily small set of times.
For jointly measurable Carath\'eodory maps we use the following
Scorza--Dragoni theorem; see \cite{Kucia}.

\begin{lemma}\label{l1}
Let $X$ be separable and let $f:G\to X$ be a jointly measurable
Carath\'eodory map.  Then $f$ is almost continuous.
\end{lemma}

Now let $X$ be a Banach lattice. Let $D\subset X$ and $F:D\to X$.
Recall that the map $F$ is \emph{quasimonotone with respect to $X_+$} if
\begin{equation}\label{quasimonotonicity-definition}
 \lim_{h\downarrow0}\frac1h
 d\bigl(y-x+h(F(y)-F(x)),X_+\bigr)=0
 \qquad\text{for all }x,y\in D\text{ with }x\leq y.
\end{equation}
By Proposition~\ref{lattice-facts}(v), this is equivalent to
$\lim_{h\downarrow0}h^{-1}\bigl\|(x-y+h(F(x)-F(y)))^+\bigr\|=0$.
In $\mathbb R^n$ with componentwise order, this recovers
the usual Kamke condition: each component $F_i$ is nondecreasing in every
variable $x_j$, $j\neq i$; see \cite{Kamke,Mueller}.

For a map $f:G\subset J\times X\to X$, the corresponding
hypothesis is that $f(t,\cdot)$ is quasimonotone on the section $G_t$ for
almost every $t\in J$.  Changing $f$ on one common null set does not alter
the mild solutions.  

\subsection{An integrable local majorant}
The following elementary observation from \cite[Proposition~1]{Bo-JEE03}
replaces an $L^1$ coefficient by a pointwise local majorant.  Its short proof
is recalled for completeness.  Here $c$ is a fixed nonnegative measurable
representative, not an $L^1$ equivalence class.

\begin{proposition}\label{c-hat}
Let $c:J\to\mathbb R_+$ be measurable with $c\in L^1(J)$.  There exists a
measurable $\widehat c:J\to\mathbb R_+\cup\{\infty\}$, finite almost
everywhere and integrable, such that, for every $t_0\in J$, there is
$\delta(t_0)>0$ satisfying
$c(t_0)\leq\widehat c(t)$ for all
$t\in(t_0-\delta(t_0),t_0+\delta(t_0))\cap J$.
\end{proposition}
\begin{proof}
For $m\geq1$ put $B_m:=\{t\in J:m-1\leq c(t)<m\}$.  By outer regularity, choose relatively open sets $V_m\subset J$ such that
$B_m\subset V_m$ and
$\lambda_1(V_m)\leq\lambda_1(B_m)+2^{-m}/m$.  Set
$\widehat c(t):=\sum_{m\geq1}m\,\varchi_{V_m}(t)$.  Then
\begin{equation*}
 \int_J\widehat c(t)\,dt
\leq\sum_{m\geq1} \big( m\lambda_1(B_m)+ 2^{-m} \big) \leq \|c \|_{L^1 (J)} +a+1<\infty.
\end{equation*}
Thus $\widehat c$ is finite almost everywhere.  If $t_0\in B_m$, then
$t_0\in V_m$, so a relative neighbourhood of $t_0$ is contained in $V_m$;
on that neighbourhood,
$c(t_0)<m\leq\widehat c(t)$.
\end{proof}

For each fixed $0\leq c\in L^1(J)$, we henceforth denote this
extended-valued majorant by $\widehat c$, and choose a finite-valued
representative $\widehat c^{\,\circ}$ equal to $\widehat c$ almost everywhere. Since the two functions agree almost everywhere,
any estimate obtained by
integrating the pointwise local majorant $\widehat c$ remains valid
with the finite-valued representative.  The corresponding notation is used for
other majorants, for example $\widehat q^{\,\circ}$ when the original
function is denoted by $q$.

\section{Subtangential conditions}\label{s:subt}

Let $K:J=[0,a]\to2^X$ have closed values, put
$K_A(t):=K(t)\cap\overline{D(A)}$ and write
$K_A(I):=\bigcup_{t\in I}K_A(t)$ for $I\subset J$.  The forcing is defined on $\Gr(K_A):=\{(t,x)\in J\times X:x\in K_A(t)\}$.  Given $t_0\in[0,a)$ and $x_0\in K_A(t_0)$, a
\emph{mild solution} of the constrained problem on $[t_0,a]$ is a
continuous map $u:[t_0,a]\to\overline{D(A)}$ such that
$u(t)\in K_A(t)$ on $[t_0,a]$, the composition
$f(\cdot,u(\cdot))$ belongs to $L^1([t_0,a];X)$, and $u$ is the mild
solution of $u'(t)+Au(t)\ni w(t)$, $u(t_0)=x_0$, with
$w=f(\cdot,u(\cdot))$.  
For the original constrained problem, we use the standing
\emph{linear growth condition}
\begin{equation}\label{fgrowth4b}
 \|f(t,x)\|\leq c(t)(1+\|x\|)
 \qquad\mbox{for all }(t,x)\in\Gr(K_A)
\end{equation}
with a fixed $c\in L^1_+(J)$.  In Section~\ref{s:bounded-subtube}
this hypothesis is reduced, for bounded sets of prescribed initial values,
to a common integrable bound on a bounded effective sub-tube.

\begin{definition}\label{viability-definition}
The tube $K$ is \emph{viable} (or \emph{weakly flow invariant}) for the
constrained problem if, for every $t_0\in[0,a)$ and every
$x_0\in K_A(t_0)$, at least one mild solution exists on
$[t_0,a]$.  If $f$ is defined on a larger state space, the tube is \emph{strongly
flow invariant} if every mild solution starting in $K_A(t_0)$ remains in
the tube.  Weak and strong flow invariance coincide whenever the initial
value problem is uniquely solvable.
\end{definition}

Since mild solutions take values in $\overline{D(A)}$, the effective
constraint at time $t$ is $K_A(t)=K(t)\cap\overline{D(A)}$.  For viability the
effective sections encountered by a solution must be nonempty.  In the existence theory developed below,
however, nonemptiness at later times is derived from a
single nonempty initial section and the subtangential conditions rather than
imposed as a separate standing hypothesis.  We say that
$\Gr(K_A)$ is \emph{closed from the left}, or that $t\mapsto K_A(t)$ is
sequentially upper semicontinuous from the left, if
\begin{equation}\label{graph-closed-from-left}
 t_n\nearrow t,\quad x_n\in K_A(t_n),\quad x_n\to x
 \quad \text{ implies } \quad x\in K_A(t).
\end{equation}
This property is necessary for viability.  Indeed, for such sequences choose
viable solutions $u_n$ with $u_n(t_n)=x_n$.  The standard semigroup estimate,
the linear-growth bound, and Gronwall's inequality yield $u_n(t)\to x$.
Since $u_n(t)\in K_A(t)$ and $K_A(t)$ is closed, $x\in K_A(t)$.

For a closed set $G\subset X$ and $x\in G$, a vector $z\in X$ is called
\emph{subtangential} if $\liminf_{h\downarrow0}h^{-1}d(x+hz,G)=0$.
Thus a subtangential vector at a boundary point may be tangential or point into the interior.
The subtangential set is the Bouligand contingent cone, which is
also commonly called a tangent cone in the literature. We prefer to speak of subtangential vectors
for such one-sided conditions.

The stronger condition obtained by replacing the limit inferior by an
ordinary (one-sided) limit will be called \emph{adjacent subtangentiality}.  Hence
subtangentiality requires the first-order approximation only along one
sequence $h_n\downarrow0$, while adjacent subtangentiality requires it for
all sufficiently small positive increments.  
The two sets need not agree as purely geometric
objects.  The distinction often disappears at the level of an existence
theorem: subtangentiality may yield a viable solution, while viability
together with sufficient regularity of the right-hand side then implies the
corresponding adjacent condition.  Under the usual hypotheses of the
classical Nagumo theorem for a continuous vector field, for example,
the subtangential condition is sufficient for viability and viability
implies the adjacent condition.  This is an equivalence of viability
criteria, not an equality of the two geometric sets.

For $z\in X$ define the translated operator
$A_zx:=Ax-z:=\{y-z:y\in Ax\}$, with $D(A_z)=D(A)$.  Since $A$ is $m$-accretive, so is $A_z$.  We denote by $S_z(h)$ the
contraction semigroup generated by $-A_z$.  Thus $h\mapsto S_z(h)x$ is the mild solution of
$v'(h)+Av(h)\ni z$, $v(0)=x$, and $S_0=S$.

\begin{definition}\label{A-subtangent-definition}
For $t\in[0,a)$ and $x\in K_A(t)$, the \emph{$A$-subtangential set} to the
tube at $(t,x)$ is
\[
 \TAK(t,x):=
 \left\{z\in X:
 \liminf_{h\downarrow0}\frac{d(S_z(h)x,K_A(t+h))}{h}=0
 \right\}.
\]
Its \emph{adjacent $A$-subtangential set} is
\[
 \TAKadj(t,x):=
 \left\{z\in X:
 \lim_{h\downarrow0}\frac{d(S_z(h)x,K_A(t+h))}{h}=0
 \right\}.
\]
\end{definition}

Both notions are one-sided because the evolution is considered
only towards later times, and
$\TAKadj(t,x)\subset\TAK(t,x)$.

The following necessity result applies to jointly measurable Carath\'eodory
forcing.  Its exceptional null set is independent of the initial state.

\begin{proposition}[Necessary adjacent subtangentiality outside a null set]
\label{necessary-regular-tangency}
Under the standing growth hypothesis \eqref{fgrowth4b}, let $X$ be separable,
assume that $f:\Gr(K_A)\to X$ is a jointly measurable Carath\'eodory map,
and suppose that $K$ is viable.  Then there is a
null set $N\subset[0,a)$ such that
\begin{equation}\label{SC1}
 \lim_{h\downarrow0} h^{-1} d(S_{f(t,x)}(h)x,K_A(t+h))=0
 \qquad\mbox{for all }t\in[0,a)\setminus N\mbox{ and }x\in K_A(t).
\end{equation}
Equivalently,
\begin{equation}\label{sub-tak}
 f(t,x)\in\TAKadj(t,x)\subset\TAK(t,x)
 \qquad\mbox{for all }t\in[0,a)\setminus N\mbox{ and }x\in K_A(t).
\end{equation}
\end{proposition}

\begin{proof}
By Lemma~\ref{l1}, after replacing successive Scorza--Dragoni sets by
their finite unions, choose increasing Scorza--Dragoni sets
$J_n\subset J$ such that $\lambda_1(J\setminus J_n)\leq2^{-n}$. 
For every $n$, let $I_n$ be the set of all $t\in J_n\cap[0,a)$ for
which
\begin{equation}\label{regular-density-points}
 \lim_{h\downarrow0}\frac{\lambda_1([t,t+h]\setminus J_n)}h=0,
 \qquad
 \lim_{h\downarrow0}\frac1h
       \int_{[t,t+h]\setminus J_n}c(s)\,ds=0.
\end{equation}
The Lebesgue differentiation theorem from the right gives
$\lambda_1(J_n\setminus I_n)=0$; hence
$N:=[0,a)\setminus\bigcup_{n\geq1}I_n$ is null.

Fix $t_0\in I_n$ and $x_0\in K_A(t_0)$.  By viability there is a mild
solution $u$ on $[t_0,a]$ such that $u(t_0)=x_0$ and
$u(t)\in K_A(t)$.  Put $z_0=f(t_0,x_0)$ and
$v(t_0+h):=S_{z_0}(h)x_0$.  Proposition~\ref{mild-difference} yields, for $0<h\leq a-t_0$,
\begin{equation}\label{frozen-comparison}
 \frac{d(S_{z_0}(h)x_0,K_A(t_0+h))}{h}
 \leq\frac1h\int_{t_0}^{t_0+h}
       \|f(s,u(s))-z_0\|\,ds.
\end{equation}
Let $M=\max_{t_0\leq s\leq a}\|u(s)\|$ and
$B_h=[t_0,t_0+h]\setminus J_n$.  On $J_n$, the Scorza--Dragoni property and continuity of $u$ imply
$\sup_{s\in[t_0,t_0+h]\cap J_n}\|f(s,u(s))-z_0\|\to0$.
On the complementary set, the growth condition gives
\begin{equation*}
 \frac1h\int_{B_h}\|f(s,u(s))-z_0\|\,ds
\leq \frac{1+M}{h}\int_{B_h}c(s)\,ds
      +\|z_0\|\frac{\lambda_1(B_h)}h \rightarrow 0
\end{equation*}
by \eqref{regular-density-points}.  The right-hand side of
\eqref{frozen-comparison} therefore tends to zero.  This proves
\eqref{SC1}, and \eqref{sub-tak} follows from
Definition~\ref{A-subtangent-definition}.
\end{proof}

\begin{remark}
Separability in Proposition~\ref{necessary-regular-tangency} is needed to 
obtain one exceptional null set independent of the state.  In an arbitrary
Banach space, every viable solution $v$ still satisfies the condition trajectory-wise. 
Indeed, put $w:=f(\cdot,v(\cdot))$ and use the pointwise
representative $w(t)=f(t,v(t))$.  At each right Lebesgue point $t$ of $w$,
$d\bigl(S_{w(t)}(h)v(t),K_A(t+h)\bigr)\leq\int_t^{t+h}\|w(s)-w(t)\|\,ds=o(h)$.
Thus $f(t,v(t))\in\TAKadj(t,v(t))$ almost everywhere, but the
exceptional set may depend on $v$ and therefore does not give the
state-independent conclusion of the proposition.
\end{remark}

Given such a null set $N$, we call the times in $[0,a)\setminus N$
\emph{regular times} and the times in $N$ \emph{exceptional times}.  These terms refer only to the two
different hypotheses imposed at those times; they do not assert any additional
regularity of the set-valued map $K$.

Regular-time adjacent subtangentiality alone is not sufficient for viability; the following
example already appears in \cite[Example~1]{Bo92}.  Let
$\phi:[0,1]\to[0,1]$ be the Cantor--Lebesgue function, put
$X=\mathbb R$, $A=0$, $f=0$, and $K(t)=\{\phi(t)\}$.  Since
$\phi'(t)=0$ for almost every $t$, condition \eqref{SC1} holds almost
everywhere.  Nevertheless, the only mild solutions of $u'=0$ are
constant, so the nonconstant singleton tube is not viable.  Some control
of the motion of the tube on the exceptional set is therefore essential.
For a null set $N\subset[0,a)$ and $\psi\in L^1_+(J)$, we impose the exceptional-set condition
\begin{equation}\label{SC2}
 \liminf_{h\downarrow0}\frac1h
 \left(
 d(S(h)x,K_A(t+h))-\int_t^{t+h}\!\psi(s)\,ds
 \right)^+ \!=0
 \quad\mbox{for all }t\in N,\, x\in K_A(t).
\end{equation}
Unlike \eqref{SC1}, the existence of one common $\psi$ is not an
automatic consequence of viability under the linear growth condition
when the tube is unbounded.  What viability always gives is the following
state-dependent estimate.  If $u$ is a viable mild solution starting from
$(t_0,x_0)$ and $M_{t_0,x_0}:=\max_{t_0\leq s\leq a}\|u(s)\|$, then Proposition~\ref{mild-difference} gives, for every admissible $h$,
\begin{equation}\label{state-dependent-exceptional-estimate}
 d(S(h)x_0,K_A(t_0+h))
 \leq\int_{t_0}^{t_0+h}c(s)(1+M_{t_0,x_0})\,ds.
\end{equation}
Thus \eqref{SC2} is necessary at this particular initial state with
$\psi=(1+M_{t_0,x_0})c$.  A common $\psi$ follows, for example, if the
relevant viable trajectories are uniformly bounded; in particular, this
is the case when $K_A(J)$ is bounded.

For the existence results below, \eqref{SC2} is imposed uniformly in the state
variable.  If $c$ is the growth bound in \eqref{fgrowth4b} and $\psi$ is
the exceptional-time bound, both requirements remain valid with the same function
$\widetilde c:=c+\psi$ in place of either function.

The two lower-limit conditions have the following equivalent sequential
forms.  Membership $f(t,x)\in\TAK(t,x)$ is equivalent to the
existence, for every admissible $(t,x)$, of $h_n\downarrow0$ and
$e_n\to0$ in $X$ such that
\begin{equation}\label{SC1-sequential}
 S_{f(t,x)}(h_n)x+h_ne_n\in K_A(t+h_n).
\end{equation}
Condition \eqref{SC2} is equivalent to the existence of
$h_n\downarrow0$, $\varepsilon_n\downarrow0$, and
$y_n\in K_A(t+h_n)$ such that
\begin{equation}\label{SC2-sequential}
 \|y_n-S(h_n)x\|
 \leq\int_t^{t+h_n}\psi(s)\,ds+h_n\varepsilon_n.
\end{equation}

\begin{remark}[Exceptional-time bounds]\label{pointwise-exceptional-remark}
\label{strong-exceptional-condition}
A stronger form of the exceptional-time hypothesis is
\begin{equation}\label{strong-exceptional-estimate}
 \liminf_{h\downarrow0}
 \frac{d(S(h)x,K_A(t+h))}{h}\leq q(t)
 \qquad\mbox{for all }t\in N\mbox{ and }x\in K_A(t),
\end{equation}
where $q\in L^1_+(J)$.  Proposition~\ref{c-hat} shows that
\eqref{strong-exceptional-estimate} implies \eqref{SC2} with
$\widehat q^{\,\circ}$, since only its integral occurs.
\end{remark}

For comparison with classical contingent cones, define the ordinary moving
subtangential set by
\[
 \TK(t,x):=
 \left\{z\in X:
 \liminf_{h\downarrow0}\frac{d(x+hz,K(t+h))}{h}=0
 \right\}.
\]
If $A=0$, then $S_z(h)x=x+hz$ and hence $\TAK(t,x)=\TK(t,x)$.  More
generally, resolvent invariance permits ordinary subtangentiality to be
converted into $A$-subtangentiality.

\begin{proposition}[Resolvent invariance and subtangentiality]
\label{separated-tangency-proposition}
Assume that $J_hK(s)\subset K(s)$ for every $h>0$ and $s\in[0,a)$.
Then
\[
 T_K(t,x)\subset T_K^A(t,x)
 \qquad\text{for every }t\in[0,a)\text{ and }x\in K_A(t).
\]
\end{proposition}

\begin{proof}
Fix $t\in[0,a)$, $x\in K_A(t)$, and $v\in T_K(t,x)$.  Choose
$h_n\downarrow0$ and $k_n\in K(t+h_n)$ with
$\|k_n-x-h_nv\|=o(h_n)$.  Resolvent invariance gives
$J_{h_n}k_n\in K(t+h_n)\cap D(A)$ and
$k_n\in(I+h_nA)J_{h_n}k_n$.  Hence
\[
 \liminf_{h\downarrow0}\frac1h
 d\bigl(x+hv,(I+hA)(K(t+h)\cap D(A))\bigr)=0.
\]
The range-condition implication proved in \cite{Bo-JEE05}, applied to the
constant continuous forcing with value $v$, now yields
$v\in T_K^A(t,x)$.
\end{proof}

For a static closed set $K(t)\equiv K$, the ordinary moving
subtangential set reduces to the Bouligand contingent cone $T_K(x)$.
If $K$ is closed and convex, then
$T_K(x)=\overline{\bigcup_{h>0}(K-x)/h}$ is a closed convex cone and the lower limit may be replaced by a limit.
Equivalently,
\begin{equation}\label{convex-tangent-normal}
 z\in T_K(x)
 \quad\Longleftrightarrow\quad
 x^*(z)\leq0\quad\hbox{for every }x^*\in N_K(x),
\end{equation}
where $N_K(x):=\{x^*\in X^*:x^*(y-x)\leq0\text{ for every }y\in K\}$
is the normal cone of convex analysis.

In particular, the order cone $X_+$ of a Banach lattice is closed and
convex.  For a map $F:D\to X$, quasimonotonicity is therefore precisely
\begin{equation}\label{quasimonotone-tangent-form}
 F(v)-F(u)\in T_{X_+}(v-u)
 \qquad\text{for all }u,v\in D\text{ with }u\leq v,
\end{equation}
or, equivalently,
$F(x+y)\in F(x)+T_{X_+}(y)$ whenever $y\in X_+$ and $x,x+y\in D$.
For the closed convex cone $X_+$, the subtangential and adjacent
subtangential cones coincide; this common cone is denoted by $T_{X_+}$.
For a time-dependent map $f$, these equivalent conditions are imposed on $f(t,\cdot)$ for almost every $t$.

\section{Reduction to a bounded effective tube}\label{s:bounded-subtube}

The basic idea of restricting a linear-growth problem to a bounded state
region goes back to \cite{Bo-JEE03}.  Here that device is adapted to the
moving effective tube and combined with the exceptional-time estimate.  For a
bounded set of prescribed initial values, the growth coefficient and the
exceptional-time majorant are first combined; the problem is then restricted
to a bounded effective tube on which one finite-valued integrable function
controls both the forcing and the exceptional motion.

The following lemma permits different coefficients in the growth and
exceptional estimates.  Replacing both coefficients by their sum preserves
the growth inequality, leaves regular subtangentiality unchanged, and only
weakens the exceptional estimate.

\begin{lemma}[Reduction to a bounded effective tube with a common integrable bound]
\label{bounded-submultifunction}
Let $A$ be $m$-accretive in a real Banach space $X$. Let
$K:J=[0,a]\to2^X$ have closed values with $K_A(0)\neq\emptyset$, and assume
that $\Gr(K_A)$ is closed from the left.  Let $f:\Gr(K_A)\to X$, let
$D\subset K_A(0)$ be nonempty and bounded, and suppose that there are a null
set $N\subset[0,a)$ and finite-valued functions
$c,\psi\in L^1_+(J)$ such that
\begin{align}
 &\|f(t,x)\|\leq c(t)(1+\|x\|)
 &&\mbox{for all }(t,x)\in\Gr(K_A),
 \label{linear-growth-approx}\\
 &f(t,x)\in\TAK(t,x)
 &&\begin{gathered}
 \text{for all }t\in[0,a)\setminus N\\[-2pt]
 \text{and }x\in K_A(t),
 \end{gathered}
 \label{linear-tangency-approx}\\
 &\liminf_{h\downarrow0}\frac1h
 \left(d(S(h)x,K_A(t+h))-\int_t^{t+h}\psi(s)\,ds\right)^+=0
 &&\begin{gathered}
 \text{for all }t\in N\\[-2pt]
 \text{and }x\in K_A(t).
 \end{gathered}
 \label{linear-exceptional-approx}
\end{align}
Then there are a closed-valued map $\widetilde K:J\to2^X$ with
$\widetilde K(t)\subset K(t)$ and a $\gamma\in L^1_+(J)$, finite-valued,
such that
\[
 D\subset\widetilde K_A(0),\qquad
 \Gr(\widetilde K)\ \text{is bounded},\qquad
 \|f(t,x)\|\leq\gamma(t) \text{ on }\Gr(\widetilde K_A),
\]
$\Gr(\widetilde K_A)$ is closed from the left,
\[
 f(t,x)\in T^A_{\widetilde K}(t,x)
 \quad\text{for all }t\in[0,a)\setminus N
 \text{ and }x\in\widetilde K_A(t),\vspace{-0.1in}
\]
and\vspace{-0.1in}
\[
 \liminf_{h\downarrow0}\frac1h
 \left(d(S(h)x,\widetilde K_A(t+h))
       -\int_t^{t+h}\gamma(s)\,ds\right)^+=0
 \quad\text{for all }t\in N\text{ and }x\in\widetilde K_A(t).
\]
Moreover, if $\Gr(K_A)$ is product measurable, so is $\Gr(\widetilde K_A)$.
\end{lemma}

\begin{proof}
Put $m:=c+\psi$.  Then \eqref{linear-growth-approx} remains true with $m$
in place of $c$, and \eqref{linear-exceptional-approx} remains true with
$m$ in place of $\psi$.  Choose $x_*\in D(A)$ and $R_0>1$ such that
$D\subset\overline B_{R_0}(x_*)$.  Let $r$ be the absolutely continuous
solution of
\begin{equation}\label{moving-radius-equation}
 r'(t)=1+m(t)+\widehat m^{\,\circ}(t)
 \bigl(1+r(t)+\|S(t)x_*\|\bigr),
 \qquad r(0)=R_0,
\end{equation}
and set
$\widetilde K(t):=K(t)\cap\overline B_{r(t)}(S(t)x_*)$.
Then $D\subset\widetilde K_A(0)$ and
$\Gr(\widetilde K_A)$ is closed from the left.  Since $r$ and the
semigroup orbit are bounded on $J$, $\Gr(\widetilde K)$ is bounded and
there is $M\geq1$ such that
\[
 \|f(t,x)\|
 \leq m(t)(1+r(t)+\|S(t)x_*\|)
 \leq Mm(t)
 \quad\text{for all }t\in J\text{ and }x\in\widetilde K_A(t).
\]
Set $\gamma:=Mm$.

Fix $t\notin N$, $x\in\widetilde K_A(t)$, and put $z=f(t,x)$.  There are
$h_n\downarrow0$ and $e_n\to0$ such that
$y_n:=S_z(h_n)x+h_ne_n\in K_A(t+h_n)$.  Proposition~\ref{mild-difference}
gives
\begin{align*}
 \|y_n-S(t+h_n)x_*\|
 &\leq r(t)+h_n\|z\|+h_n\|e_n\|\\
 &\leq r(t)+h_nm(t)
  (1+r(t)+\|S(t)x_*\|)+h_n\|e_n\|.
\end{align*}
By the local-majorant property of $\widehat m$, the equality of the
integrals of $\widehat m$ and $\widehat m^{\,\circ}$, and continuity of
$r$ and the semigroup orbit, \eqref{moving-radius-equation} implies
\[
 \liminf_{h\downarrow0}\frac{r(t+h)-r(t)}h\geq
 1+m(t)(1+r(t)+\|S(t)x_*\|).
\]
Since $e_n\to0$, it follows for all sufficiently large $n$ that
$\|y_n-S(t+h_n)x_*\|\leq r(t+h_n)$.  Hence
$y_n\in\widetilde K_A(t+h_n)$ and
$z\in T^A_{\widetilde K}(t,x)$.

If $t\in N$ and $x\in\widetilde K_A(t)$, choose
$h_n\downarrow0$, $\eta_n\downarrow0$, and $y_n\in K_A(t+h_n)$ with
$\|y_n-S(h_n)x\|\leq\int_t^{t+h_n}m(s)\,ds+h_n\eta_n$.
Then
$\|y_n-S(t+h_n)x_*\|\leq
r(t)+\int_t^{t+h_n}m(s)\,ds+h_n\eta_n$.
Equation~\eqref{moving-radius-equation} gives
$r(t+h)-r(t)\geq h+\int_t^{t+h}m(s)\,ds$.
For all sufficiently large $n$, $\eta_n\leq1$, and therefore
$y_n\in\widetilde K_A(t+h_n)$.  Thus the exceptional condition holds for
$\widetilde K$ with $m$, and hence also with $\gamma=Mm$.

Finally,\vspace{-0.1in}
\[
 \Gr(\widetilde K_A)
 =\Gr(K_A)\cap
 \{(t,x)\in J\times X:\|x-S(t)x_*\|\leq r(t)\}.
\]
The second set is Borel, since $r$ and $t\mapsto S(t)x_*$ are continuous.
Hence product measurability is inherited.
\end{proof}

\begin{corollary}[Pointwise exceptional bounds]\label{bounded-pointwise-corollary}
In Lemma~\ref{bounded-submultifunction}, the integral exceptional condition
\eqref{linear-exceptional-approx} may be replaced by
\[
 \liminf_{h\downarrow0}h^{-1}d(S(h)x,K_A(t+h))\leq q(t)
 \quad\text{for all }t\in N\text{ and }x\in K_A(t),
\]
where $q\in L^1_+(J)$ is finite-valued.  The same conclusion holds.
\end{corollary}

\begin{proof}
By Remark~\ref{pointwise-exceptional-remark} and
Proposition~\ref{c-hat}, the pointwise estimate implies
\eqref{linear-exceptional-approx} with the finite-valued representative
$\widehat q^{\,\circ}$ in place of $\psi$.  Apply
Lemma~\ref{bounded-submultifunction}.
\end{proof}

\begin{definition}[Bounded effective-tube setting]\label{bounded-effective-setting}
We say that $(A,K,f)$ satisfies the \emph{bounded effective-tube setting}
on $J=[0,a]$ if $\Gr(K_A)$ is bounded and closed from the left, and there
are a null set $N\subset[0,a)$ and a finite-valued
$c\in L^1_+(J)$ such that
\begin{align}
 &\|f(t,x)\|\leq c(t)
 &&\text{for all }t\in J\text{ and }x\in K_A(t),
 \label{bounded-setting-forcing}\\
 &f(t,x)\in T_K^A(t,x)
 &&\text{for all }t\in[0,a)\setminus N
   \text{ and }x\in K_A(t).
 \label{bounded-setting-regular}
\end{align}
Moreover,
\begin{equation}\label{bounded-setting-exceptional}
 \begin{gathered}
 \displaystyle
 \liminf_{h\downarrow0}\frac1h
 \left(d(S(h)x,K_A(t+h))-\int_t^{t+h}c(s)\,ds\right)^+=0
 \quad\text{for all }t\in N,\ x\in K_A(t).
 \end{gathered}
\end{equation}
The same symbol $c$ is deliberately used in both estimates.
\end{definition}

The bounded effective-tube setting is inherited under restriction to every
compact subinterval $J_0=[\alpha,\beta]\subset J$.  When such a restricted
problem is considered, we use the same terminology with the time interval
understood to be $J_0$, the common bound replaced by $c|_{J_0}$, and the
exceptional set replaced by $N\cap[\alpha,\beta)$.

Nonemptiness of every effective section is intentionally not part of
Definition~\ref{bounded-effective-setting}; it is a consequence of the conditions above
once a nonempty effective
section at the left endpoint of the interval under consideration is
prescribed.

By Lemma~\ref{bounded-submultifunction}, every problem satisfying the
linear-growth and regular/exceptional subtangential hypotheses of
Section~\ref{s:subt} reduces, for each bounded prescribed set of initial
values, to this setting.  After applying
that lemma, we relabel the bounded effective tube $\widetilde K_A$ as $K_A$
(and its underlying constraint $\widetilde K$ as $K$), and relabel its common
bound as $c$.  This convention avoids carrying additional tildes in the
remaining arguments.

\section{Geometric continuation and right lower semicontinuity}\label{s:continuation}

Under the bounded effective-tube setting of
Definition~\ref{bounded-effective-setting}, the same coefficient controls
both regular and exceptional steps.  The following geometric lemma produces a set $P$ of constrained nodes
at each of which the auxiliary path can be compared with the mild
evolution restarted there.

\begin{lemma}[Geometric continuation under the subtangential conditions]\label{approx1}
Let $A$ be $m$-accretive in a real Banach space $X$.
Let $K:J=[0,a]\to 2^X$ have closed values with
$K_A(0)\neq\emptyset$.  Assume that $\Gr(K_A)$ is closed from the left.  Suppose that there are a null set $N\subset[0,a)$ and a finite-valued
$c\in L^1_+(J)$ such that
\begin{equation}\label{geometric-regular-tangency}
 c(t)\overline B_1(0)\cap\TAK(t,x)\neq\emptyset
 \qquad\mbox{for all }t\in[0,a)\setminus N\mbox{ and }x\in K_A(t)\vspace{-0.1in}
\end{equation}
and\vspace{-0.1in}
\begin{equation}\label{geometric-exceptional-tangency}
 \liminf_{h\downarrow0}\frac1h
 \left(
 d(S(h)x,K_A(t+h))-\int_t^{t+h}c(s)\,ds
 \right)^+=0
 \quad\mbox{for all }t\in N,\, x\in K_A(t).
\end{equation}
Put $\gamma:=1+c$.  Given $0\leq t_0<t_1\leq a$ and
$x_0\in K_A(t_0)$, there exist
a continuous map $v:[t_0,t_1]\to X$,
a strongly measurable $w\in L^1([t_0,t_1];X)$, and
a set $P\subset[t_0,t_1]$ containing $t_0$ and $t_1$, such that $v(t_0)=x_0$ and
\begin{align}
 &v(\tau)\in K_A(\tau) &&\mbox{for all }\tau\in P,
 \label{geometric-nodes}\\[0.5ex]
 &\|w(t)\|\leq\widehat c^{\,\circ}(t) &&\text{for a.e. }t\in[t_0,t_1],
 \label{geometric-w-bound}\\[-0.5ex]
 &\|v(t)-u(t;\tau,v(\tau),w)\|
 \leq\int_\tau^t \gamma(s)\,ds
 &&\mbox{for all }\tau\in P\mbox{ and }\tau\leq t\leq t_1.
 \label{geometric-restart-estimate}
\end{align}
In particular, $K_A(t)\neq\emptyset$ for every $t\in J$.
\end{lemma}

\begin{proof}
It suffices to construct the objects on $[t_0,t_1]$; all integrals and mild
solutions in the proof are restricted to this interval.

Fix $(b,x)$ with
$t_0\leq b<t_1$ and $x\in K_A(b)$.  If $b\notin N$, choose
$z\in\TAK(b,x)$ with $\|z\|\leq c(b)$.  By
Proposition~\ref{c-hat}, after reducing the step size we have
$c(b)\leq\widehat c(s)$ near $b$.  The sequential form of subtangentiality gives
$h>0$, $b+h\leq t_1$, and $x_1\in K_A(b+h)$ such that $\xi:=x_1-S_z(h)x$ satisfies $\|\xi\|\leq h$.
If $b\in N$, condition \eqref{geometric-exceptional-tangency} gives,
after reducing $h$ if necessary, an $x_1\in K_A(b+h)$ such that
$\xi:=x_1-S(h)x$ satisfies
$\|\xi\|\leq h+\int_b^{b+h}c(s)\,ds$.  In the first case put $q(s)=S_z(s-b)x$ and $w_b(s)=z$; in the second
put $q(s)=S(s-b)x$ and $w_b(s)=0$.  In both cases
$R:=\int_b^{b+h}\gamma(s)\,ds\geq\|\xi\|$.  Define
\[
 p_b(s):=
 \begin{cases}
  \gamma(s)\xi/R,&R>0,\\
  0,&R=0,
 \end{cases}
 \qquad
 v_b(s):=q(s)+\int_b^s p_b(\sigma)\,d\sigma.
\]
Then $\|p_b(s)\|\leq\gamma(s)$ a.e.,
$v_b(b)=x$, and $v_b(b+h)=x_1\in K_A(b+h)$.  Moreover,
\begin{equation}\label{local-geometric-error}
 \|v_b(s)-u(s;b,x,w_b)\|
 \leq\int_b^s \gamma(\sigma)\,d\sigma
 \qquad\mbox{for all }s\in[b,b+h].
\end{equation}

We apply Zorn's lemma.  Let $\mathcal M$ consist of all quadruples
$(v,w,P,b)$ such that $b\in[t_0,t_1]$,
$v\in C([t_0,b];X)$, $w\in L^1([t_0,b];X)$,
$P\subset[t_0,b]$, and $t_0,b\in P$, and \eqref{geometric-nodes}--\eqref{geometric-restart-estimate} hold on
$[t_0,b]$, with $b$ in place of $t_1$ and with $v(t_0)=x_0$.  The
degenerate object with $b=t_0$ belongs to $\mathcal M$.  Order $\mathcal M$ by
extension: $(v,w,P,b)\preceq(\bar v,\bar w,\bar P,\bar b)$ if
$b\leq\bar b$, the functions agree on the common interval (a.e.\ for $w$),
and $P\subset\bar P$.

Every chain has an upper bound.  Only the case in which the supremum
$b^*$ of its terminal times is not attained needs explanation.  Choose an
increasing sequence of chain elements $(v_n,w_n,P_n,b_n)$ such that
$b_n\uparrow b^*$.  Every element of the chain is then extended by one of
these selected elements.  Their consistency defines $v^*,w^*$ on
$[t_0,b^*)$ and $P^*=\bigcup_nP_n$.  After modifying the representatives on a common null set, the $w_n$ agree on their overlaps; hence their piecewise definition on the successive intervals is strongly measurable.
The bound \eqref{geometric-w-bound} gives $w^*\in L^1([t_0,b^*];X)$.  
For $\tau=b_n$ and $t\in[b_n,b^*)$,
\eqref{geometric-restart-estimate} yields the pointwise estimate
\begin{equation}\label{geometric-tail-estimate}
 \|v^*(t)-u(t;b_n,v^*(b_n),w^*)\|
 \leq\int_{b_n}^{t}\gamma(s)\,ds
 \leq\int_{b_n}^{b^*}\gamma(s)\,ds.
\end{equation}
We show that $v^*(t)$ has a limit as $t\uparrow b^*$.  Let $\eta>0$
and choose $n$ so large that
$2\int_{b_n}^{b^*}\gamma(s)\,ds<\eta/2$.  The mild solution
$t\mapsto u(t;b_n,v^*(b_n),w^*)$ is uniformly continuous on
$[b_n,b^*]$.  Hence, for $s,t\in[b_n,b^*)$ sufficiently close to $b^*$,
\[
 \|v^*(s)-v^*(t)\|
\leq 2\int_{b_n}^{b^*}\gamma(r)\,dr +\|u(s;b_n,v^*(b_n),w^*)
          -u(t;b_n,v^*(b_n),w^*)\|<\eta.
\]
Thus $v^*(t)$ has a limit as $t\uparrow b^*$; define $v^*(b^*)$ to be this limit.
Since $b_n\in P_n$ and $v^*(b_n)\in K_A(b_n)$, left-closedness of
$\Gr(K_A)$ yields $v^*(b^*)\in K_A(b^*)$.  Add $b^*$ to $P^*$ and pass to
the limit in \eqref{geometric-restart-estimate}.  This produces an
upper bound for the chain.

Let $(v,w,P,b)$ be maximal.  If $b<t_1$, apply the local construction
above at $(b,v(b))$, extend by $w\oplus w_b$ and $v\oplus v_b$, and add
the new terminal point $b+h$ to $P$; here $\oplus$ denotes concatenation.  For every
$\tau\in P$ and $s\in[b,b+h]$, the semigroup property and contractivity give
\begin{align*}
 &\|v_b(s)-u(s;\tau,v(\tau),w\oplus w_b)\| \leq
 \|v_b(s)-u(s;b,v(b),w_b)\|
 +\|v(b)-u(b;\tau,v(\tau),w)\|\\
 &\quad\leq\int_b^s \gamma(\sigma)\,d\sigma
       +\int_\tau^b \gamma(\sigma)\,d\sigma.
\end{align*}
The resulting object is a strict extension, a contradiction.
Consequently $b=t_1$.  

Taking $t_0=0$ and arbitrary $t_1$ proves the last assertion.
\end{proof}

\begin{definition}[Lower semicontinuity from the right]
\label{def:right-lsc}
Let $K:J=[0,a]\to2^X\setminus\{\emptyset\}$ have closed values.  We say
that $K$ is \emph{lower semicontinuous from the right} at $t<a$ if,
for every $x\in K(t)$ and every open neighbourhood $U$ of $x$, there
is $\delta>0$ such that $K(s)\cap U\neq\emptyset$ for all
$s\in J$ with $t\leq s<t+\delta$.  It is lower semicontinuous from the
right on $J$ if this holds at every $t<a$.  In the metric space $X$, this
is equivalent to
$\lim_{s\downarrow t}d(x,K(s))=0$ for all $t<a$ and $x\in K(t)$.
This is distinct from the sequential upper semicontinuity from the left
expressed by \eqref{graph-closed-from-left}.
\end{definition}

\begin{corollary}
\label{right-persistence-corollary}
Under the assumptions of Lemma~\ref{approx1}, the set-valued map
$t\mapsto K_A(t)$ is lower semicontinuous from the right.  
\end{corollary}

\begin{proof}
Fix $t<a$ and $x\in K_A(t)$ and apply Lemma~\ref{approx1} on $[t,t+h]$
with initial value $x$.  Its terminal node belongs to $K_A(t+h)$, and
Proposition~\ref{mild-difference}, \eqref{geometric-w-bound}, and
\eqref{geometric-restart-estimate} give
\[
 d(x,K_A(t+h))\leq\|x-S(h)x\|
 +\int_t^{t+h}\bigl(1+c(s)+\widehat c^{\,\circ}(s)\bigr)\,ds.
\]
The right-hand side tends to zero with $h\downarrow0$.
\end{proof}

\section{Reduction to a separable subspace}
\label{s:separable-reduction}

Any viable mild solution is contained in a separable resolvent-invariant
subspace.  Its compact range and a separable essential range of its forcing have
separable closed linear span.  Repeatedly adjoining the images of the resulting
subspaces under the resolvents and taking the closure yields a closed separable
subspace in which the corresponding part of $A$ is $m$-accretive.  This section
constructs such a subspace directly from the data, before a solution is known.

The bounded reduction has already been carried out in
Section~\ref{s:bounded-subtube}, and Section~\ref{s:continuation} supplies
right lower semicontinuity of the bounded effective tube.  The present
section therefore has two tasks: construct
separably valued almost minimizing selections at every time, and then add
their ranges, the corresponding forcing values, and the resolvent images
iteratively.  This yields a closed separable subspace $Y$ such that, with
$K_Y(t):=K_A(t)\cap Y$,
\[
 J_\lambda Y\subset Y\quad\text{for all }\lambda>0,\quad
 d_Y(y,K_Y(t))=d_X(y,K_A(t))\quad\text{for all }t\in J,\, y\in Y,
\]
and the forcing is $Y$-valued on the reduced graph outside one null set.
Whenever the original space $X$ and the reduced space $Y$ occur simultaneously,
we call $X$ the \emph{ambient space}. Correspondingly, \emph{ambient} objects
are those of the original problem in $X$, whereas \emph{reduced} objects are
their restrictions or counterparts in $Y$.  The above
all-time distance identity then permits the regular and exceptional
subtangential conditions to be transferred to the part
$A_Y=A\cap(Y\times Y)$, with the exceptional null set enlarged to account
for the modification of the forcing; operators are here identified with their
graphs.

\subsection{One-sided almost minimizing selections}

Equip $J$ with the lower-limit topology $\tau_+$ generated by the
sets $[t,t+h)\cap J$.  Thus a multifunction is lower semicontinuous from the
right in the sense of Definition~\ref{def:right-lsc} precisely when it
is lower semicontinuous on $(J,\tau_+)$.  The space $(J,\tau_+)$ is a subspace
of the Sorgenfrey line.  It is zero-dimensional, regular, and hereditarily
Lindel\"of; consequently all its subspaces are paracompact.  Moreover, every
$\tau_+$-open subset of $J$ is Borel for the usual topology; see
\cite{Engelking}.  The use of this stronger one-sided topology is related to
the selection principle of Bressan and Colombo~\cite{BressanColombo92}, which
seeks selections continuous for an additional topology stronger than the
original one.  Here the lower-limit topology is itself the topology of lower
semicontinuity, so we apply Michael's zero-dimensional selection theorem
directly: a lower semicontinuous map from a zero-dimensional paracompact space
into the family of nonempty closed subsets of a completely metrizable space
admits a continuous selection~\cite[Theorem~1.2]{Michael56}.  Convexity of the values is not required.

For maps from $J$ into a metric space, continuity with respect to
$\tau_+$ is equivalent to continuity from the right in the usual order sense:
$s\to t+$ implies $p(s)\to p(t)$ for every $t<a$.  Hence every right lower
semicontinuous multifunction with nonempty closed values admits such a
right-continuous selection $p:J\to X$.  The selection is Borel measurable for
the usual topology.  Its range is separable because $(J,\tau_+)$ is
Lindel\"of and every Lindel\"of metric space is separable; consequently $p$ is
strongly measurable.  The argument below applies Michael's theorem to
truncated multifunctions in order to obtain selections that are additionally
almost minimizing with respect to a prescribed point.

\begin{lemma}[Upper semicontinuity from the right and measurability of distances]
\label{right-distance-measurability}
Let $K:J=[0,a]\to\PX$ be lower semicontinuous from the right and let $q\in X$.
Then $d_q:J\to \R$ with $d_q(t):=d(q,K(t))$ is upper semicontinuous from the right and Borel
measurable.
\end{lemma}

\begin{proof}
Fix $t<a$ and $\varepsilon>0$.  Choose $x\in K(t)$ with
$\|q-x\|<d_q(t)+\varepsilon$.  Lower semicontinuity from the right gives $\delta>0$
such that $d(x,K(s))<\varepsilon$ for $t\leq s<t+\delta$.  Hence $d_q(s)\leq\|q-x\|+d(x,K(s))<d_q(t)+2\varepsilon$, which proves the first assertion.  The strict sublevel sets of $d_q$ are
$\tau_+$-open and therefore Borel in the usual topology.
\end{proof}

\begin{theorem}[Almost minimizing selections]
\label{thm:almost-minimizing-selection}
Let $X$ be an arbitrary Banach space and let $K:J=[0,a]\to\PX$ be lower semicontinuous from the right
with closed values.  For every $q\in X$ and $\eta>0$
there is a strongly measurable map $p_{q,\eta}:J\to X$ whose range is
separable and such that
\begin{equation}\label{almost-minimizing-selection}
 p_{q,\eta}(t)\in K(t),\qquad
 \|p_{q,\eta}(t)-q\|\leq d(q,K(t))+\eta\quad\text{for all }t\in J.
\end{equation}
\end{theorem}

\begin{proof}
Put $\delta=\eta/2$ and
$I_n:=\{t\in J:(n-1)\delta\leq d(q,K(t))<n\delta\}$, $n\geq1$.
By Lemma~\ref{right-distance-measurability}, the sets $I_n$ are Borel and form
a partition of $J$.  Equip each $I_n$ with the topology induced by $\tau_+$
on $J$, put $U_n:=B_{(n+1)\delta}(q)$, and set
$\Phi_n(t):=K(t)\cap U_n$ for $t\in I_n$.  
With its relative norm topology, $U_n$ is completely metrizable.  If
$t\in I_n$, then
$d(q,K(t))<n\delta$, so $K(t)\cap B_{n\delta}(q)\neq\emptyset$ and hence
$\Phi_n(t)\neq\emptyset$.  The values of $\Phi_n$ are closed
relative to $U_n$.  Moreover, $\Phi_n$ is lower semicontinuous with
respect to the topology induced by $\tau_+$: if $O\subset U_n$ is open,
then $\{t\in I_n:\Phi_n(t)\cap O\neq\emptyset\}
=I_n\cap\{t\in J:K(t)\cap O\neq\emptyset\}$ is open in $I_n$ for
this topology.  Michael's theorem therefore gives
a selection $p_n:I_n\to U_n$ that is continuous with respect to the
topology induced by $\tau_+$.

Since $I_n$, equipped with this topology, is Lindel\"of and $p_n$ is
continuous for this topology, its range $p_n(I_n)$ is separable.
Moreover, every set open for the topology induced by $\tau_+$ is Borel
for the relative usual topology on $I_n$.  Hence $p_n$ is Borel
measurable for the relative usual topology.

Define $p_{q,\eta}:=p_n$ on $I_n$.  Since $(I_n)$ is a countable Borel
partition of $J$, the map $p_{q,\eta}$ is Borel measurable and has
separable range, hence it is strongly measurable.  If $t\in I_n$, then
$\|p_{q,\eta}(t)-q\|<(n+1)\delta\leq d(q,K(t))+2\delta=d(q,K(t))+\eta$, which proves \eqref{almost-minimizing-selection}.
\end{proof}

The map $p_{q,\eta}$ may be viewed as a measurable, separably valued
approximate metric projection of $q$ onto the moving sets $K(t)$.  We use the
term almost minimizing selection, rather than ``almost projection'', because
the values $K(t)$ need not be proximinal and no genuine metric projection is
asserted.

\begin{corollary}[Separable enlargement preserving section distances]
\label{one-step-section-distances}
Let $E\subset X$ be nonempty and separable and let $K$ satisfy the assumptions
of Theorem~\ref{thm:almost-minimizing-selection}.  There is a closed separable
subspace $Z\subset X$ containing $E$ such that
\begin{equation}\label{one-step-distance-identity}
 d_Z(q,K(t)\cap Z)=d_X(q,K(t))\quad\text{for all }t\in J\text{ and }q\in E.
\end{equation}
In particular, $K(t)\cap Z\neq\emptyset$ for every $t\in J$.
\end{corollary}

\begin{proof}
Choose a dense sequence $(q_j)$ in $E$.  For every $j,k\geq1$, apply
Theorem~\ref{thm:almost-minimizing-selection} with $\eta=2^{-k}$ and let $p_{j,k}$
be the resulting selector.  Put
$Z:=\overline{\operatorname{span}}\bigl(E\cup
\bigcup_{j,k\geq1}p_{j,k}(J)\bigr)$.  This space is separable.  For every $t\in J$,
$d_Z(q_j,K(t)\cap Z)\leq\|q_j-p_{j,k}(t)\|\leq d_X(q_j,K(t))+2^{-k}$.
Letting $k\to\infty$ and using $K(t)\cap Z\subset K(t)$ gives equality for
all $q_j$.  Both distance functions are $1$-Lipschitz in the first variable,
so density gives \eqref{one-step-distance-identity} for every $q\in E$.  The
right-hand side is finite, which yields nonemptiness.
\end{proof}

\begin{remark}[Static, singleton, and diagonal examples]
If $K(t)\equiv C$ is a static nonempty closed set, one may choose a single
$x_\eta\in C$ with $\|x_\eta-q\|\leq d(q,C)+\eta$ and use the constant
selection, even when $C$ is nonseparable.  If $K(t)=\{z(t)\}$,
lower semicontinuity from the right is equivalent to continuity of $z$
from the right; hence $z(J)$ is separable.  By contrast, the singleton-valued multifunction $K(t)=\{e_t\}$ in
$\ell^2([0,1])$ is not lower semicontinuous from the right, since
$\|e_s-e_t\|=\sqrt2$ for $s\neq t$.  Thus this standard example of a multifunction admitting no separably valued selection
does not contradict Theorem~\ref{thm:almost-minimizing-selection}, since it is not lower semicontinuous from the right.
\end{remark}

\begin{remark}[Application to the moving tube]
Under the hypotheses of Lemma~\ref{approx1}, Corollary~\ref{right-persistence-corollary} shows that
$t\mapsto K_A(t)$ is lower semicontinuous from the right.
Lemma~\ref{right-distance-measurability} therefore implies that, for every
$q\in X$, the map $ t\mapsto d(q,K_A(t))$ is Borel measurable. 
Theorem~\ref{thm:almost-minimizing-selection}
then supplies separably valued almost minimizing selections, while
Corollary~\ref{one-step-section-distances} yields a closed separable enlargement in which the
distances to all sections $K_A(t)$ are preserved for the prescribed
points.
\end{remark}

\subsection{Construction of the separable subspace}

The entire ranges of the almost-minimizing selectors are separable, whereas the ranges of the corresponding forcing terms need not be. Since the forcing values need to be included only outside a null set, a standard measurable-range result reconciles these requirements.

\begin{lemma}[Essential separability on a Lebesgue interval]
\label{lem:essential-separable-range}
Let $(M,d)$ be metrizable and let $z:J\to M$ be measurable with respect to the
Lebesgue $\sigma$-algebra and the Borel $\sigma$-algebra of $M$.
Then there are a null set $N_z\subset J$ and a closed separable subset
$M_z\subset M$ such that $z(J\setminus N_z)\subset M_z$.  If $M$ is a Banach
space, then $z$ is strongly measurable.
\end{lemma}

\begin{proof}
Lebesgue measure on $J$ is a finite Radon measure, hence a compact measure in the sense of
\cite[451A(b)]{Fremlin4}.  Therefore
\cite[451R]{Fremlin4} yields a null set $N_z\subset J$ and a closed
separable subset $M_z\subset M$ such that
$z(J\setminus N_z)\subset M_z$.
If $M$ is a Banach space, measurability together with essential
separability implies strong measurability.
\end{proof}

\begin{lemma}[Composition with a measurable graph selection]
\label{measurable-graph-composition}
Let $G\in\mathcal L(J)\otimes\mathcal B(X)$, let $f:G\to X$ be measurable
with respect to the trace $\sigma$-algebra, and let $p:J\to X$ be strongly
measurable with $(t,p(t))\in G$ for every $t$.  Then the composition
$f(\cdot,p(\cdot))$ is strongly measurable and essentially separably valued.
\end{lemma}

\begin{proof}
The map $z_p(t)=(t,p(t))$ is measurable, so
$f(\cdot,p(\cdot))=f\circ z_p$ is an $X$-valued measurable map.  Apply
Lemma~\ref{lem:essential-separable-range}.
\end{proof}

\begin{theorem}[Preservation of sections and forcing values]
\label{preserve-sections-values}
Let $A$ be $m$-accretive in a real Banach space $X$. Let
$K:J=[0,a]\to\PX$ be lower semicontinuous from the
right with closed values, and let $f:\Gr(K)\to X$ be a jointly measurable
Carath\'eodory map.  Given a separable set $H\subset X$, there are a closed
separable subspace $Y\supset H$ and a full-measure set $I\subset J$ such
that $f(t,\cdot)$ is continuous on $K(t)$ for every $t\in I$ and:
\begin{enumerate}
\item[(a)] $J_\lambda Y\subset Y$ for every $\lambda>0$.  Consequently the
part $A_Y:=A\cap(Y\times Y)$ is $m$-accretive in $Y$, its resolvents and
semigroup are the restrictions of those of $A$, and
\begin{equation}\label{part-domain-identity}
 \overline{D(A_Y)}^{\,Y}=Y\cap\overline{D(A)}.
\end{equation}
\item[(b)] With $K_Y(t):=K(t)\cap Y$, the identity
\begin{equation}\label{section-distance-identity}
 d_Y(y,K_Y(t))=d_X(y,K(t))\quad\text{ holds for all }t\in J,\, y\in Y.
\end{equation}
In particular, $K_Y$ is lower semicontinuous from the right with nonempty values.
\item[(c)] The forcing satisfies
\begin{equation}\label{forcing-inclusion}
 f(t,K_Y(t))\subset Y\quad\text{for all }t\in I.
\end{equation}
\item[(d)] The reduced graph belongs to
$\mathcal L(J)\otimes\mathcal B(Y)$.  If $M:=J\setminus I$ and
\begin{equation}\label{reduced-forcing}
 f_Y(t,x):=
 \begin{cases}
 f(t,x),&t\in I,\\
 0,&t\in M,
 \end{cases}
 \qquad\text{for all }(t,x)\in\Gr(K_Y),
\end{equation}
then $f_Y:\Gr(K_Y)\to Y$ is a jointly measurable Carath\'eodory map. 
\end{enumerate}
\end{theorem}

\begin{proof}
Choose a null set $N_f\subset J$ such that $f(t,\cdot)$ is continuous
on $K(t)$ for every $t\in J\setminus N_f$, as guaranteed by the
Carath\'eodory property.  Put
$Y_0=\overline{\operatorname{span}}\, H$.  Suppose that $Y_n$ has been
constructed and choose a dense sequence $(q_{n,j})_{j\geq1}$ in $Y_n$.  For
$j,k\geq1$, let $p_{n,j,k}$ be the selector of
Theorem~\ref{thm:almost-minimizing-selection} for $q_{n,j}$ with error $2^{-k}$.
Its entire range is separable and satisfies
\begin{equation}\label{saturation-selector-estimate}
 p_{n,j,k}(t)\in K(t),\qquad
 \|p_{n,j,k}(t)-q_{n,j}\|\leq d(q_{n,j},K(t))+2^{-k}\quad\text{for all }t\in J.
\end{equation}
By Lemma~\ref{measurable-graph-composition}, the composition
$w_{n,j,k}:=f(\cdot,p_{n,j,k}(\cdot))$ is strongly measurable and essentially
separably valued.  Choose a closed separable subspace $W_{n,j,k}$ containing
its range outside a null set and define
\begin{equation}\label{saturation-recursion}
 Y_{n+1}:=\overline{\operatorname{span}}\Big(
 Y_n\cup\bigcup_{j,k\geq1}\bigl(p_{n,j,k}(J)\cup W_{n,j,k}\bigr)
 \cup\bigcup_{r\in\Q_+}J_r(Y_n)\Big).
\end{equation}
This is a closed separable subspace.  Let $I$ be the complement of $N_f$ and
of the countable union of the exceptional sets introduced above.

For every $t\in J$, \eqref{saturation-selector-estimate} gives
\begin{equation}\label{stage-distance-identity}
 d(q,K(t)\cap Y_{n+1})=d(q,K(t))\quad\text{for all }q\in Y_n.
\end{equation}
Indeed, first let $q=q_{n,j}$ and $k\to\infty$, and then use the
$1$-Lipschitz dependence on $q$.  To verify the forcing inclusion, fix
$t\in I$ and $x\in K(t)\cap Y_n$.  Choose $q_{n,j_\ell}\to x$ and
$k_\ell\to\infty$.  Since $x\in K(t)$,
\eqref{saturation-selector-estimate} gives
$p_{n,j_\ell,k_\ell}(t)\to x$.  The corresponding values of $f$ belong to
$W_{n,j_\ell,k_\ell}\subset Y_{n+1}$; continuity of $f(t,\cdot)$ therefore
gives
\begin{equation}\label{stage-forcing-inclusion}
 f(t,K(t)\cap Y_n)\subset Y_{n+1}.
\end{equation}

Put $Y=\overline{\bigcup_nY_n}$.  For rational $r>0$, continuity of $J_r$
and \eqref{saturation-recursion} give $J_rY\subset Y$.  For
$\lambda,r>0$ and $x\in X$, the resolvent identity yields
\[
 J_\lambda x=J_r\left(\frac r\lambda x+
 \left(1-\frac r\lambda\right)J_\lambda x\right),
\]
and hence $\|J_rx-J_\lambda x\|\leq\left|1-r/\lambda\right|\|x-J_\lambda x\|$.
Fix $y\in Y$ and choose positive rationals $r_m\to\lambda$.  Then
$J_{r_m}y\in Y$ and the last estimate gives $J_{r_m}y\to J_\lambda y$.
Closedness of $Y$ therefore yields $J_\lambda y\in Y$.  Thus
$J_\lambda Y\subset Y$ for every $\lambda>0$.  If $y\in Y$, then
\[
 y=J_\lambda y+\lambda\frac{y-J_\lambda y}{\lambda},\qquad
 J_\lambda y\in D(A_Y),\quad
 \frac{y-J_\lambda y}{\lambda}\in A_YJ_\lambda y.
\]
Hence $R(I+\lambda A_Y)=Y$ and $A_Y$ is $m$-accretive.  Since
$D(A_Y)\subset D(A)\cap Y$, taking the closure in $Y$ gives
$\overline{D(A_Y)}^{\,Y}\subset Y\cap\overline{D(A)}$.  Conversely, if
$x\in Y\cap\overline{D(A)}$, then $J_\lambda x\in D(A_Y)$ and
$J_\lambda x\to x$ as $\lambda\downarrow0$.  This proves
\eqref{part-domain-identity}.

If $q\in Y_n$, then $K(t)\cap Y_{n+1}\subset K(t)\cap Y\subset K(t)$,
and hence \eqref{stage-distance-identity} gives
\[
 d_X(q,K(t))
 \leq d_Y(q,K(t)\cap Y)
 \leq d_X(q,K(t)\cap Y_{n+1})
 = d_X(q,K(t)).
\]
Thus \eqref{section-distance-identity} holds for every
$q\in\bigcup_nY_n$ and every $t\in J$.  Since $\bigcup_nY_n$ is dense
in $Y$ and both distance functions are $1$-Lipschitz in their first
argument, \eqref{section-distance-identity} holds for every
$t\in J$ and $y\in Y$.
In particular, $K_Y(t)\neq \emptyset$ for $t\in J$.  If $x\in K_Y(t)$ and $s\downarrow t$, then
$d_Y(x,K_Y(s))=d_X(x,K(s))\to0$, so $K_Y$ is lower semicontinuous from the right.

The forcing inclusion passes to $Y$: Fix $t\in I$ and $x\in K(t)\cap Y$.
Choose indices $n_\ell,j_\ell,k_\ell$ such that $q_{n_\ell,j_\ell}\to x$
and $2^{-k_\ell}\to0$.  The almost-minimizing estimate
\eqref{saturation-selector-estimate} gives
$p_{n_\ell,j_\ell,k_\ell}(t)\to x$, while all corresponding forcing values
belong to $Y$.  Closedness of $Y$ and continuity of $f(t,\cdot)$ yield
\eqref{forcing-inclusion}.

Finally, $\Gr(K_Y)=\Gr(K)\cap(J\times Y)$ is product measurable when viewed
in $J\times Y$.  The map \eqref{reduced-forcing} is measurable with respect to the
trace product $\sigma$-algebra on $\Gr(K_Y)$ and is
$Y$-valued, while $f_Y(t,\cdot)$ is continuous on $K_Y(t)$ for every
$t\in J$. For every fixed $x\in Y$, trace measurability of $f_Y$ on the
product-measurable graph makes the zero extension of
$t\mapsto f_Y(t,x)$ measurable; since $Y$ is separable, it is strongly
measurable. Hence $f_Y$ is a jointly measurable Carath\'eodory map.
\end{proof}

\subsection{Separable reduction in the bounded effective-tube setting}

Throughout this subsection $f$ is Carath\'eodory.  Whenever a compact
subinterval $J_0=[\alpha,\beta]\subset J$ is considered, the bounded
effective-tube setting is understood for the restricted problem according
to the convention following Definition~\ref{bounded-effective-setting}.
Thus sectionwise strong time measurability and almost-everywhere state
continuity on $J_0$ are baseline assumptions; the two alternatives below
concern only the additional structure needed for separable reduction.  The
graph route assumes joint measurability, whereas the fixed-cylinder route
assumes a Carath\'eodory extension.  Neither structure implies the other in
the present nonseparable setting.

The change of variables $r=t-\alpha$ identifies the restricted problem
with the corresponding problem on $[0,\beta-\alpha]$, so the results of
the preceding sections also apply with $\alpha$ as initial time.

If $Y$ is resolvent invariant, put
$K_Y(t):=K_A(t)\cap Y$ and $A_Y:=A\cap(Y\times Y)$.  The corresponding
$A_Y$-subtangential set in $Y$ is denoted by
$T_{K_Y}^{A_Y}(t,x)$.

\begin{definition}[Separable-reduction alternatives]
\label{separable-reduction-assumptions}
Let the bounded effective-tube setting hold on $J_0=[\alpha,\beta]$ and
let $f:\Gr(K_A|_{J_0})\to X$ be Carath\'eodory.
We say that the \emph{graph-measurable alternative \emph{(SR1)}}  holds if
$f$ is jointly measurable on $\Gr(K_A|_{J_0})$.
We say that the \emph{fixed-cylinder alternative \emph{(SR2)}} holds if there are a
set $C\subset X$ with $K_A(J_0)\subset C$ and a Carath\'eodory map
$F:J_0\times C\to X$ 
whose restriction to $\Gr(K_A|_{J_0})$ is $f$.
\end{definition}

\begin{theorem}[Reduction to a separable subspace]
\label{separable-reduction-theorem}
Let $A$ be $m$-accretive in a real Banach space $X$, let
$K:J=[0,a]\to2^X$ have closed values, let $f:\Gr(K_A)\to X$ be Carath\'eodory,
and let $J_0=[\alpha,\beta]\subset J$.  Suppose that the bounded
effective-tube setting holds on $J_0$ with common bound $c$ and exceptional
set $N_0$, and that \emph{(SR1)} or \emph{(SR2)} holds on $J_0$.

Let $D\subset K_A(\alpha)$ be nonempty, bounded, and separable, and let
$H\subset X$ be separable with $D\subset H$.  Then there are a closed
separable subspace $Y\supset H$, a full-measure set $I\subset J_0$, a null
set $N\subset[\alpha,\beta)$, a finite-valued $\gamma\in L^1_+(J_0)$, and a
jointly measurable Carath\'eodory map $f_Y:\Gr(K_Y)\to Y$, where
$K_Y(t):=K_A(t)\cap Y$, such that:
\begin{enumerate}
\item[(a)] $D\subset K_Y(\alpha)$ and $J_\lambda Y\subset Y$ for every
$\lambda>0$.  The part $A_Y:=A\cap(Y\times Y)$ is $m$-accretive in $Y$,
its resolvents and semigroup are the restrictions of those of $A$, and
$\overline{D(A_Y)}^{\,Y}=Y\cap\overline{D(A)}$.  Consequently, $K_Y(t)\subset\overline{D(A_Y)}^{\,Y}$ for every $t\in J_0$.

\item[(b)] Both $K_A$ and $K_Y$ have nonempty values and are lower
semicontinuous from the right on $J_0$.  Their graphs are closed from the
left, and $\Gr(K_Y)$ is product measurable.

\item[(c)] For every $t\in J_0$ and $y\in Y$,
\begin{equation}\label{reduction-all-time-distance}
 d_Y(y,K_Y(t))=d_X(y,K_A(t)).
\end{equation}
For every $t\in I$, $f(t,\cdot)$ is continuous on $K_A(t)$ and
$f(t,K_Y(t))\subset Y$.  Moreover,
$f_Y(t,x)=f(t,x)$ for all $x\in K_Y(t)$, and
$\|f_Y(t,x)\|\leq\gamma(t)$ for all $t\in J_0$ and $x\in K_Y(t)$.

\item[(d)]
\[
 f_Y(t,x)\in T_{K_Y}^{A_Y}(t,x)
 \quad\text{for all }t\in[\alpha,\beta)\setminus N
 \text{ and }x\in K_Y(t),\vspace{-0.05in}
\]
and\vspace{-0.05in}
\[
 \begin{gathered}
 \displaystyle
 \liminf_{h\downarrow0}\frac1h
 \left(d_Y(S_Y(h)x,K_Y(t+h))-\int_t^{t+h}\gamma(s)\,ds\right)^+=0,\\[2pt]
 \text{for all }t\in N\text{ and }x\in K_Y(t).
 \end{gathered}
\]
\end{enumerate}
The construction is common to all initial values in $D$.
\end{theorem}

\begin{proof}
After the change of variables $r=t-\alpha$, it suffices to consider
$J_0=J=[0,a]$.  The bounded effective-tube setting gives
\eqref{bounded-setting-forcing}--\eqref{bounded-setting-exceptional} with
common bound $c$ and exceptional set $N_0$.  Since
$f(t,x)\in T_K^A(t,x)$ and $\|f(t,x)\|\leq c(t)$ outside $N_0$,
Lemma~\ref{approx1} and Corollary~\ref{right-persistence-corollary} show
that $K_A$ has nonempty sections and is lower semicontinuous from the right.

\smallskip
\noindent
\emph{The graph-measurable route \emph{(SR1)}.}
Apply Theorem~\ref{preserve-sections-values} to the set-valued map $K_A$,
with the restricted forcing and the prescribed separable set $H$.  This
gives a closed separable subspace $Y\supset H$, a full-measure set $I$,
the all-time distance identity \eqref{reduction-all-time-distance},
resolvent invariance, and a reduced jointly measurable Carath\'eodory forcing $f_Y$ that agrees
with $f$ on $I$ and is zero on its complement.

\smallskip
\noindent
\emph{The fixed-cylinder route \emph{(SR2)}.}
Put $Y_0=\overline{\operatorname{span}}H$.  Suppose that $Y_n$ has been
constructed and choose a dense sequence $(q_{n,j})_{j\geq1}$ in $Y_n$.
For every $j,k\geq1$, let $p_{n,j,k}$ be the almost minimizing selector of
$K_A$ associated with $q_{n,j}$ and error $2^{-k}$.  Choose also
$r_{n,j,k}\in C$ such that
\begin{equation}\label{fixed-cylinder-state-approximation}
 \|r_{n,j,k}-q_{n,j}\|\leq d(q_{n,j},C)+2^{-k},
\end{equation}
and, by Lemma~\ref{lem:essential-separable-range}, a closed separable subspace
$W_{n,j,k}$ containing the essential range of $F(\cdot,r_{n,j,k})$.  Define
\[
 Y_{n+1}:=\overline{\operatorname{span}}\Big(
 Y_n\cup\bigcup_{j,k\geq1}p_{n,j,k}(J)
 \cup\bigcup_{j,k\geq1}\bigl(\{r_{n,j,k}\}\cup W_{n,j,k}\bigr)
 \cup\bigcup_{r\in\Q_+}J_r(Y_n)\Big).
\]
This is a closed separable subspace.  Let
$Y=\overline{\bigcup_nY_n}$ and let $I$ be the complement in $J$ of
the union of the null set on which $F(t,\cdot)$ is not continuous on $C$
and the countably many exceptional sets associated with the essential ranges.

The selector argument in the proof of
Theorem~\ref{preserve-sections-values} gives
\eqref{reduction-all-time-distance}.  Furthermore,
\begin{equation}\label{fixed-cylinder-inclusion}
 F(t,C\cap Y)\subset Y\quad\text{for all }t\in I.
\end{equation}
Indeed, if $x\in C\cap Y$, choose indices $n_\ell,j_\ell,k_\ell$ such that
$q_{n_\ell,j_\ell}\to x$ and $2^{-k_\ell}\to0$.  Since $x\in C$,
$d(q_{n_\ell,j_\ell},C)\leq\|q_{n_\ell,j_\ell}-x\|$, and hence
\eqref{fixed-cylinder-state-approximation} gives
$\|r_{n_\ell,j_\ell,k_\ell}-x\|
\leq2\|q_{n_\ell,j_\ell}-x\|+2^{-k_\ell}\to0$.
For $t\in I$, the corresponding values of $F$ belong to
$W_{n_\ell,j_\ell,k_\ell}\subset Y$, so
\eqref{fixed-cylinder-inclusion} follows from the closedness of $Y$ and
continuity of $F(t,\cdot)$.  
Including the rational resolvent images in the construction and using
the resolvent identity as in Theorem~\ref{preserve-sections-values}
gives $J_\lambda Y\subset Y$ for every $\lambda>0$, the $m$-accretivity of
$A_Y$, and the domain identity in (a).

Since $K_A$ is lower semicontinuous from the right, the distance identity
implies that $K_Y$ has nonempty values and is lower semicontinuous from the
right.  Hence
Lemma~\ref{right-distance-measurability} shows that
$t\mapsto d_Y(y,K_Y(t))$ is Borel measurable for each $y\in Y$, while
$y\mapsto d_Y(y,K_Y(t))$ is $1$-Lipschitz.  Since $Y$ is separable,
Proposition~\ref{cross-meas} makes the distance function product measurable;
its zero set is $\Gr(K_Y)$.

Put $C_Y:=C\cap Y$ and define $F_Y:J\times C_Y\to Y$ by
$F_Y(t,y):=F(t,y)$ for $t\in I$ and $F_Y(t,y):=0$ for $t\in J\setminus I$.
For every $y\in C_Y$, the map $F_Y(\cdot,y)$ is strongly measurable in
$X$, takes values in the closed subspace $Y$, and is therefore strongly
measurable as a $Y$-valued map.  Moreover, $F_Y(t,\cdot)$ is continuous
for every $t$.  Since $C_Y$ is separable, Proposition~\ref{cross-meas}
shows that $F_Y$ is product measurable.  Its restriction to the
product-measurable graph $\Gr(K_Y)$ is the required jointly measurable
Carath\'eodory map $f_Y$.

\smallskip
\noindent
\emph{Common conclusions.}
In both constructions, for every $t\in I$,
$f(t,K_Y(t))\subset Y$ and
$f_Y(t,x)=f(t,x)$ for all $x\in K_Y(t)$.  
The distance identity implies that $K_Y(t)$ is nonempty for every
$t\in J$.  Together with the right lower semicontinuity of $K_A$, it also
implies that $K_Y$ is lower semicontinuous from the right.
The reduced graph is closed from the
left because it is the intersection of $\Gr(K_A)$ with $J\times Y$, and it
is product measurable in both constructions: under \emph{(SR1)} by
Theorem~\ref{preserve-sections-values}, and under \emph{(SR2)} by the
distance-function argument in the fixed-cylinder branch.
The domain identity in (a), together with
$K_A(t)\subset\overline{D(A)}$, gives
$K_Y(t)\subset\overline{D(A_Y)}^{\,Y}$ for every $t$.

\smallskip
\noindent
\emph{Transfer of the common estimates.}
Let $\widehat c^{\,\circ}$ be the finite-valued representative of the local
majorant associated with $c$ by Proposition~\ref{c-hat}, and put
$\gamma:=c+\widehat c^{\,\circ}$,
$M:=J\setminus I$, and $N:=N_0\cup(M\cap[0,a))$.  Then $\|f_Y(t,x)\|\leq\gamma(t)$ on the entire reduced graph.
Fix $t\notin N$ and $x\in K_Y(t)$.  Then $t\in I$ and
$z=f_Y(t,x)=f(t,x)\in Y$.  For $y\in Y$ and $\lambda>0$,
$(I+\lambda(A_Y-z))^{-1}y=J_\lambda(y+\lambda z)\in Y$.
Hence the translated semigroup in $Y$ is the restriction of $S_z$.
By \eqref{bounded-setting-regular} and
\eqref{reduction-all-time-distance}, there is a sequence $h_n\downarrow0$
such that
$d_Y(S_z(h_n)x,K_Y(t+h_n)) =d_X(S_z(h_n)x,K_A(t+h_n))=o(h_n)$.  
This proves the subtangential condition in (d) at all regular times.

If $t\in N_0$, the exceptional estimate
\eqref{bounded-setting-exceptional} transfers directly through
\eqref{reduction-all-time-distance}, and replacing $c$ by $\gamma$ only
weakens it.  It remains to consider $t\in (M\cap[0,a))\setminus N_0$.  Put
$z=f(t,x)$.  By \eqref{bounded-setting-regular}, along a sequence
$h_n\downarrow0$,
$d_X(S_z(h_n)x,K_A(t+h_n))=o(h_n)$.  Proposition~\ref{mild-difference}
and \eqref{bounded-setting-forcing} give
$d_X(S(h_n)x,K_A(t+h_n))\leq h_nc(t)+o(h_n)$.
For all sufficiently large $n$, the local dominance of $\widehat c$ and
the equality of its integrals with those of $\widehat c^{\,\circ}$ yield
$h_nc(t)\leq\int_t^{t+h_n}\widehat c^{\,\circ}(s)\,ds$.
Using \eqref{reduction-all-time-distance} proves the exceptional condition
on $(M\cap[0,a))\setminus N_0$.  This completes the proof.
\end{proof}

\begin{remark}[Scope of the reduction]
The construction yields one space $Y$ for the entire reduced graph and for
all initial values in $D$.  It does not assert that every ambient solution
starting from these values is contained in $Y$.  For example, let $X$ be
nonseparable, $A=0$, $K(t)=X$, and
\[
 f(x)=\begin{cases}x/\sqrt{\|x\|},&x\neq0,\\0,&x=0.\end{cases}
\]
Besides the zero solution, for every unit vector $e\in X$ and $\tau\geq0$,
$u_{e,\tau}(t)=\frac{(t-\tau)_+^2}{4}\,e$ is a solution starting from $0$.
A separable invariant subspace can contain only the solutions in directions
belonging to that subspace and cannot capture all ambient solutions.

\end{remark}

\section{Approximate solutions}\label{s:approx}
We start with the construction of approximate solutions in a separable Banach space.

\begin{lemma}[Construction of approximate solutions]\label{approx3}
Let $A$ be $m$-accretive in a real separable Banach space $X$, and let
$K:J=[0,a]\to2^X$ have closed values with $K_A(0)\neq\emptyset$.
Assume that $\Gr(K_A)$ is closed from the left.  Let
$f:\Gr(K_A)\to X$ be a jointly measurable Carath\'eodory map and suppose
that there are a finite-valued $c\in L^1_+(J)$ and a null set
$N\subset[0,a)$ such that
\begin{align}
 &\|f(t,x)\|\leq c(t)
 &&\mbox{for all }(t,x)\in\Gr(K_A),
 \label{int-bound}\\
 &f(t,x)\in\TAK(t,x)
 &&\mbox{for all }t\in[0,a)\setminus N\mbox{ and }x\in K_A(t),\vspace{-0.1in}
 \label{subt-cond2}
\end{align}
and\vspace{-0.1in}
\begin{equation}\label{except-cond2}
 \begin{gathered}
  \liminf_{h\downarrow0}\frac1h
  \left(
   d(S(h)x,K_A(t+h))-\int_t^{t+h}\! c(s)\,ds
  \right)^+ \!=0\;
  \text{ for all }t\in N,\, x\in K_A(t).
 \end{gathered}
\end{equation}
Set $m:=c+\widehat c^{\,\circ}$.
Then, for every $u_0\in K_A(0)$ and every $\varepsilon\in(0,1]$, there
exist a closed set $E_\varepsilon\subset J\setminus N$, a continuous map
$v_\varepsilon:J\to X$, a forcing $w_\varepsilon\in L^1(J;X)$, a set
$P_\varepsilon\subset J$, and a strongly measurable map
$x_\varepsilon:E_\varepsilon\to X$ such that, with
\begin{equation}\label{error-density}
 \rho_\varepsilon(t):=
 \varepsilon+m(t)\varchi_{J\setminus E_\varepsilon}(t),
 \qquad
 \delta_\varepsilon:=\int_0^a\rho_\varepsilon(t)\,dt,
\end{equation}
the following holds:
\begin{equation}\label{delta-epsilon-bound}
 \lambda_1(J\setminus E_\varepsilon)
 +\int_{J\setminus E_\varepsilon}m(s)\,ds\leq\varepsilon,
 \qquad
 \delta_\varepsilon
 =a\varepsilon+\int_{J\setminus E_\varepsilon}m(s)\,ds
 \leq(a+1)\varepsilon,
\end{equation}
\begin{align}
 &v_\varepsilon(0)=u_0,
 \qquad 0,a\in P_\varepsilon,
 \qquad v_\varepsilon(\tau)\in K_A(\tau)
 &&\mbox{for all }\tau\in P_\varepsilon,
 \label{two-level-nodes}\\
 &\|w_\varepsilon(t)\|\leq c(t)
 &&\text{for a.e. }t\in J,
 \label{two-level-forcing-bound}\\
 &w_\varepsilon(t)\in\{0\}\cup f(t,K_A(t))
 &&\text{for a.e. }t\in J,
 \label{two-level-selection-or-zero}\\
 &\|v_\varepsilon(t)-
 u(t;\tau,v_\varepsilon(\tau),w_\varepsilon)\|
 \leq\int_\tau^t\rho_\varepsilon(s)\,ds
 &&\mbox{for all }\tau\in P_\varepsilon,\, \tau\leq t\leq a,
 \label{two-level-restart}\\
 &x_\varepsilon(t)\in K_A(t),\;\;
 w_\varepsilon(t)=f(t,x_\varepsilon(t)),\;\;
 \|x_\varepsilon(t)-v_\varepsilon(t)\|\leq\varepsilon
 &&\text{for a.e. }t\in E_\varepsilon.
 \label{current-time-selector}
\end{align}
Moreover, for every $t\in J$ there is
$\sigma_\varepsilon(t)\in P_\varepsilon$ such that
\begin{equation}\label{lagged-selector-property}
 (t-\varepsilon)^+\leq\sigma_\varepsilon(t)\leq t,
 \qquad
 \|v_\varepsilon(t)-v_\varepsilon(\sigma_\varepsilon(t))\|
 \leq\varepsilon.
\end{equation}
Consequently, if
$u_\varepsilon:=u(\cdot;0,u_0,w_\varepsilon)$, then
\begin{align}
 &\|v_\varepsilon(t)-u_\varepsilon(t)\|
 \leq\delta_\varepsilon
 &&\mbox{for all }t\in J,
 \label{v-u-global}\\
 &\|x_\varepsilon(t)-u_\varepsilon(t)\|
 \leq\varepsilon+\delta_\varepsilon
 &&\text{for a.e. }t\in E_\varepsilon,
 \label{appr-sol-incl}\\
 &v_\varepsilon(\sigma_\varepsilon(t))
 \in K_A(\sigma_\varepsilon(t)),\qquad
 \|v_\varepsilon(\sigma_\varepsilon(t))-u_\varepsilon(t)\|
 \leq\varepsilon+\delta_\varepsilon
 &&\mbox{for all }t\in J.
 \label{global-lagged-constraint}
\end{align}
\end{lemma}

\begin{proof}
Fix $u_0\in K_A(0)$ and $0<\varepsilon\leq1$, and define the finite
measure
\[
 \nu(B):=\lambda_1(B)+\int_B m(s)\,ds,
 \qquad B\in\mathcal L(J).
\]
By the Scorza--Dragoni property and absolute continuity of the integral,
choose a closed Scorza--Dragoni set $I_0\subset J$ such that
$\nu(J\setminus I_0)<\varepsilon/4$.  Since $\nu(N)=0$, inner regularity
gives a closed set $I\subset I_0\setminus N$ with
$\nu(I_0\setminus I)<\varepsilon/4$.  Thus $I$ is again a
Scorza--Dragoni set and $\nu(J\setminus I)<\varepsilon/2$.
By the right Lebesgue differentiation theorem, there is a measurable set
$M\subset I$ of full measure in $I$ such that, for every
$t\in M\cap[0,a)$,
\[
 \lim_{h\downarrow0}\frac{\lambda_1([t,t+h]\setminus I)}h=0,
 \qquad
 \lim_{h\downarrow0}\frac1h
   \int_{[t,t+h]\setminus I}c(s)\,ds=0.
\]
Since $\nu\ll\lambda_1$, $\nu(I\setminus M)=0$.  By inner regularity,
choose a closed $E=E_\varepsilon\subset M$ with
$\nu(M\setminus E)<\varepsilon/2$.  Then
$\nu(J\setminus E)\leq\varepsilon$, so
\eqref{delta-epsilon-bound} follows directly from
\eqref{error-density}.  Set $\rho:=\rho_\varepsilon$.

We shall also use the all-time estimate
\[
 \liminf_{h\downarrow0}\frac1h
 \left(
 d(S(h)x,K_A(t+h))-\int_t^{t+h}m(s)\,ds
 \right)^+=0
 \quad \text{ for all } t\in[0,a),\, x\in K_A(t).
\]
For $t\in N$ this follows from \eqref{except-cond2}, since
$m\geq c$.  If $t\notin N$, put $z=f(t,x)$ and choose
$h_n\downarrow0$ with
$d(S_z(h_n)x,K_A(t+h_n))=o(h_n)$.  Proposition~\ref{mild-difference}
and the local-majorant property give, for all sufficiently large $n$,
\[
\begin{aligned}
 d(S(h_n)x,K_A(t+h_n))
 \leq d(S_z(h_n)x,K_A(t+h_n))+h_nc(t) \leq o(h_n)+\int_t^{t+h_n}\widehat c^{\,\circ}(s)\,ds,
\end{aligned}
\]
which proves the claim.

We now construct a local extension.  Let $b<a$ and $x\in K_A(b)$.

\smallskip
\noindent
\emph{Case 1: $b\in E$.}
Put $z=f(b,x)$ and $q(s)=S_z(s-b)x$.  Lemma~\ref{approx1} and
Corollary~\ref{right-persistence-corollary} show that all sections
$K_A(s)$ are nonempty and that $K_A$ is lower semicontinuous from the
right.  Hence Lemma~\ref{right-distance-measurability} and
Proposition~\ref{cross-meas} show that the map $(s,y)\longmapsto d(y,K_A(s))$
is product measurable on $J\times X$.  Moreover, by the
joint-measurability hypothesis in the statement of the lemma,
$\Gr(K_A)$ is product measurable.

For $s\in[b,a]$ define
\[
 \Pi(s):=\left\{y\in K_A(s):
 \|y-q(s)\|\leq d(q(s),K_A(s))+(s-b)\right\}.
\]
The multifunction $\Pi$ has nonempty closed values.  Indeed,
$q(b)=x\in K_A(b)$, so $x\in\Pi(b)$, while for $s>b$ nonemptiness
follows from the definition of the distance.  Furthermore,
\[
 \Gr(\Pi)
 =
 \Gr(K_A)\cap([b,a]\times X)\cap
 \left\{(s,y):
 \|y-q(s)\|\leq d(q(s),K_A(s))+(s-b)\right\},
\]
and therefore $\Gr(\Pi)$ is product measurable.  The
von Neumann--Aumann measurable selection theorem yields a
Lebesgue-measurable selection $y:[b,a]\to X$; see
\cite[Theorem~III.22]{CaVa}.  Since $X$ is separable, $y$ is strongly
measurable.

By right lower semicontinuity,
$d(x,K_A(s))\to0$ as $s\downarrow b$, while $q(s)\to x$.  From the
definition of $\Pi$,
\[
\begin{aligned}
 \|y(s)-x\|
 &\leq \|y(s)-q(s)\|+\|q(s)-x\|\\
 &\leq d(q(s),K_A(s))+(s-b)+\|q(s)-x\|\\
 &\leq 2\|q(s)-x\|+d(x,K_A(s))+(s-b)\rightarrow0.\vspace{-0.05in}
\end{aligned}
\]
Thus\vspace{-0.05in}
\begin{equation}\label{projection-convergence}
 y(s)\rightarrow x\quad \text{ as } s\to b+.
\end{equation}

Set $w_b:=f(\cdot,y(\cdot))$.  Since
$(s,y(s))\in\Gr(K_A)$ for every $s\in[b,a]$,
Proposition~\ref{superposition} gives strong measurability of $w_b$;
moreover, $\|w_b(s)\|\leq c(s)$
for almost every $s\in[b,a]$.  Since $b\in E\subset M\subset I$, the
Scorza--Dragoni property on $I$, together with
\eqref{projection-convergence}, gives
\[
 \sup_{s\in[b,b+h]\cap I}\|w_b(s)-z\|\rightarrow0
 \quad \text{ as } h\to 0+.
\]
On the complementary set
$B_h:=[b,b+h]\setminus I$, the defining properties of $M$ give
\[
 \frac1h\int_{B_h}\|w_b(s)-z\|\,ds
 \leq
 \frac1h\int_{B_h}c(s)\,ds
 +\|z\|\frac{\lambda_1(B_h)}h
 \rightarrow0.
\]
Consequently,
\begin{equation}\label{local-average-continuity}
 \frac1h\int_b^{b+h}\|w_b(s)-z\|\,ds\rightarrow0
 \quad \text{ as } h\to 0+.
\end{equation}

Let $U(s)=u(s;b,x,w_b)$.  By Proposition~\ref{mild-difference} and
\eqref{local-average-continuity},
$\|U(b+h)-q(b+h)\|=o(h)$.  Choose a sequence $h_n\downarrow0$ realizing
$z\in\TAK(b,x)$.  Then
$d(U(b+h_n),K_A(b+h_n))=o(h_n)$, and for all sufficiently large $n$ we
may choose $x_1\in K_A(b+h_n)$ such that
\begin{equation}\label{good-endpoint-error}
 \|x_1-U(b+h_n)\|\leq\frac{\varepsilon h_n}{4}.
\end{equation}
Set $h=h_n$ with $n$ so large that $h\leq\varepsilon$, $b+h\leq a$, and
\begin{equation}\label{good-local-smallness}
 \sup_{b\leq s\leq b+h}\|U(s)-x\|\leq\varepsilon/4,
 \qquad
 \sup_{b\leq s\leq b+h}\|y(s)-x\|\leq\varepsilon/2.
\end{equation}
Put $\xi=x_1-U(b+h)$ and $R=\int_b^{b+h}\rho(s)\,ds$.
Since $\rho(s)\geq\varepsilon$ for all $s\in J$,
$R\geq\varepsilon h>0$,
while \eqref{good-endpoint-error} gives
\[
 \|\xi\|\leq\frac{\varepsilon h}{4}\leq R.
\]
Define
\[
 p_b(s)=\frac{\rho(s)}{R}\,\xi,
 \qquad
 V(s)=U(s)+\int_b^s p_b(\sigma)\,d\sigma.
\]
Then $\|p_b(s)\|\leq\rho(s)$ for almost every $s$, and
\[
 V(b)=x,\qquad V(b+h)=U(b+h)+\xi=x_1.
\]
Moreover,
\begin{equation}\label{local-rho-error}
 \|V(s)-u(s;b,x,w_b)\|
 \leq\int_b^s\rho(\sigma)\,d\sigma
 \qquad\mbox{for all }s\in[b,b+h].
\end{equation}
For $b\leq s\leq b+h$,
\[
 \left\|\int_b^s p_b(\sigma)\,d\sigma\right\|
 \leq\|\xi\|
 \leq\frac{\varepsilon h}{4}
 \leq\frac{\varepsilon}{4};
\]
recall that $h\leq1$.  Together with \eqref{good-local-smallness}, this yields
\begin{equation}\label{good-local-selector-distance}
 \|V(s)-x\|\leq\varepsilon/2,
 \qquad
 \|y(s)-V(s)\|\leq\varepsilon
 \quad\mbox{for all }s\in[b,b+h].
\end{equation}
Set $x_E=y$ on $E\cap[b,b+h]$.

\smallskip
\noindent
\emph{Case 2: $b\notin E$.}
Since $E$ is closed, choose $h_0>0$ with
$[b,b+h_0]\cap E=\emptyset$.  Along a sequence realizing the all-time
estimate, choose $h>0$ sufficiently small that
$h\leq\min\{\varepsilon,h_0\}$ and a point $x_1\in K_A(b+h)$ can be
chosen with
\begin{equation}\label{bad-endpoint-error}
 \|x_1-S(h)x\|
 \leq\varepsilon h+\int_b^{b+h}m(s)\,ds
 =\int_b^{b+h}\rho(s)\,ds.
\end{equation}
Put $w_b=0$, $U(s)=S(s-b)x$, $\xi=x_1-U(b+h)$, and
$R=\int_b^{b+h}\rho(s)\,ds$.  With the same definitions of $p_b$ and
$V$ as in Case~1, \eqref{local-rho-error} holds and, along this
construction,
\[
 \sup_{b\leq s\leq b+h}\|V(s)-x\|
 \leq\sup_{0\leq r\leq h}\|S(r)x-x\|+R\rightarrow0
 \qquad(h\downarrow0).
\]
Thus $h$ and $x_1$ may be chosen so that, in addition,
\begin{equation}\label{bad-local-smallness}
 \|V(s)-x\|\leq\varepsilon
 \qquad\mbox{for all }s\in[b,b+h].
\end{equation}
No current-time selector is needed on this interval, since
$[b,b+h]\cap E=\emptyset$.

\smallskip
We now use Zorn's lemma.  For $b\in[0,a]$, put $J_b=[0,b]$ and
$E_b=E\cap J_b$.  Let $\mathcal M_\varepsilon$ consist of all quintuples
$(v,w,x_E,P,b)$ with
\[
 \begin{gathered}
 v\in C(J_b;X),\quad v(0)=u_0,\quad
 w\in L^1(J_b;X),\quad P\subset J_b,\\
 x_E:E_b\to X\ \text{strongly measurable},
 \end{gathered}
\]
for which the analogues of
\eqref{two-level-nodes}--\eqref{lagged-selector-property} hold on $J_b$,
with $a$ replaced by $b$, $E_\varepsilon$ by $E_b$, and
$(v_\varepsilon,w_\varepsilon,x_\varepsilon,P_\varepsilon)$ by
$(v,w,x_E,P)$.  Equalities involving $w$ and $x_E$ are understood almost
everywhere.  Order $\mathcal M_\varepsilon$ by extension, i.e.\
\begin{align*}
 &(v,w,x_E,P,b)\preceq
 (\bar v,\bar w,\bar x_E,\bar P,\bar b) \quad :\Leftrightarrow\\
 & b\leq\bar b,\quad
 \bar v|_{J_b}=v,\quad
 \bar w|_{J_b}=w\ \text{a.e.},\quad
 \bar x_E|_{E_b}=x_E\ \text{a.e.},\quad
 P\subset\bar P.
\end{align*}
Then the degenerate object at $b=0$ belongs to $\mathcal M_\varepsilon$.

Let $(v,w,x_E,P,b)\in\mathcal M_\varepsilon$ with $b<a$.  Apply the
appropriate local construction at $(b,v(b))$, concatenate the old and
new pieces, and add $b+h$ to $P$.  The forcing and selector properties
are immediate.  For $\tau\in P$ and $s\in[b,b+h]$, the semigroup property
and contractivity give
\[
\begin{aligned}
 \|V(s)-u(s;\tau,v(\tau),w\oplus w_b)\| & \leq
 \|V(s)-u(s;b,v(b),w_b)\|
 +\|v(b)-u(b;\tau,v(\tau),w)\|\\
 & \leq
 \int_b^s\rho(r)\,dr+\int_\tau^b\rho(r)\,dr
 =\int_\tau^s\rho(r)\,dr,
\end{aligned}
\]
so the restart estimate is preserved.  Taking $b$ as lagged node on the
new interval verifies the lag condition by $h\leq\varepsilon$ and
\eqref{good-local-selector-distance} or \eqref{bad-local-smallness}.
Thus every object with $b<a$ has a strict extension.

Every chain has an upper bound.  Let $b^*$ be the supremum of its
terminal times.  If $b^*$ is attained, consistency on the common
intervals and the union of the node sets give an upper bound.  Otherwise
choose an increasing sequence of chain elements
$(v_n,w_n,x_{E,n},P_n,b_n)$ with $b_n\uparrow b^*$.  Every chain element
is then extended by one of these selected elements.  Their consistency defines $v^*,w^*$ on
$[0,b^*)$, $x_E^*$ on $E\cap[0,b^*)$, and
$P^*:=\bigcup_nP_n$.  Choosing compatible representatives gives strong
measurability of $w^*$ and $x_E^*$ on their respective domains.  The
bound $\|w^*(t)\|\leq c(t)$ a.e.\ gives
$w^*\in L^1([0,b^*];X)$, and the a.e.\ forcing and selector conditions
pass to the union.

For $\tau=b_n$ and $t\in[b_n,b^*)$, the restart estimate gives
\eqref{geometric-tail-estimate} with $\rho$ in place of the function
$\gamma$ occurring there.
The completion argument of Lemma~\ref{approx1} therefore yields a limit
$v^*(b^*):=\lim_{t\uparrow b^*}v^*(t)$.  Since
$b_n\in P_n$ and $v^*(b_n)\in K_A(b_n)$, left-closedness of
$\Gr(K_A)$ gives $v^*(b^*)\in K_A(b^*)$.  Add $b^*$ to $P^*$.
The restart estimate passes to $b^*$ by continuity of the prescribed-forcing mild solution; the lag property is inherited for $t<b^*$ and at
$t=b^*$ holds with $\sigma=b^*$.  If $b^*\in E$, set
$x_E^*(b^*):=v^*(b^*)$; endpoint values do not affect the a.e.\ selector
conditions.  Thus the chain has an upper bound.

Zorn's lemma gives a maximal element, whose terminal time must be $a$.
Denote it by
$(v_\varepsilon,w_\varepsilon,x_\varepsilon,P_\varepsilon,a)$.
Its defining properties are precisely
\eqref{two-level-nodes}--\eqref{lagged-selector-property}.
Finally, \eqref{v-u-global} follows from \eqref{two-level-restart} with
$\tau=0$; \eqref{appr-sol-incl} then follows from
\eqref{current-time-selector}, and \eqref{global-lagged-constraint}
from \eqref{two-level-nodes}, \eqref{lagged-selector-property}, and
\eqref{v-u-global}.
\end{proof}

\begin{remark}\label{two-level-interpretation}
The current-time selectors $x_\varepsilon(t)\in K_A(t)$, together with
$w_\varepsilon(t)=f(t,x_\varepsilon(t))$, identify the limiting forcing.
The lagged nodes
$v_\varepsilon(\sigma_\varepsilon(t))\in K_A(\sigma_\varepsilon(t))$
with $t-\sigma_\varepsilon(t)\to0$ recover the constraint in a uniform
limit.  The two constructions serve distinct purposes.  In particular,
lagged nodes at different times do not permit the forcing values to be compared by a
state-Lipschitz estimate when the time dependence is merely measurable.
\end{remark}

\begin{corollary}[Approximate solutions after separable reduction]
\label{approx4}
Let $D\subset K_A(0)$ be nonempty, bounded, and separable.  Suppose that
the bounded effective-tube setting holds on $J=[0,a]$, that
$f:\Gr(K_A)\to X$ is Carath\'eodory, and that \emph{(SR1)} or \emph{(SR2)}
holds on $J$.  Then there are a closed separable resolvent-invariant subspace
$Y\subset X$, a full-measure set $I\subset J$, a null set
$N\subset[0,a)$, and a finite-valued $\gamma\in L^1_+(J)$, all common to
all $u_0\in D$, with the following properties.  Put
$K_Y(t):=K_A(t)\cap Y$, let $A_Y$ be the part of $A$ in $Y$, and let
$f_Y$ be the reduced forcing.
Then $D\subset K_Y(0)$, $f(t,\cdot)$ is continuous on $K_A(t)$ for every
$t\in I$, and
\[
 f_Y(t,x)=f(t,x)
 \quad\text{for all }t\in I\text{ and }x\in K_Y(t),
\]
and
\[
 \|f_Y(t,x)\|\leq\gamma(t)
 \quad\text{for all }t\in J\text{ and }x\in K_Y(t),
\]
\[
 f_Y(t,x)\in T_{K_Y}^{A_Y}(t,x)
 \quad\text{for all }t\in[0,a)\setminus N\text{ and }x\in K_Y(t),\vspace{-0.05in}
\]
while\vspace{-0.05in}
\[
 \liminf_{h\downarrow0}\frac1h
  \left(d_Y(S_Y(h)x,K_Y(t+h))
        -\int_t^{t+h}\gamma(s)\,ds\right)^+=0
 \quad\text{for all }t\in N\text{ and }x\in K_Y(t).
\]
For every $0<\varepsilon\leq1$ and every $u_0\in D$, the conclusions of
Lemma~\ref{approx3} hold for this reduced problem.  Viewed in the ambient
space, the node values and current-time selectors belong to $K_A$ at their
respective times and all estimates remain valid.
\end{corollary}

\begin{proof}
Apply Theorem~\ref{separable-reduction-theorem} with
$H=\overline{\operatorname{span}}\,D$.  Its conclusion already provides one
common bound $\gamma$ for the reduced forcing and the exceptional estimate,
so Lemma~\ref{approx3} applies directly in $Y$.  Since the resolvents and
semigroup of $A_Y$ are the restrictions of those of $A$ and $Y$ has the
inherited norm, every constructed object and estimate is valid in the space $X$ as well.
\end{proof}

\section{Existence, uniqueness, and continuous dependence}\label{s:exist}
The existence results in this section are proved for Carath\'eodory forcing
under the bounded effective-tube setting of
Definition~\ref{bounded-effective-setting}.  The approximation is
constructed in the reduced separable problem, while the resulting solutions
are interpreted in the ambient Banach space.  The final corollary of the present
section transfers all three existence mechanisms back to the original
linear-growth setting.  

We first identify uniform limits of approximate solutions.

\begin{lemma}[Limit identification]\label{limit-identification}
Let $u_0\in K_A(0)$.  Suppose that the bounded effective-tube setting holds on
$J=[0,a]$, that $f:\Gr(K_A)\to X$ is Carath\'eodory, and that
\emph{(SR1)} or \emph{(SR2)} holds on $J$.  Let $Y$, $I$, and
$\gamma$ denote the reduced subspace, the full-measure agreement set, and
the common integrable bound provided by Corollary~\ref{approx4} for
$D=\{u_0\}$.  Choose a sequence $\varepsilon_n\downarrow0$ such that
$\sum_n\varepsilon_n<\infty$, and let
$E_n,w_n,u_n=u(\cdot;0,u_0,w_n),v_n,x_n,\sigma_n$ be the corresponding
approximate solutions.  Write
$\eta_n:=\varepsilon_n+\delta_{\varepsilon_n}$.  Then
$\|w_n(t)\|\leq\gamma(t)$ for almost every $t\in J$ and every $n$.

If, for some $\tau\in(0,a]$, $u_n\to u$ in $C([0,\tau];X)$, then
\begin{align}
 &u(t)\in K_Y(t)\subset K_A(t)
 &&\mbox{for all }t\in[0,\tau],
 \label{limit-in-tube}\\
 &w_n\rightarrow f(\cdot,u(\cdot))
 &&\hbox{in }L^1([0,\tau];X).
 \label{forcing-identification}
\end{align}
Moreover, $u$ is the mild solution of \eqref{sivp4} on $[0,\tau]$.
\end{lemma}

\begin{proof}
Every $u_n$ takes values in the closed subspace $Y$, and therefore so does
its uniform limit $u$.  Fix $t\in[0,\tau]$ and put
$y_n(t):=v_n(\sigma_n(t))$.  By \eqref{global-lagged-constraint},
$y_n(t)\in K_Y(\sigma_n(t))$, $\|y_n(t)-u_n(t)\|\leq\eta_n$, and
$(t-\varepsilon_n)^+\leq\sigma_n(t)\leq t$.  Hence
$\sigma_n(t)\to t$ and $y_n(t)\to u(t)$.  If $\sigma_n(t)=t$ for
infinitely many $n$, closedness of $K_Y(t)$ gives $u(t)\in K_Y(t)$ along
that subsequence.  Otherwise $\sigma_n(t)<t$ eventually; after passing to
an increasing subsequence, $\sigma_n(t)\nearrow t$, and left-closedness of
$\Gr(K_Y)$ gives the same conclusion.  This proves \eqref{limit-in-tube}.

For each $n$, discard the null set on which the identities in
\eqref{current-time-selector} may fail.  The union of these null sets is
still a null set.  Since $\sum_n\lambda_1(J\setminus E_n)<\infty$, the
Borel--Cantelli lemma shows that, for almost every $t\in[0,\tau]$, one has
$t\in I$ and $t\in E_n$ for all sufficiently large $n$.  For every such
$t$, \eqref{appr-sol-incl} and the uniform convergence $u_n\to u$ give
$x_n(t)\to u(t)$.  By the defining properties of $I$, $f(t,\cdot)$ is continuous on
$K_A(t)$, and the reduced forcing agrees with $f$ on
$K_Y(t)$.  Therefore, for all sufficiently large $n$,
$w_n(t)=f(t,x_n(t))\to f(t,u(t))$.  
Dominated convergence yields \eqref{forcing-identification},
hence $u_n=u(\cdot;0,u_0,w_n)\to u(\cdot;0,u_0,f(\cdot,u(\cdot)))$ in $C([0,\tau];X)$
by Proposition~\ref{mild-difference}.
Since, by assumption, $u_n\to u$ in the same space, uniqueness of the
limit gives $u=u(\cdot;0,u_0,f(\cdot,u(\cdot)))$.  Thus $u$ is the mild
solution of \eqref{sivp4} on $[0,\tau]$.
\end{proof}

\subsection{Compactness conditions}

The first existence criterion uses compactness of the homogeneous semigroup
on bounded subsets of the constrained state space.  The second assumes
compactness directly for families of forced-trajectory values.

\begin{theorem}[Viability under compactness]
\label{ex1}
Let $A$ be $m$-accretive in a real Banach space $X$, let
$K:J=[0,a]\to2^X$ have closed values with $K_A(0)\neq\emptyset$, and let
$f:\Gr(K_A)\to X$ be Carath\'eodory.  Suppose that the bounded
effective-tube setting holds on $J$ and that \emph{(SR1)} or \emph{(SR2)}
holds on $J$.
Finally, suppose that at least one of the following conditions is satisfied.
\begin{enumerate}
\item[(a)] For every $h>0$ and every bounded $B\subset K_A(J)$,
\begin{equation}\label{compact-on-tube}
 S(h)B\text{ is relatively compact in }X.
\end{equation}
\item[(b)] For every $t_0\in[0,a)$, every $u_0\in K_A(t_0)$, every
$\phi\in L^1_+([t_0,a])$, and every sequence
$(w_n)\subset L^1([t_0,a];X)$ satisfying
$\|w_n(t)\|\leq\phi(t)$ for almost every $t\in[t_0,a]$ and every $n\geq1$,
one has
\begin{equation}\label{forced-trajectory-compactness}
 \{u(t;t_0,u_0,w_n):n\geq1\}
 \text{ relatively compact in }X
 \quad\text{for every }t\in(t_0,a].
\end{equation}
\end{enumerate}
Then $K$ is viable.
\end{theorem}
\begin{proof}
Fix $t_0\in[0,a)$ and $u_0\in K_A(t_0)$.  After the change of variables
$r=t-t_0$, it suffices to assume $t_0=0$; the bounded effective-tube setting and
either separable-reduction alternative are preserved under this restriction.
Apply Corollary~\ref{approx4} with $D=\{u_0\}$ and choose
$\varepsilon_n\downarrow0$ with $\sum_n\varepsilon_n<\infty$.  Use the
notation of Lemma~\ref{limit-identification}; in particular,
$\|w_n(t)\|\leq\gamma(t)$ for almost every $t\in J$ and every $n$, where
$\gamma\in L^1_+(J)$ is independent of $n$.  Moreover,
\[
 y_n(t):=v_n(\sigma_n(t))\in K_Y(\sigma_n(t))\subset K_A(\sigma_n(t)),
 \qquad \|y_n(t)-u_n(t)\|\leq\eta_n\rightarrow0.
\]
The set $K_A(J)$ is bounded by the bounded effective-tube setting.

Assume first that (a) holds.  Fix $t>0$ and $h\in(0,t)$.  By
Proposition~\ref{mild-difference},
\begin{align*}
 \|u_n(t)-S(h)y_n(t-h)\|
 &\leq\|u_n(t)-S(h)u_n(t-h)\|
       +\|u_n(t-h)-y_n(t-h)\|\\
 &\leq\int_{t-h}^t\gamma(s)\,ds+\eta_n.
\end{align*}
The set $\{y_n(t-h):n\geq1\}$ is a bounded subset of $K_A(J)$, so
\eqref{compact-on-tube} makes
$\{S(h)y_n(t-h):n\geq1\}$ relatively compact.  Consequently, for every
$p\geq1$,
$\beta(\{u_n(t):n\geq p\})\leq\int_{t-h}^t\gamma(s)\,ds+
\sup_{n\geq p}\eta_n$.  Deleting finitely many points does not change $\beta$; hence the left-hand
side equals $\beta(\{u_n(t):n\geq1\})$.  Letting first $p\to\infty$ and
then $h\downarrow0$ proves that $\{u_n(t):n\geq1\}$ is relatively compact.

Under (b), the common bound $\|w_n(t)\|\leq\gamma(t)$ gives the same
relative compactness directly from
\eqref{forced-trajectory-compactness}.  Thus, in either case, every section
$\{u_n(t):n\geq1\}$ is relatively compact for $t>0$; at $t=0$ it is the
singleton $\{u_0\}$.

The sequence $(u_n)$ is equicontinuous.  If
$0\leq s\leq\min\{t,\bar t\}\leq a$, the semigroup property and
Proposition~\ref{mild-difference} imply
\begin{equation}
 \|u_n(t)-u_n(\bar t)\|
 \leq\|S(|t-\bar t|)u_n(s)-u_n(s)\|  +\int_s^t\gamma(\tau)\,d\tau
       +\int_s^{\bar t}\gamma(\tau)\,d\tau.
 \label{equi-estimate}
\end{equation}
To verify equicontinuity at a fixed $t_*>0$, choose
$s\in(0,t_*)$ so close to $t_*$ that the two integrals in
\eqref{equi-estimate} are uniformly small whenever $t$ and $\bar t$ lie in
a sufficiently small neighbourhood of $t_*$.  The closure of
$\{u_n(s):n\geq1\}$ is compact.  Strong continuity of $S$, uniformly on
this compact set, then makes the first term on the right-hand side of \eqref{equi-estimate}
uniformly small as $|t-\bar t|\to0$.  At zero use
$\|u_n(t)-u_0\|\leq\|S(t)u_0-u_0\|+\int_0^t\gamma(s)\,ds$
to show equicontinuity there.
Arzel\`a--Ascoli therefore gives a subsequence converging in $C(J;X)$.
Lemma~\ref{limit-identification} identifies its limit as a viable mild
solution.
\end{proof}

\begin{remark}[The two compactness alternatives]
\label{forced-compactness-verification}
Condition~\emph{(b)} is not a formal consequence of the tube-restricted
condition~\emph{(a)}, since the arbitrary forced trajectories occurring in
\eqref{forced-trajectory-compactness} need not remain in $K_A(J)$.  It does
follow if $S(h)$ maps every bounded subset of $\overline{D(A)}$ into a
relatively compact set for every $h>0$: for fixed $t>t_0$ and
$h\in(0,t-t_0)$, the set
$\{u(t-h;t_0,u_0,w_n):n\geq1\}$ is bounded and
$\|u(t;t_0,u_0,w_n)-S(h)u(t-h;t_0,u_0,w_n)\|
\leq\int_{t-h}^t\phi(s)\,ds$.  The measure-of-noncompactness argument used in the proof, followed by
$h\downarrow0$, gives \eqref{forced-trajectory-compactness}.

For fixed $t>t_0$, condition~\emph{(b)} also follows if the corresponding
endpoints are bounded, for every $\phi\in L^1_+([t_0,a])$, in a Banach
space $Z_{t_0,t}$ compactly embedded in $X$.  Such a bound may come from an energy or smoothing estimate for the
corresponding forced trajectories under the prescribed integrable forcing
bound.  Thus condition~\emph{(b)} can be strictly weaker
than compactness of all positive-time semigroup images.  For the filtration
operator $Au=-\Delta\varphi(u)$ in $L^1(\Omega)$, continuity and strict
monotonicity of $\varphi$ yield the required compactness of forced solution
values, whereas the homogeneous semigroup need not be compact; cf.\ \cite{BoWi12}.
\end{remark}

\subsection{A measure-of-noncompactness condition}

Assume that, for some $k\in L^1_+(J)$,
\begin{equation}\label{beta-f}
 \beta(f(t,B))\leq k(t)\beta(B)
 \quad\text{for a.e. }t\in J
 \text{ and every bounded }B\subset K_A(t).
\end{equation}
This includes sums of a compact Carath\'eodory map and a map that is
Lipschitz in the state variable with an integrable Lipschitz bound.
The compactness proof uses the following estimate.

\begin{lemma}\label{lemma-beta-w}
Let $X^*$ be uniformly convex, let $A$ be $m$-accretive in $X$, and
suppose that $-A$ generates an equicontinuous semigroup.  Let
$t_0\in J$, $x_0\in\overline{D(A)}$, and let
$(w_j)\subset L^1(J;X)$ satisfy
$\|w_j(t)\|\leq\phi(t)$ almost everywhere for all $j$, where
$\phi\in L^1_+(J)$.  Then
\begin{equation}\label{beta-w}
 \beta\bigl(\{u(t;t_0,x_0,w_j):j\geq1\}\bigr)
 \leq\int_{t_0}^t
 \beta\bigl(\{w_j(s):j\geq1\}\bigr)\,ds
 \qquad\mbox{for all }t\in[t_0,a].
\end{equation}
\end{lemma}

This is proved in \cite{Bo-1998}; that paper also shows that an estimate
of this form may fail without additional geometry of $X^*$.

\begin{theorem}[Viability under a measure-of-noncompactness estimate]\label{ex2}
Let $A$ be $m$-accretive in a real Banach space $X$, let
$K:J=[0,a]\to2^X$ have closed values with $K_A(0)\neq\emptyset$, and let
$f:\Gr(K_A)\to X$ be Carath\'eodory.  Suppose that the bounded
effective-tube setting holds on $J$ and that \emph{(SR1)} or \emph{(SR2)}
holds on $J$.
If $X^*$ is uniformly convex, the semigroup generated by $-A$ is
equicontinuous, and \eqref{beta-f} holds, then $K$ is viable.
\end{theorem}
\begin{proof}
Fix $t_0\in[0,a)$ and $u_0\in K_A(t_0)$.  After the change of variables
$r=t-t_0$, it suffices to assume $t_0=0$; the bounded effective-tube setting and
the separable-reduction alternative are preserved under restriction.
Apply Corollary~\ref{approx4} with $D=\{u_0\}$ and choose a summable
sequence $\varepsilon_n\downarrow0$.  Use the notation of
Lemma~\ref{limit-identification}; in particular, let $I$ be the
full-measure set on which the reduced forcing agrees with $f$.  Put
$\theta_p:=\sup_{n\geq p}\eta_n$,
$I_p:=I\cap\bigcap_{n\geq p}E_n$, and
$r_p:=\int_{J\setminus I_p}\gamma(s)\,ds$.  When choosing $E_n$, we may moreover assume that
$\int_{J\setminus E_n}\gamma(s)\,ds\leq\varepsilon_n$.
Since $J\setminus I$ is null, summability of $(\varepsilon_n)$ gives
\[
 r_p\leq\sum_{n\geq p}
       \int_{J\setminus E_n}\gamma(s)\,ds\rightarrow0,
 \qquad \theta_p\rightarrow0.
\]
Let $\beta_Y$ denote the Hausdorff measure of noncompactness in the closed
separable subspace $Y$.  For every bounded $B\subset Y$,
$\beta_X(B)\leq\beta_Y(B)\leq2\beta_X(B)$: the second inequality follows by
recentering each nonempty ball of a finite ambient cover at a point of $B$
and doubling its radius.  Put
$\varphi_p(t):=\beta_Y(\{u_n(t):n\geq p\})$.  A fixed countable dense set of
centres in $Y$ expresses every strict sublevel set of $\varphi_p$ as a
countable union of countable intersections of measurable inequalities;
hence $\varphi_p$ is measurable.  The same argument applies to
$s\mapsto\beta_Y(\{w_n(s):n\geq p\})$.

For almost every $s\in I_p$, the identities in
\eqref{current-time-selector} hold, the reduced forcing agrees with $f$, and
$\{x_n(s):n\geq p\}\subset\{u_n(s):n\geq p\}
 +\theta_p(\overline B_1(0)\cap Y)$.  Hence
\begin{align*}
 \beta_Y(\{w_n(s):n\geq p\})
 &\leq2\beta_X(\{w_n(s):n\geq p\})
 \leq2k(s)\beta_X(\{x_n(s):n\geq p\})\\
 &\leq2k(s)(\varphi_p(s)+\theta_p)
 \qquad\mbox{for a.e. }s\in I_p.
\end{align*}
On $J\setminus I_p$, the common bound gives
$\beta_Y(\{w_n(s):n\geq p\})\leq\gamma(s)$.  Since
$Y^*\simeq X^*/Y^\perp$ is a quotient of the uniformly convex space
$X^*$, it is uniformly convex, and the restricted semigroup is
equicontinuous.  Lemma~\ref{lemma-beta-w}, applied in $Y$, therefore yields
\[
 \varphi_p(t)\leq2\int_0^t k(s)\varphi_p(s)\,ds
 +2\theta_p\|k\|_1+r_p.
\]
Gronwall's inequality gives, uniformly in $t$,
\begin{equation}\label{beta-tail-bound}
 \varphi_p(t)\leq
 e^{2\|k\|_1}\bigl(2\theta_p\|k\|_1+r_p\bigr)
 \rightarrow0.
\end{equation}
Deleting finitely many points does not change $\beta_Y$, so
\eqref{beta-tail-bound} proves relative compactness of every section
$\{u_n(t):n\geq1\}$ in $Y$, and hence in $X$.

Equicontinuity follows from \eqref{equi-estimate} by the argument in the
proof of Theorem~\ref{ex1}.  Arzel\`a--Ascoli and
Lemma~\ref{limit-identification} complete the proof.
\end{proof}

\subsection{Local Lipschitz and one-sided Lipschitz conditions}

We call $f$ \emph{locally Lipschitz in the state variable} if, for every
$x_0\in X$, there are $r=r(x_0)>0$ and
$\omega=\omega_{x_0}\in L^1_+(J)$ such that
\begin{equation}\label{local-lipschitz}
 \|f(t,x)-f(t,y)\|\leq\omega(t)\|x-y\|
\end{equation}
for almost every $t$ and all
$x,y\in K_A(t)\cap\overline B_r(x_0)$.
It is \emph{locally one-sided Lipschitz} if instead
\begin{equation}\label{local-osl}
 (f(t,x)-f(t,y),x-y)_-
 \leq\omega(t)\|x-y\|^2
\end{equation}
holds under the same local restrictions.

\begin{theorem}[Locally Lipschitz well-posedness for moving constraints]\label{ex4}
Let $A$ be $m$-accretive in a real Banach space $X$, let
$K:J=[0,a]\to2^X$ have closed values with $K_A(0)\neq\emptyset$, and let
$f:\Gr(K_A)\to X$ be Carath\'eodory.  Suppose that the bounded
effective-tube setting holds on $J$ and that \emph{(SR1)} or \emph{(SR2)}
holds on $J$.
Then, for every $t_0\in[0,a)$ and $u_0\in K_A(t_0)$, the constrained
problem has a unique viable mild solution on $[t_0,a]$ if either
\begin{enumerate}
\item[(a)] $f$ is locally Lipschitz in the state variable, or
\item[(b)] $X^*$ is uniformly convex and $f$ is locally one-sided
Lipschitz in the state variable.
\end{enumerate}
For each fixed $t_0$, the solution depends continuously on the initial
value in $K_A(t_0)$ in the norm of $C([t_0,a];X)$.
\end{theorem}
\begin{proof}
Fix $t_0\in[0,a)$ and $u_0\in K_A(t_0)$.  After the change of variables
$r=t-t_0$, it suffices to assume $t_0=0$; the bounded effective-tube setting and
the separable-reduction alternative are preserved under restriction.
Apply Corollary~\ref{approx4} with $D=\{u_0\}$, and let $I$ and
$\gamma$ be, respectively, its full-measure agreement set and common
integrable bound.  All approximations are constructed in the same reduced
tube $K_Y$ and satisfy $\|w_n(t)\|\leq\gamma(t)$ almost everywhere.  Choose a
summable sequence $\varepsilon_n\downarrow0$.  When choosing $E_n$, we
may moreover assume that
$\int_{J\setminus E_n}\gamma(s)\,ds\leq\varepsilon_n$.
Use the notation of Lemma~\ref{limit-identification} and put
$I_p:=I\cap\bigcap_{n\geq p}E_n$,
$\theta_p:=\sup_{n\geq p}\eta_n$, and
$r_p:=\int_{J\setminus I_p}\gamma(s)\,ds$.  Since $J\setminus I$ is
null,
$r_p\leq\sum_{n\geq p}\varepsilon_n\to0$, and therefore
$\theta_p+r_p\to0$.

Choose $r>0$ and $\omega\in L^1_+(J)$ from
\eqref{local-lipschitz} or \eqref{local-osl}, with centre $u_0$.
The estimate
\begin{equation}\label{uniform-short-time}
 \|u_n(t)-u_0\|
 \leq\|S(t)u_0-u_0\|+\int_0^t\gamma(s)\,ds
\end{equation}
shows that there is $a_0>0$ such that
$\|u_n(t)-u_0\|\leq r/2$ on $[0,a_0]$ for every $n$.  After increasing
$p$, the points $x_n(t)$ also belong to
$\overline B_r(u_0)$ whenever $n\geq p$ and $t\in E_n\cap[0,a_0]$.

\smallskip
\noindent
\emph{Case (a).}
For $m,n\geq p$, Proposition~\ref{mild-difference} and the common regularity set $I_p$ give, on $[0,a_0]$,
\begin{align*}
 \|u_n(t)-u_m(t)\|
 &\leq\int_{[0,t]\cap I_p}
       \|f(s,x_n(s))-f(s,x_m(s))\|\,ds
       +2\int_{[0,t]\setminus I_p}\gamma(s)\,ds\\
 &\leq\int_0^t\omega(s)\|u_n(s)-u_m(s)\|\,ds
       +2\theta_p\|\omega\|_1+2r_p.
\end{align*}
Gronwall's inequality shows that $(u_n)$ is Cauchy in
$C([0,a_0];X)$.

\smallskip
\noindent
\emph{Case (b).}
Uniform convexity of $X^*$ makes the duality map $\mathcal F$ single
valued and uniformly continuous on bounded sets.  Write
$(z,y):=\mathcal F(y)(z)$.  The family $(u_n)$ is uniformly bounded on
$[0,a_0]$, and, for $n\geq p$, the selectors $x_n$ are uniformly bounded on
$I_p\cap[0,a_0]$; choose $R$ larger than these bounds.  Let
\[
 \vartheta_R(q):=
 \sup\{\|\mathcal F(\xi)-\mathcal F(\zeta)\|:
       \|\xi\|,\|\zeta\|\leq2R,
       \|\xi-\zeta\|\leq q\}.
\]
Then $\vartheta_R(q)\to0$ as $q\downarrow0$.  Put
$d=u_n-u_m$ and $e=x_n-x_m$.  On $I_p$,
$\|d-e\|\leq2\theta_p$, and hence
\begin{align*}
 (f(s,x_n)-f(s,x_m),d)
 &= (f(s,x_n)-f(s,x_m),e)\\
 &\quad+[\mathcal F(d)-\mathcal F(e)]
        (f(s,x_n)-f(s,x_m))\\
 &\leq\omega(s)\|e\|^2
       +2\gamma(s)\vartheta_R(2\theta_p)\\
 &\leq2\omega(s)\|d\|^2
       +8\omega(s)\theta_p^2
       +2\gamma(s)\vartheta_R(2\theta_p).
\end{align*}
On $J\setminus I_p$, the absolute value of the corresponding integrand
is at most $4R\gamma(s)$.  The squared estimate in
Proposition~\ref{mild-difference} therefore gives
\begin{align*}
 \|u_n(t)-u_m(t)\|^2
 &\leq4\int_0^t\omega(s)\|u_n(s)-u_m(s)\|^2\,ds\\
 &\quad+16\theta_p^2\|\omega\|_1
       +4\vartheta_R(2\theta_p)\|\gamma\|_1
       +8Rr_p.
\end{align*}
Gronwall's inequality again shows that $(u_n)$ is Cauchy in
$C([0,a_0];X)$.

In either case, Lemma~\ref{limit-identification} identifies the local
limit as a viable mild solution.  For local uniqueness, let $u$ and $v$ be
two viable solutions with the same initial value.  On any interval on which
both remain in one of the corresponding state balls, Proposition~\ref{mild-difference}
and \eqref{local-lipschitz} give in case~(a)
$\|u(t)-v(t)\|\leq\int_s^t\omega(r)\|u(r)-v(r)\|\,dr$, whereas in case~(b) the squared estimate and \eqref{local-osl} give
$\|u(t)-v(t)\|^2\leq2\int_s^t\omega(r)\|u(r)-v(r)\|^2\,dr$.
Gronwall's lemma yields local uniqueness.

The local solution extends to all of $J$.  Uniqueness permits a maximal
continuation on an interval $[0,\tau)$.  The local solution constructed
above is viable in the fixed reduced tube $K_Y$, and its right-hand side
is bounded by $\gamma$.  Hence
$\|u(t)-S(t)u_0\|\leq\int_0^t\gamma(s)\,ds$ for all $t<\tau$.
The composition $f(\cdot,u(\cdot))$ therefore belongs to
$L^1([0,\tau);X)$.  Extend it by zero past $\tau$ and solve the mild problem
with this prescribed right-hand side.  By uniqueness for the prescribed-forcing problem, this solution coincides with $u$ on every interval $[0,s]$, $s<\tau$.  Continuity of that solution shows that
$u(t)$ has a limit $u_\tau$ as $t\uparrow\tau$.
Left-closedness of the fixed reduced graph yields
$u_\tau\in K_Y(\tau)$.  On $[\tau,a]$, the fixed reduced problem retains the reduced forcing
$f_Y$, which is jointly measurable and Carath\'eodory, together with the
common bound $\gamma$ and the regular and exceptional subtangential
conditions.  Hence the local existence argument
restarts the solution at $(\tau,u_\tau)$ with the same bound $\gamma$.
Consequently $\tau=a$.

Finally, retain the initial difference in the estimates for two actual
solutions.  In case (a), Proposition~\ref{mild-difference} and Gronwall's
lemma give, as long as both solutions remain in one local Lipschitz ball,
\begin{equation}\label{local-solution-stability}
 \|u(t;x)-u(t;y)\|
 \leq \exp\!\left(\int_s^t\omega(r)\,dr\right)
       \|u(s;x)-u(s;y)\|.
\end{equation}
In case (b), the squared estimate and Gronwall's lemma give the same
estimate \eqref{local-solution-stability} after taking square roots.
For each point $z$ of the compact range of the reference solution
$u_*(t)=u(t;u_0)$, choose $r_z>0$ and the corresponding coefficient
$\omega_z\in L^1_+(J)$.  A finite family of the balls $B_{r_z/2}(z)$
covers $u_*(J)$.  The inverse images of these balls form a finite open
cover of $J$; hence a finite partition may be chosen so that the image of
each partition interval is contained in one of them.  On such an interval
the reference trajectory has distance at least $r_z/2$ from the complement
of $B_{r_z}(z)$.  Successive application of
\eqref{local-solution-stability} therefore excludes a first exit of a
nearby solution and gives, for all $x\in K_A(0)$ sufficiently close to
$u_0$,
\[
 \|u(\cdot;x)-u(\cdot;u_0)\|_{C(J;X)}
 \leq C_{u_0}\|x-u_0\|
\]
with $C_{u_0}<\infty$.  This proves continuous dependence in $C(J;X)$.
\end{proof}

\begin{remark}[Initial-pair form of the bounded existence proofs]
\label{initial-pair-existence}
The existence arguments above are local to the prescribed initial pair.
After restriction to $[t_0,a]$, existence from a fixed
$(t_0,u_0)\in\Gr(K_A)$ uses the bounded effective-tube setting,
Carath\'eodory regularity, and one separable-reduction alternative only on
that interval.  In
Theorem~\ref{ex1}(b), forced-trajectory compactness is needed only for this
$t_0$ and $u_0$.
\end{remark}

\begin{remark}[Static constraints]
\label{static-lipschitz-general-space}
The measurable-time locally Lipschitz result for a static closed constraint
was proved in \cite{Bo-JEE03}, without separability or compactness of the
ambient space.  Theorem~\ref{ex4} extends this mechanism to moving constraints
through the current-time selectors.
\end{remark}

\begin{corollary}[Return to the linear-growth setting]
\label{general-existence-theory}
Let $A$ be $m$-accretive in a real Banach space $X$, let
$K:J=[0,a]\to2^X$ have closed values with $K_A(0)\neq\emptyset$, assume
that $\Gr(K_A)$ is closed from the left, and let
$f:\Gr(K_A)\to X$ be Carath\'eodory.  Assume that there are a null set
$N\subset[0,a)$ and finite-valued
$c,\psi\in L^1_+(J)$ such that
\begin{align*}
 &\|f(t,x)\|\leq c(t)(1+\|x\|)
 &&\text{for all }(t,x)\in\Gr(K_A),\\
 &f(t,x)\in T_K^A(t,x)
 &&\text{for all }t\in[0,a)\setminus N\text{ and }x\in K_A(t),\vspace{-0.05in}
\end{align*}
and\vspace{-0.05in}
\[
 \liminf_{h\downarrow0}\frac1h
 \left(d(S(h)x,K_A(t+h))-\int_t^{t+h}\psi(s)\,ds\right)^+=0
 \quad\text{for all }t\in N\text{ and }x\in K_A(t).
\]
Assume, in addition, one of the following two alternatives:
\begin{enumerate}
\item[(G)] $f$ is jointly measurable on $\Gr(K_A)$;
\item[(C)] there are a set $C\subset X$ with $K_A(J)\subset C$ and a
Carath\'eodory map $F:J\times C\to X$ whose restriction to $\Gr(K_A)$ is
$f$.
\end{enumerate}
Then all conclusions of the bounded existence theory remain valid in the
following sense.
\begin{enumerate}
\item[(a)] If either compactness condition of Theorem~\ref{ex1} holds,
then $K$ is viable.
\item[(b)] If $X^*$ is uniformly convex, $-A$ generates an equicontinuous
semigroup, and \eqref{beta-f} holds, then $K$ is viable.
\item[(c)] If $f$ is locally Lipschitz in the state variable, or if $X^*$
is uniformly convex and $f$ is locally one-sided Lipschitz, then a unique viable mild solution
exists for every
admissible initial pair $(t_0,u_0)$; for each
fixed initial time the solution depends continuously on the initial value.
\end{enumerate}
The same assertions hold if the integral exceptional estimate is replaced
by the pointwise estimate of Corollary~\ref{bounded-pointwise-corollary}.
\end{corollary}

\begin{proof}
Fix an admissible initial pair $(t_0,u_0)$ and restrict the data to
$[t_0,a]$.  After translating time, apply
Lemma~\ref{bounded-submultifunction} with $D=\{u_0\}$.  It yields a bounded effective tube satisfying the bounded
effective-tube setting with one finite-valued integrable bound.
The Carath\'eodory property passes to the bounded effective tube by
restriction.  Under alternative~(G),
Lemma~\ref{bounded-submultifunction} gives product measurability of the
restricted effective graph, and joint measurability of $f$ passes to that
graph by restriction, so \emph{(SR1)} holds.
Under alternative~(C), the same fixed-cylinder Carath\'eodory map restricts to the
bounded effective tube, so \emph{(SR2)} holds.

The compactness assumptions in Theorem~\ref{ex1} pass to the bounded
effective tube, and the measure-of-noncompactness estimate \eqref{beta-f} passes
by monotonicity under restriction.  The local Lipschitz and one-sided
Lipschitz hypotheses pass trivially.  Applying
Theorems~\ref{ex1},~\ref{ex2}, or~\ref{ex4}, respectively, therefore gives
a solution viable in
the bounded effective tube and hence in the original tube.  Repeating this
argument for every admissible initial pair proves the viability statements.

Under either Lipschitz hypothesis, the direct Gronwall argument in the
proof of Theorem~\ref{ex4} shows uniqueness among all viable solutions in
the original tube, independently of the bounded reduction used to construct
them.  The same local stability estimate, applied to the resulting ambient
solutions from nearby initial values, gives the asserted continuous
dependence.  Finally, the pointwise exceptional variant follows from
Corollary~\ref{bounded-pointwise-corollary}.
\end{proof}

\section{Periodic solutions}\label{s:periodic}

For a $T$-periodic equation, a mild solution on $[0,T]$ with matching
endpoints extends by concatenation to a $T$-periodic mild solution.  Under
unique solvability, periodicity is therefore equivalent to the fixed-point
condition $P_Tx=x$ for the Poincar\'e operator $P_Tx:=u(T;x)$.
Without uniqueness, the endpoint map becomes set-valued.  Nonempty compact
values and a closed graph alone are insufficient for a fixed-point argument,
since the values need not be convex or acyclic.

Moving tubes permit time-dependent bounds
and require only the return inclusion $K_A(T)\subset K_A(0)$; a bounded
static invariant set for all intermediate times need not exist.  Accordingly,
the periodic criterion uses a viable return tube and requires convex
fixed-point geometry only in its initial section.
Even in the well-posed case, compactness and viability alone do not
guarantee a periodic solution.  The finite-dimensional counterexample in
\cite[Example~1]{Bo12} has a compact semigroup and a bounded Lipschitz forcing
satisfying the $A$-subtangential condition; all viable initial-value problems
are uniquely solvable, but no $T$-periodic solution exists for
$0<T<2\pi$.  The effective constraint in that example does not have the
fixed-point geometry needed for $P_T$.  Thus fixed-point geometry
of the initial section is a genuine part of the periodic problem, rather
than merely a technical feature of the proof.

Fixed-point methods for the Poincar\'e operator go back to
Browder \cite{Browder65}; see also
\cite{Pruss79,Becker81,Vrabie90,Bader98,Bader00,BaderKryszewski03,
KryszewskiGaborSiemianowski18,HiranoShioji04,APS06,Paicu08}.
The theory developed in \cite{Bo12} for static constraints without uniqueness
uses tangency-preserving approximations.
Here uniqueness and continuous dependence instead give a
single-valued Poincar\'e operator for a moving return tube.  The
forced-trajectory compactness alternative allows the initial states to vary
in the initial section; the measure-of-noncompactness argument of
Section~\ref{s:exist} is carried out for approximations with a fixed initial
value and does not by itself yield compactness of $P_T$ when the initial
states vary.

\begin{theorem}[Periodic solution from a viable return tube]
\label{periodic-theorem}
Let $T>0$, let $A$ be $m$-accretive in a real Banach space $X$, and let
$K:J=[0,T]\to2^X$ have closed values.  Assume that
$\Gr(K_A)$ is closed from the left, and set $C:=K_A(0)$.  Suppose that $C$
is nonempty, bounded, closed, and convex, and that
\begin{equation}\label{periodic-endpoint-inclusion}
 K_A(T)\subset C.
\end{equation}
Let $f:\Gr(K_A)\to X$ be Carath\'eodory.  Suppose that there are a
time-independent state domain $D_f\subset X$ containing $K_A(J)$ and a
$T$-periodic map $F:\mathbb R\times D_f\to X$ whose restriction to
$J\times D_f$ is Carath\'eodory and agrees with $f$ on $\Gr(K_A)$.  Assume that there are finite-valued
$c,\psi\in L^1_+(J)$ and a null set $N\subset[0,T)$ such that
\begin{equation}\label{periodic-linear-growth}
 \|f(t,x)\|\leq c(t)(1+\|x\|)
 \qquad\text{for all }(t,x)\in\Gr(K_A),
\end{equation}
\[
 f(t,x)\in T_K^A(t,x)
 \quad\text{for all }t\in[0,T)\setminus N\text{ and }x\in K_A(t),\vspace{-0.1in}
\]
and\vspace{-0.1in}
\[
 \liminf_{h\downarrow0}\frac1h
 \left(d(S(h)x,K_A(t+h))-\int_t^{t+h}\psi(s)\,ds\right)^+=0
 \quad\text{for all }t\in N\text{ and }x\in K_A(t).
\]
The last condition may be replaced by the pointwise exceptional
condition of Corollary~\ref{bounded-pointwise-corollary}.

Assume one of the following two uniqueness conditions:
\begin{enumerate}
\item[(i)] \emph{Local Lipschitz condition.}  For every $x_0\in X$ there
are $r>0$ and $\omega\in L^1_+(J)$ such that
$\|F(t,x)-F(t,y)\|\leq\omega(t)\|x-y\|$ for almost every
$t\in J$ and all $x,y\in D_f\cap\overline B_r(x_0)$;
\item[(ii)] \emph{Local one-sided Lipschitz condition.}  The space $X^*$
is uniformly convex and, for every $x_0\in X$, there are $r>0$ and
$\omega\in L^1_+(J)$ such that
$(F(t,x)-F(t,y),x-y)_-\leq\omega(t)\|x-y\|^2$ for almost every
$t\in J$ and all $x,y\in D_f\cap\overline B_r(x_0)$.
\end{enumerate}
Finally, assume one of the following compactness conditions:
\begin{enumerate}
\item[(a)] \emph{Positive-time semigroup compactness.}  For every $h>0$
and every bounded $B\subset K_A(J)$, $S(h)B$ is relatively compact in $X$;
\item[(b)] \emph{Forced-trajectory compactness.}  For every
$\phi\in L^1_+(J)$, every sequence $(x_n)\subset C$, and every sequence
$(w_n)\subset L^1(J;X)$ satisfying
$\|w_n(t)\|\leq\phi(t)$ for almost every $t\in J$ and every $n\geq1$,
one has
\begin{equation}\label{periodic-forced-trajectory-compactness}
 \{u(T;0,x_n,w_n):n\geq1\}
 \quad\text{relatively compact in }X.
\end{equation}
\end{enumerate}
Then the Poincar\'e operator $P_T:C\to C$ is continuous and compact, hence
has a fixed point.  Consequently, the periodically extended equation
$u'(t)+Au(t)\ni F(t,u(t))$ has a $T$-periodic mild solution whose restriction to $J$ is viable in
$K$.  
\end{theorem}
\begin{proof}
The fixed-cylinder Carath\'eodory map $F$ supplies alternative~(C) of
Corollary~\ref{general-existence-theory}.  Since the restriction of
$F$ to the moving graph is $f$, either uniqueness condition gives the
corresponding local Lipschitz or one-sided Lipschitz condition for $f$.
Corollary~\ref{general-existence-theory}(c) therefore gives, for every
$x\in C$, a unique viable solution $u(\cdot;x)$ on $J$, and the solution
depends continuously on $x$ in $C(J;X)$.  Define $P_Tx:=u(T;x)$.
Viability and \eqref{periodic-endpoint-inclusion} show that $P_T(C)\subset C$,
and continuous dependence makes $P_T$ continuous.

We first obtain a bound uniform in $x\in C$.  Fix $x_*\in C$ and put
$R_*:=\sup_{x\in C}\|x-x_*\|$ and $M_*:=\max_{t\in J}\|S(t)x_*\|$.
By the contraction property of $S$, Proposition~\ref{mild-difference},
and \eqref{periodic-linear-growth},
$\|u(t;x)\|\leq R_*+M_*+\int_0^t c(s)(1+\|u(s;x)\|)\,ds$.
Gronwall's inequality yields a constant $M<\infty$ such that
$\|u(t;x)\|\leq M$ for $t\in J$ and $x\in C$. Hence, with $\phi(t):=c(t)(1+M)$, one has
\begin{equation}\label{periodic-common-forcing-bound}
 \|f(t,u(t;x))\|\leq\phi(t)
 \quad\text{for almost every }t\in J\text{ and every }x\in C.
\end{equation}

Let $(x_n)\subset C$ and write $u_n=u(\cdot;x_n)$.  Under compactness
condition~(a), for every $h\in(0,T)$,
$\|u_n(T)-S(h)u_n(T-h)\|\leq\int_{T-h}^T\phi(s)\,ds$.
The set $\{u_n(T-h):n\geq1\}$ is bounded and contained in $K_A(J)$, so its
image under $S(h)$ is relatively compact.  Consequently,
$\beta(\{u_n(T):n\geq1\})\leq\int_{T-h}^T\phi(s)\,ds$.
Letting $h\downarrow0$ proves that $\{u_n(T):n\geq1\}$ is relatively
compact.  Under condition~(b), apply
\eqref{periodic-forced-trajectory-compactness} with
$w_n(t)=f(t,u_n(t))$ and the common bound
\eqref{periodic-common-forcing-bound}.  Thus $P_T$ is compact in either
case.

Schauder's theorem yields $x\in C$ such that $P_Tx=x$.  Hence the
corresponding viable mild solution $u(\cdot;x)$ satisfies
$u(T;x)=u(0;x)$.  Since $F$ is $T$-periodic, this endpoint identity permits
the solution to be concatenated over successive periods; its periodic
extension is therefore a $T$-periodic mild solution of the extended equation.
\end{proof}

\begin{remark}[Static constraints without uniqueness]
\label{periodic-static-theory-remark}
For a static closed bounded convex constraint,
\cite[Theorems~2 and~3]{Bo12} yields periodic mild solutions without
uniqueness of the
initial-value problem.  More precisely, under resolvent invariance, ordinary
subtangentiality, and positive-time compactness, it treats integrably bounded
$T$-periodic Carath\'eodory forcing when the constraint has nonempty interior
or is proximinal.  The proof constructs tangency-preserving locally Lipschitz
approximations and uses a retraction onto the effective constraint; 
in the proximinal case it also employs small outer neighbourhoods of
the constraint and approximate projections before passing to a compact
limit.
Thus, for static constraints, the conclusion is available for
Carath\'eodory forcing without the additional local state-Lipschitz
regularity and uniqueness required in Theorem~\ref{periodic-theorem}.

That approximation mechanism is tied to static convex geometry: the
ordinary tangent cone is fixed in time, and convex interpolation of tangent
velocities remains tangent to the same constraint.  For a moving tube the
admissible tangent sets depend on time, while the forcing may be defined only
on the variable graph, so the construction does not directly preserve the
moving subtangential conditions.  Theorem~\ref{periodic-theorem} therefore
uses uniqueness and continuous dependence to obtain a single-valued
Poincar\'e operator, but in return permits a genuinely moving return tube and
requires convex fixed-point geometry only at time zero.  The two theories thus
have complementary scopes.
\end{remark}

\begin{remark}[Separable reduction and the Poincar\'e operator]
Theorem~\ref{periodic-theorem} does not require the initial section $C$ to
be separable or all initial values to share one separable reduction.
Different initial values may lead to different reduced spaces; uniqueness
identifies the resulting ambient solutions and therefore defines a single
Poincar\'e operator.  The separable reductions therefore need not be common
to all initial values; compactness of the resulting endpoint family is
supplied separately by either condition~(a) or condition~(b).
\end{remark}

\section{Comparison principles}\label{s:comparison}
The classical comparison problem asks for conditions under which ordered
initial values give rise to ordered solutions.  In finite-dimensional
cooperative systems the relevant condition is quasimonotonicity: each
component of the vector field is nondecreasing in all other components.  In
a Banach lattice this is expressed by the cone-subtangentiality condition introduced
in Section~\ref{s:prel}.

The comparison theory has two levels.  The direct comparison principle
concerns two mild solutions that are already known to exist.  It is the
standard positive-part consequence of $m$-$T$-accretivity,
quasimonotonicity, and a local Lipschitz bound; it uses none of the
subtangential, reduction, or compactness hypotheses.  The second level is
ordered existence: given a viable solution $u$ and a larger initial value
$v_0$, the existence theory is applied to the solution-dependent upper tube
$K(t)\cap(u(t)+X_+)$ to produce a viable solution starting from $v_0$ and
remaining above $u$, even when uniqueness is unavailable.

We impose the lattice compatibility condition
\begin{equation}\label{K-cond}
 x,y\in K_A(t)\quad\Rightarrow\quad x\vee y\in K_A(t)
 \qquad\mbox{for all }t\in J.
\end{equation}
It is satisfied whenever the effective section $K_A(t)$, rather than merely $K(t)$, is an upper order
set or an order interval.  Thus one must also check that the intersection
with $\overline{D(A)}$ is stable under the relevant lattice operation.

Let $u,v:I\to X$ be mild solutions on a compact interval $I$ with ordered
initial values.  The following elementary estimate is used in the direct
comparison argument.

\begin{lemma}\label{quasimonotone-bracket}
Let $X$ be a Banach lattice, let $D\subset X$, and let $F:D\to X$ be
quasimonotone with respect to $X_+$.  Suppose that $x,y,x\vee y\in D$ and
that $F$ is Lipschitz with constant $L$ on a set containing $y$ and
$x\vee y$.  Then
\begin{equation}\label{quasimonotone-bracket-estimate}
 [x-y,F(x)-F(y)]_+\leq L\|(x-y)^+\|.
\end{equation}
\end{lemma}

\begin{proof}
Since $x$ and $y$ need not be comparable, set $z=x\vee y$ and write
$F(x)-F(y)=p+q$ with $p:=F(x)-F(z)$ and $q:=F(z)-F(y)$; quasimonotonicity
then applies to $x\leq z$, while $q$ is controlled by the Lipschitz estimate.
Set $a:=x-y$. Then $z-x=a^-$ and $z-y=a^+$. Since $x\leq z$,
quasimonotonicity and Proposition~\ref{lattice-facts}(v) give
\[
\frac 1 h \|(hp-a^-)^+\|
 =
\frac 1 h d(a^- -hp,X_+)\rightarrow 0
 \quad \text{ as } h\to 0+.
\]
Proposition~\ref{lattice-facts}(ii) yields
\[
 (a+hp)^+
 =\bigl(a^++hp-a^-\bigr)^+
 \leq a^++(hp-a^-)^+.
\]
Hence $\|(a+hp)^+\|\leq\|a^+\|+o(h)$, and therefore
$[a,p]_+\leq0$.

By Proposition~\ref{lattice-facts}(vii), the map
$\xi\mapsto\|\xi^+\|$ is $1$-Lipschitz. Thus
\[
 [a,p+q]_+
 \leq [a,p]_+ + \|q\|
 \leq \|F(z)-F(y)\|
 \leq L\|z-y\|
 =L\|a^+\|.
\]
Since $p+q=F(x)-F(y)$, this proves
\eqref{quasimonotone-bracket-estimate}.
\end{proof}

\begin{theorem}[Direct comparison of given solutions]
\label{direct-comparison}
Let $A:D(A)\to2^X\setminus\{\emptyset\}$ be $m$-$T$-accretive in the
Banach lattice $X$, let $I=[t_0,t_1]$, and let
$K:I\to2^X$ satisfy
$x,y\in K_A(t)\Rightarrow x\vee y\in K_A(t)$ for every $t\in I$.
Let $f:\Gr(K_A)\to X$ be locally Lipschitz in the state variable and
suppose that $f(t,\cdot)$ is quasimonotone on $K_A(t)$ for almost every
$t\in I$.  Let $u$ and $v$ be mild solutions on $I$ of
$z'+Az\ni f(t,z)$ such that $u(t),v(t)\in K_A(t)$ for every $t\in I$.  If
$u(t_0)\leq v(t_0)$, then $u(t)\leq v(t)$ for every $t\in I$.
\end{theorem}

\begin{proof}
Define $E=\{t\in I:u(s)\leq v(s)\text{ for every }s\in[t_0,t]\}$.
The set $E$ is nonempty and closed.  We show that every $\tau\in E$ with
$\tau<t_1$ has a right neighbourhood contained in $E$.

Put $x_0=v(\tau)$.  By local Lipschitz continuity there are $\delta>0$ and
$\omega\in L^1_+(I)$ such that, for almost every $t\in I$,
\begin{equation}\label{local-comparison-Lipschitz}
 \|f(t,x)-f(t,y)\|\leq\omega(t)\|x-y\|
\end{equation}
whenever $x,y\in K_A(t)\cap B_\delta(x_0)$.  The lattice operations are continuous, and
$u(\tau)\vee v(\tau)=v(\tau)=x_0$.  Hence there exists $\eta>0$ such that
\[
 v(t),\ u(t)\vee v(t)\in B_\delta(x_0)
 \qquad\text{for all }t\in[\tau,\tau+\eta]\cap I.
\]
The lattice compatibility assumption gives
$u(t)\vee v(t)\in K_A(t)$.

Apply Proposition~\ref{positive-difference} on $[\tau,t]$.  For almost
every $s\in[\tau,\tau+\eta]$, Lemma~\ref{quasimonotone-bracket}, with
$F=f(s,\cdot)$ and $L=\omega(s)$, yields
$[u(s)-v(s),f(s,u(s))-f(s,v(s))]_+
\leq\omega(s)\|(u(s)-v(s))^+\|$.
Since $u(\tau)\leq v(\tau)$,
$\|(u(t)-v(t))^+\|
\leq\int_\tau^t\omega(s)\|(u(s)-v(s))^+\|\,ds$
for every $t\in[\tau,\tau+\eta]\cap I$.  Gronwall's lemma gives
$(u(t)-v(t))^+=0$ on this interval.  Thus $E$ is right-open in $I$.
Closedness then implies $E=I$.
\end{proof}

For ordered existence without uniqueness, fix a viable solution $u$ and
consider the upper tube $K(t)\cap(u(t)+X_+)$.  This tube must be formed before
the bounded restriction, since an arbitrary intersection with a moving ball
need not preserve closure under lattice suprema.  Once the upper-tube
tangency has been established, however, the known solution $u(t)$ provides a
natural centre for restriction to a smaller bounded tube.  
The next lemma implements this
solution-centred reduction to a bounded sub-tube.

\begin{lemma}[Reduction of a solution-dependent upper tube]
\label{upper-tube-inheritance}
Let $A:D(A)\to2^X\setminus\{\emptyset\}$ be $m$-$T$-accretive in the
Banach lattice $X$, and let $K:J=[0,a]\to2^X$ have closed values.  Assume that
$\Gr(K_A)$ is closed from the left and that \eqref{K-cond} holds.
Let $f:\Gr(K_A)\to X$ be Carath\'eodory and satisfy, for a null set
$N\subset[0,a)$ and a finite-valued $c\in L^1_+(J)$,
\begin{align}
 \|f(t,x)\|&\leq c(t)(1+\|x\|)
 &&\text{for all }(t,x)\in\Gr(K_A),\label{upper-linear-growth}\\
 f(t,x)&\in T_K^A(t,x)
 &&\text{for all }t\in[0,a)\setminus N\text{ and }x\in K_A(t),
 \label{upper-regular-tangency}
\end{align}
and\vspace{-0.1in}
\begin{equation}\label{upper-exceptional-tangency}
 \liminf_{h\downarrow0}\frac1h
 \left(d(S(h)x,K_A(t+h))-\int_t^{t+h}c(s)\,ds\right)^+=0
 \quad\text{for all }t\in N,\, x\in K_A(t).
\end{equation}
Assume that $f(t,\cdot)$ is quasimonotone on $K_A(t)$ for almost every
$t$, and let $u$ be a viable mild solution.
In addition, assume one of the following two alternatives,
corresponding to the two separable-reduction routes:
\begin{enumerate}
\item[(G)] $f$ is jointly measurable on $\Gr(K_A)$; or
\item[(C)] there are a set $C\subset X$ with $K_A(J)\subset C$ and a
Carath\'eodory map $F:J\times C\to X$ whose restriction to $\Gr(K_A)$ is
$f$.
\end{enumerate}
Define
\begin{equation}\label{upper-tube-definition}
 K^u(t):=K(t)\cap\bigl(u(t)+X_+\bigr),\qquad
 K_A^u(t):=K_A(t)\cap\bigl(u(t)+X_+\bigr).
\end{equation}
For every nonempty bounded set $D_u\subset K_A^u(0)$, there is a
bounded closed-valued sub-tube $\widetilde K^u(t)\subset K^u(t)$ with
$D_u\subset\widetilde K_A^u(0)$ such that the upper problem on
$\widetilde K^u$ satisfies the bounded effective-tube setting of
Definition~\ref{bounded-effective-setting}, and the restricted forcing is
Carath\'eodory.  Under alternative~(G) the bounded upper problem satisfies
\emph{(SR1)}, whereas under alternative~(C) it satisfies \emph{(SR2)}.
\end{lemma}

\begin{proof}
Put $w:=f(\cdot,u(\cdot))$, choose a pointwise representative
$w(t)=f(t,u(t))$, and set
$c_u(t):=c(t)(1+\|u(t)\|)$.  Then $c_u\in L^1_+(J)$ and
$\|w(t)\|\leq c_u(t)$ almost everywhere.  The sections in
\eqref{upper-tube-definition} are nonempty and closed, and continuity of
$u$, closedness of $X_+$, and left-closedness of $\Gr(K_A)$ imply that
$\Gr(K_A^u)$ is closed from the left.  The Carath\'eodory property passes
by restriction.  Under~(G),
$\Gr(K_A^u)=\Gr(K_A)\cap\{(t,x):x-u(t)\in X_+\}$ is product measurable and
$f|_{\Gr(K_A^u)}$ is jointly measurable; under~(C), the same fixed-cylinder
extension $F$ is available.

Enlarge $N$ by the null set on which quasimonotonicity fails and by the
complement of the right Lebesgue points of $w$; call the resulting set
$N_u$.  Fix $t\notin N_u$, $x\in K_A^u(t)$, put $q=f(t,x)$, and set
$z(s)=S_q(s)x$.  Proposition~\ref{positive-difference} gives
\begin{equation}\label{upper-order-contact-estimate}
 \|(u(t+h)-z(h))^+\|
 \leq\int_0^h
 [u(t+s)-z(s),w(t+s)-q]_+\,ds.
\end{equation}
Since $u(t)\leq x$, quasimonotonicity gives
$[u(t)-x,w(t)-q]_+=0$.  By the upper semicontinuity and the
second-variable Lipschitz estimate recalled after
\eqref{positive-bracket-definition}, for every $\delta>0$,
\[
 [u(t+s)-z(s),w(t+s)-q]_+
 \leq\delta+\|w(t+s)-w(t)\|
\]
for all sufficiently small $s\geq0$.  Averaging and using that $t$ is a
right Lebesgue point of $w$ therefore shows from
\eqref{upper-order-contact-estimate} that
\begin{equation}\label{upper-order-contact}
 \|(u(t+h)-S_q(h)x)^+\|=o(h)\quad \text{ as } h\downarrow 0.
\end{equation}
We shall repeatedly use the elementary lattice estimate
\[
 \|y\vee u-z\|\leq2\|y-z\|+\|(u-z)^+\|.
\]
Indeed, $y\vee u=y+(u-y)^+$ and the positive-part map is $1$-Lipschitz.
Choose $h_n\downarrow0$ and $y_n\in K_A(t+h_n)$ from
\eqref{upper-regular-tangency}, so that
$\|y_n-S_q(h_n)x\|=o(h_n)$, and put
$\widetilde y_n:=y_n\vee u(t+h_n)\in K_A^u(t+h_n)$.  Applying the lattice
estimate above with $z=S_q(h_n)x$ and using \eqref{upper-order-contact} gives
$\|\widetilde y_n-S_q(h_n)x\|=o(h_n)$.  Hence
\begin{equation}\label{upper-tube-regular}
 f(t,x)\in T_{K^u}^A(t,x)
 \quad \text{ for all } t\in[0,a)\setminus N_u,\ x\in K_A^u(t).
\end{equation}

At exceptional times we use two elementary consequences of
Proposition~\ref{positive-difference} (recall that $u(t)\leq x$):
\begin{equation}\label{upper-unforced-order-estimates}
 \begin{aligned}
 \|(u(t+h)-S(h)x)^+\|
 &\leq\int_t^{t+h}c_u(s)\,ds,\\
 \|(S(h)x-u(t+h))^+\|
 &\leq\|x-u(t)\|+\int_t^{t+h}c_u(s)\,ds.
 \end{aligned}
\end{equation}
If $t\in N$, choose $h_n\downarrow0$ and $y_n\in K_A(t+h_n)$ from
\eqref{upper-exceptional-tangency} and again set
$\widetilde y_n=y_n\vee u(t+h_n)$.  The lattice estimate above, together with the first estimate in
\eqref{upper-unforced-order-estimates}, yields
\begin{equation}\label{upper-old-exceptional}
 \|\widetilde y_n-S(h_n)x\|
 \leq\int_t^{t+h_n}(2c+c_u)(s)\,ds+o(h_n).
\end{equation}
If $t\in N_u\setminus N$, put $q=f(t,x)$ and use the original regular
subtangentiality to choose $y_n$ with
$\|y_n-S_q(h_n)x\|=o(h_n)$; set
$\widetilde y_n:=y_n\vee u(t+h_n)$.  Since
$\|S_q(h_n)x-S(h_n)x\|\leq h_n\|q\|$, the lattice estimate above gives
\begin{equation}\label{upper-new-exceptional}
 \|\widetilde y_n-S(h_n)x\|
 \leq2h_nc(t)(1+\|x\|)
      +\int_t^{t+h_n}c_u(s)\,ds+o(h_n).
\end{equation}

It remains to make the upper tube bounded.  Let $\widehat c$ be the local
majorant from Proposition~\ref{c-hat}, choose $R_0>1$ with
$D_u\subset\overline B_{R_0}(u(0))$, and let
\begin{equation}\label{upper-moving-radius}
 r'(t)=1+c(t)+c_u(t)
       +\widehat c^{\,\circ}(t)(1+\|u(t)\|+r(t)),
 \qquad r(0)=R_0.
\end{equation}
Set
\begin{equation}\label{upper-bounded-tube}
 \widetilde K^u(t):=K^u(t)\cap\overline B_{r(t)}(u(t)).
\end{equation}
This is a bounded closed-valued sub-tube, contains $D_u$ at time zero,
and its effective graph is closed from the left; the measurability
conclusions under~(G) and~(C) pass to this further restriction.

No lattice compatibility is required of $\widetilde K^u$; it remains only to verify that the endpoints constructed above also lie in the moving ball.
Let $x\in\widetilde K_A^u(t)$.  If $t\notin N$, the endpoints are obtained
from $y_n=S_q(h_n)x+o(h_n)$, irrespective of whether $t$ belongs to $N_u$.
Using $\widetilde y_n-u(t+h_n)=(y_n-u(t+h_n))^+$ and
Proposition~\ref{positive-difference},
\[
 \|\widetilde y_n-u(t+h_n)\|
 \leq r(t)+h_nc(t)(1+\|u(t)\|+r(t))
       +\int_t^{t+h_n}c_u(s)\,ds+o(h_n).
\]
The local-majorant property of $\widehat c$, together with
\eqref{upper-moving-radius} and continuity of $u$ and $r$, makes the
right-hand side at most $r(t+h_n)$ for
all large $n$; the unit term absorbs the $o(h_n)$ remainder.  If $t\in N$,
the second estimate in \eqref{upper-unforced-order-estimates} instead gives
\[
 \|\widetilde y_n-u(t+h_n)\|
 \leq r(t)+\int_t^{t+h_n}(c+c_u)(s)\,ds+o(h_n)
 \leq r(t+h_n)
\]
for all large $n$.  Thus the regular endpoints and both exceptional
sequences lie in $\widetilde K_A^u(t+h_n)$.  In particular,
\begin{equation}\label{upper-bounded-regular}
 f(t,x)\in T_{\widetilde K^u}^A(t,x)
 \quad(t\in[0,a)\setminus N_u,\ x\in\widetilde K_A^u(t)).
\end{equation}

Finally, boundedness gives $M\geq1$ with $1+\|x\|\leq M$ on
$\Gr(\widetilde K_A^u)$.  Hence $\|f(t,x)\|\leq Mc(t)$ there.
In \eqref{upper-new-exceptional}, the local-majorant property converts
$2h_nMc(t)$ along the defining sequence into
\[
 2M\int_t^{t+h_n}\widehat c^{\,\circ}(s)\,ds.
\]
Consequently
\begin{equation}\label{upper-common-bound}
 \gamma:=Mc+2c+c_u+2M\widehat c^{\,\circ}\in L^1_+(J)
\end{equation}
controls both the forcing and the exceptional-time estimate.  Together with
\eqref{upper-bounded-regular}, this is exactly the bounded effective-tube
setting.  The restricted forcing is Carath\'eodory, with \emph{(SR1)} under
(G) and \emph{(SR2)} under (C).
\end{proof}

\begin{theorem}[Existence of an ordered companion]
\label{upper-solution-existence}
Let $A:D(A)\to2^X\setminus\{\emptyset\}$ be $m$-$T$-accretive in the
Banach lattice $X$, and let
$K:J=[0,a]\to2^X$ have closed values.  Assume that
$\Gr(K_A)$ is closed from the left and that
\[
 x,y\in K_A(t)\quad\Rightarrow\quad x\vee y\in K_A(t)
 \qquad\text{for all }t\in J.
\]
Let $f:\Gr(K_A)\to X$ be Carath\'eodory and suppose that there are a null
set $N\subset[0,a)$ and a finite-valued $c\in L^1_+(J)$ such that
\begin{align*}
 \|f(t,x)\|&\leq c(t)(1+\|x\|)
 &&\text{for all }(t,x)\in\Gr(K_A),\\
 f(t,x)&\in T_K^A(t,x)
 &&\text{for all }t\in[0,a)\setminus N\text{ and }x\in K_A(t),
\end{align*}
and\vspace{-0.1in}
\[
 \liminf_{h\downarrow0}\frac1h
 \left(d(S(h)x,K_A(t+h))-\int_t^{t+h}c(s)\,ds\right)^+=0
 \qquad\text{for all }t\in N,\, x\in K_A(t).
\]
Assume that $f(t,\cdot)$ is quasimonotone on $K_A(t)$ for almost every
$t$.  In addition, assume one
of the following alternatives:
\begin{enumerate}
\item[(G)] $f$ is jointly measurable on $\Gr(K_A)$;
\item[(C)] there are a set $C\subset X$ with $K_A(J)\subset C$ and a
      Carath\'eodory map $F:J\times C\to X$ whose restriction to
      $\Gr(K_A)$ is $f$.
\end{enumerate}
Let $u$ be a viable mild solution and let $v_0\in K_A(0)$ satisfy
$u(0)\leq v_0$.  Suppose, in addition, that at least one of the following
conditions holds:
\begin{enumerate}
\item $S(h)B$ is relatively compact for every $h>0$ and every bounded
      $B\subset K_A(J)$;
\item for every $\varphi\in L^1_+(J)$ and every sequence
      $(w_n)\subset L^1(J;X)$ satisfying
      $\|w_n(t)\|\leq\varphi(t)$ a.e.\ on $J$, the set
$\{u(t;0,v_0,w_n):n\geq1\}$ is relatively compact for every $t\in(0,a]$;
\item $X^*$ is uniformly convex, the semigroup $S$ is equicontinuous, and
      there is $k\in L^1_+(J)$ such that
$\beta(f(t,B))\leq k(t)\beta(B)$ for a.e.\ $t\in J$ and every bounded $B\subset K_A(t)$;
\item $f$ is locally Lipschitz in the state variable;
\item $X^*$ is uniformly convex and $f$ is locally one-sided Lipschitz in
      the state variable.
\end{enumerate}
Then there is a viable mild solution $v$ with $v(0)=v_0$ such that
$u(t)\leq v(t)$ for all $t\in J$.  Under conditions (4) or (5), the viable solution starting from $v_0$ is
unique and hence coincides with this ordered companion.  In particular, the
unique viable solutions starting from ordered initial values are ordered.
\end{theorem}

\begin{proof}
Apply Lemma~\ref{upper-tube-inheritance} with $D_u=\{v_0\}$ and let
$\widetilde K^u$ be the resulting upper sub-tube, whose effective tube
$\widetilde K_A^u$ is bounded.  Conditions
(1)--(3) pass from $K_A$ to $\widetilde K_A^u$.  For (1), every bounded
subset of $\widetilde K_A^u(J)$ is a bounded subset of $K_A(J)$; for (3),
monotonicity of $\beta$ under set inclusion gives the restricted estimate.
Thus, for the bounded upper problem, condition~(1) permits application
of Theorem~\ref{ex1}(a), condition~(3) of Theorem~\ref{ex2}, and
conditions~(4) or~(5) of Theorem~\ref{ex4}.
Under condition~(2), the proof of
Theorem~\ref{ex1}(b) applies to the prescribed pair $(0,v_0)$ by
Remark~\ref{initial-pair-existence}.  In every case a mild solution viable
in $\widetilde K^u$, and hence in $K^u$, is obtained; the order conclusion
follows from the definition of the upper tube.

In cases~(4) and~(5), the original effective-tube linear-growth and
subtangential hypotheses, the Carath\'eodory property, and the corresponding
measurability alternative permit application of
Corollary~\ref{general-existence-theory}.
Hence the viable solution is unique for every initial value in the original
tube.  Applying the preceding ordered-companion argument with the unique
solution from the smaller initial value gives the final assertion.
\end{proof}

\begin{remark}\label{upper-solution-reduction-remark}
The argument is unchanged after restriction to any interval $[t_0,a]$.
More precisely, after
restriction to $[t_0,a]$, the same conclusion holds for every
$v_0\in K_A(t_0)$ with $u(t_0)\leq v_0$ provided the hypotheses of the
theorem hold for the restricted problem.  Under alternative~(2), this
requires the corresponding forced-trajectory compactness assumption with
initial time $t_0$ and initial value $v_0$.  The other four alternatives are
stable under restriction to a later initial interval.  In the compactness
alternatives the theorem gives ordered existence, not a selection rule:
there may be several ordered companions, even for the same larger initial value.
Moreover, further solutions may exist that are not ordered.
\end{remark}

A symmetric theory is obtained when the tube is closed under infima:
\begin{equation}\label{K-converse}
 x,y\in K_A(t)\quad\Rightarrow\quad x\wedge y\in K_A(t)
 \qquad\mbox{for all }t\in J.
\end{equation}

\begin{corollary}[Lower companions and the infimum variant]
\label{comparison-infimum}
Assume \eqref{K-converse} in place of \eqref{K-cond}.  Then the conclusions
of Lemma~\ref{upper-tube-inheritance} and
Theorem~\ref{upper-solution-existence} hold with the lower tube
$K(t)\cap\bigl(u(t)-X_+\bigr)$ and with a prescribed initial value $v_0\leq u(0)$.  The final uniqueness and comparison assertion of
Theorem~\ref{upper-solution-existence} remains valid.  Likewise,
Theorem~\ref{direct-comparison} remains valid with the infimum condition in
place of the supremum condition.
\end{corollary}

\begin{proof}
Apply the preceding results after the isometric reflection
$Rx=-x$: set
\[
 \widetilde A=RAR^{-1},\qquad
 \widetilde K(t)=-K(t),\qquad
 \widetilde f(t,x)=-f(t,-x).
\]
All accretivity, Carath\'eodory/measurability, growth, quasimonotonicity,
and subtangential hypotheses are preserved, as are the corresponding
compactness and Lipschitz alternatives, while infima and lower tubes in
the original variables become suprema and upper tubes. The direct comparison statement transforms in the same way.
\end{proof}

\begin{remark}
(a) In the unconstrained dense-domain case $K(t)=X$ and
$\overline{D(A)}=X$, the direct argument also permits two right-hand sides
$f$ and $g$.  More precisely, if the upper right-hand side
$g(t,\cdot)$ has the local state-Lipschitz regularity used in
Theorem~\ref{direct-comparison}, it suffices to replace ordinary
quasimonotonicity by the cross condition
\[
 \lim_{h\to0+}\frac1h
 d\bigl(y-x+h(g(t,y)-f(t,x)),X_+\bigr)=0
 \qquad\mbox{for almost every $t$ and all }x\leq y;
\]
compare \cite{Martin}.  Indeed, in the proof of
Lemma~\ref{quasimonotone-bracket} one takes
$p=f(t,x)-g(t,x\vee y)$ and controls
$q=g(t,x\vee y)-g(t,y)$ by the local Lipschitz bound.  In particular, this
covers right-hand sides of the form $F(t,x)+f_0(t)$ and
$F(t,x)+g_0(t)$ when $F(t,\cdot)$ satisfies the hypotheses of
Theorem~\ref{direct-comparison} and
$f_0,g_0\in L^1(J;X)$ satisfy $f_0(t)\leq g_0(t)$ almost everywhere.

(b) In the autonomous case with a static state space, unique solution maps
form a monotone semiflow.  For a time-dependent right-hand side or tube, the
proper object is an order-preserving evolution process.  Without uniqueness,
Theorem~\ref{upper-solution-existence} gives the corresponding one-sided
existence property rather than a single-valued process.

(c) Competitive sign patterns can be treated by changing the ordering cone.
On a product lattice this amounts to reversing selected coordinates and
applying the preceding results to the conjugated operator and vector field,
provided the conjugated operator is $m$-$T$-accretive.

(d) In the static setting recalled in
Remark~\ref{static-lipschitz-general-space}, existence, uniqueness, and
continuous dependence are already available.  If the static constraint is
lattice compatible and $f$ is quasimonotone, Theorem~\ref{direct-comparison}
gives the same comparison conclusion in an arbitrary Banach lattice.
\end{remark}

\section{Lyapunov pairs}\label{s:lyapunov}
Nonautonomous Lyapunov pairs have a simple viability interpretation in an
extended state space.  Formally, adjoin to a solution $u$ a scalar variable
$r$ with $r'(t)=-W(t,u(t))$.  Then
$r(t)=r(t_0)-\int_{t_0}^tW(s,u(s))\,ds$, so that, for
$r(t_0)=V(t_0,x_0)$, the Lyapunov inequality is exactly $V(t,u(t))\leq r(t)$:
the extended trajectory $(u,r)$ remains in the moving epigraph of $V$.
This is the guiding idea below.  It is naturally implemented in $X\times\mathbb R$
with the product operator $\mathcal A=A\times0$.  Since $W$ may be
extended-valued and merely lower semicontinuous, it is approximated from below
by finite-valued Lipschitz functions $W_n$ and the viability theory is applied
to $(f,-W_n)$.  Proposition~\ref{epi-derivative-identity} converts regular-time
epigraph subtangentiality into an $A$-contingent derivative inequality for
$V$.  The resulting epigraph-viability argument separates regular-time
subtangentiality from the additional bound required to control the motion of
the epigraph at exceptional times.

Unlike Sections~\ref{s:periodic} and~\ref{s:comparison}, the necessity
argument needs Proposition~\ref{necessary-regular-tangency} with one
state-independent exceptional null set, hence a separable setting.  We
therefore fix a bounded separable reduced problem from
Theorem~\ref{separable-reduction-theorem} and relabel its state space,
operator, tube, and forcing as $X$, $A$, $K$, and $f$.  Assume that $f$ is
locally Lipschitz.  On every restriction $[t_0,a]$, \emph{(SR1)} holds and
Theorem~\ref{ex4}(a) gives a unique viable solution from each
$(t_0,x_0)\in\Gr(K_A)$.  Fix the finite-valued bound $c\in L^1_+(J)$ supplied
by Theorem~\ref{separable-reduction-theorem}, so that $\|f(t,x)\|\leq c(t)$
on $\Gr(K_A)$. We write $u(\cdot;t_0,x_0)$ for the corresponding solution.

Let $V:\Gr(K_A)\to(-\infty,+\infty]$ and
$W:\Gr(K_A)\to[0,+\infty]$. For $t\in J$, put
$\operatorname{dom}V(t):=\{x\in K_A(t):V(t,x)<+\infty\}$ and assume that
this set is nonempty.

\begin{definition}\label{lyapunov-pair-definition}
The pair $(V,W)$ is a \emph{Lyapunov pair} for \eqref{sivp4} on the tube
$K$ if, for every $t_0\in[0,a)$ and every
$x_0\in\operatorname{dom}V(t_0)$, the map
$W(\cdot,u(\cdot;t_0,x_0))$ is measurable and
\begin{equation}\label{lyapunov-inequality}
 V(t,u(t;t_0,x_0))+
 \int_{t_0}^{t}W(s,u(s;t_0,x_0))\,ds
 \leq V(t_0,x_0)
 \qquad\text{for all }t\in[t_0,a].
\end{equation}
For $W=0$ this is the usual notion of a Lyapunov function.
\end{definition}

The quantification over all starting times is essential in the
nonautonomous setting; uniqueness removes ambiguity in the viable solution.
All objects below refer to this fixed reduced problem.
Now extend $V$ to $J\times X$ by
\[
 V_K(t,x):=
 \begin{cases}
  V(t,x),&x\in K_A(t),\\
  +\infty,&x\notin K_A(t).
 \end{cases}
\]
We assume that $V_K$ is $\mathcal L(J)\otimes\mathcal B(X)$-measurable
and sequentially lower semicontinuous from the left on the effective graph:
\begin{equation}\label{left-lsc-V}
 \begin{array}{c}
  t_n\leq t,\;\; t_n\to t,\;\;
  x_n\in K_A(t_n),\;\; x_n\to x
 \end{array}
 \; \text{ implies }\;
 V(t,x)\leq\liminf_{n\to\infty}V(t_n,x_n).
\end{equation}
In particular, every section $V(t,\cdot)$ is lower semicontinuous on
$K_A(t)$.  Define the epigraph tube in
$\mathcal X:=X\times\mathbb R$, equipped with the sum norm, by
\begin{equation}\label{lyapunov-epigraph}
 C(t):=\{(x,\mu)\in K_A(t)\times\mathbb R:
                   V(t,x)\leq\mu\}.
\end{equation}
Then measurability of $V_K$ gives
$\Gr(C)\in\mathcal L(J)\otimes\mathcal B(\mathcal X)$.  Condition~\eqref{left-lsc-V} and left-closedness of $\Gr(K_A)$ imply
that every $C(t)$ is closed and that $\Gr(C)$ is closed from the left.

On $\mathcal X$ consider the $m$-accretive operator
$\mathcal A(x,\mu):=Ax\times\{0\}$ with
$D(\mathcal A)=D(A)\times\mathbb R$.  For $(v,\lambda)\in\mathcal X$, the semigroup generated by
$-\mathcal A+(v,\lambda)$ is
\begin{equation}\label{product-translated-semigroup}
 \mathcal S_{(v,\lambda)}(h)(x,\mu)
   =(S_v(h)x,\mu+h\lambda).
\end{equation}

For $t<a$, $x\in\operatorname{dom}V(t)$, and $v\in X$, define
the time-dependent $A$-contingent derivative by
\begin{equation}\label{lyapunov-contingent-derivative}
 D_A^-V(t,x;v):=
 \liminf_{\substack{h\downarrow0\\z\to0}}
 \frac{V_K(t+h,S_v(h)x+hz)-V(t,x)}{h}.
\end{equation}
The extension $V_K$ makes the expression meaningful outside the tube.  The
correction $hz$ makes explicit the $o(h)$ state displacement already implicit
in the distance-based definition of $A$-subtangentiality; compare the
sequential formulation~\eqref{SC1-sequential}.  Without this correction, the
derivative would test only the exact translated semigroup trajectory
$S_v(h)x$.

\begin{proposition}[Epigraph identity]\label{epi-derivative-identity}
For every $t\in[0,a)$ and $x\in\operatorname{dom}V(t)$,
\begin{equation}\label{epi-derivative-formula}
 \operatorname{epi}\bigl(D_A^-V(t,x;\cdot)\bigr)
 =\bigcap_{\mu\geq V(t,x)}
   T_C^{\mathcal A}(t;(x,\mu))
 =T_C^{\mathcal A}(t;(x,V(t,x))).
\end{equation}
Here $T_C^{\mathcal A}$ is the $\mathcal A$-subtangential set of
Definition~\ref{A-subtangent-definition}, formed in the product space.
\end{proposition}

\begin{proof}
Suppose first that $(v,\lambda)\in T_C^{\mathcal A}(t;(x,V(t,x)))$.
Then there are $h_n\downarrow0$ and
$(x_n,\mu_n)\in C(t+h_n)$ such that
\[
 \|x_n-S_v(h_n)x\|+
 |\mu_n-V(t,x)-h_n\lambda|=o(h_n).
\]
Writing
$x_n=S_v(h_n)x+h_nz_n$ and
$\mu_n=V(t,x)+h_n(\lambda+\theta_n)$,
we have $z_n\to0$, $\theta_n\to0$, and
\[
 V(t+h_n,S_v(h_n)x+h_nz_n)
 \leq V(t,x)+h_n(\lambda+\theta_n).
\]
Consequently $D_A^-V(t,x;v)\leq\lambda$.

Conversely, suppose $D_A^-V(t,x;v)\leq\lambda$ and let
$\mu\geq V(t,x)$.  There are $h_n\downarrow0$, $z_n\to0$, and
$\theta_n\to0$ such that
\[
 V(t+h_n,S_v(h_n)x+h_nz_n)
 \leq V(t,x)+h_n(\lambda+\theta_n).
\]
Therefore
\[
 \bigl(S_v(h_n)x+h_nz_n,
       \mu+h_n(\lambda+\theta_n)\bigr)\in C(t+h_n),
\]
and \eqref{product-translated-semigroup} gives
\[
 \frac1{h_n}d_{\mathcal X}
 \bigl(\mathcal S_{(v,\lambda)}(h_n)(x,\mu),C(t+h_n)\bigr)
 \leq\|z_n\|+|\theta_n|\rightarrow0.
\]
Thus $(v,\lambda)\in T_C^{\mathcal A}(t;(x,\mu))$ for every
$\mu\geq V(t,x)$.  Since $\mu=V(t,x)$ occurs in the intersection, the latter
is contained in $T_C^{\mathcal A}(t;(x,V(t,x)))$; together with the first
part of the proof, this yields both equalities in
\eqref{epi-derivative-formula}.
\end{proof}

Thus, whenever $W(t,x)<+\infty$, the inequality
$D_A^-V(t,x;f(t,x))\leq-W(t,x)$ is equivalent to
\[
 (f(t,x),-W(t,x))\in T_C^{\mathcal A}(t;(x,V(t,x))),
\]
and, by \eqref{epi-derivative-formula}, to the corresponding membership in
$T_C^{\mathcal A}(t;(x,\mu))$ for every $\mu\geq V(t,x)$.  Hence the
derivative inequality used below is precisely the regular-time
$\mathcal A$-subtangentiality condition for the epigraph tube, expressed in
terms of $V$.  If $W(t,x)=+\infty$, the derivative inequality is instead
understood in the extended-real sense, as stated below.

For the lower-semicontinuous dissipation term, extend $W$ to $J\times X$ by
\begin{equation}\label{extended-W}
 W_K(t,x):=
 \begin{cases}
  W(t,x),&x\in K_A(t),\\
  +\infty,&x\notin K_A(t).
 \end{cases}
\end{equation}
Assume that $W_K$ is a nonnegative normal integrand, that is,
$W_K:J\times X\to[0,+\infty]$ is
$\mathcal L(J)\otimes\mathcal B(X)$-measurable and
$W_K(t,\cdot)$ is lower semicontinuous for every $t$.

\begin{lemma}[Lipschitz regularization of the dissipation]
\label{dissipation-approximation}
Let $X$ be separable.  For $n\geq1$ set
\begin{equation}\label{Wn-definition}
 W_n(t,x):=\inf_{y\in X}
 \bigl\{\min\{W_K(t,y),n\}+n\|x-y\|\bigr\}.
\end{equation}
Then $W_n:J\times X\to[0,n]$ is jointly measurable and Carath\'eodory,
$W_n(t,\cdot)$ is globally Lipschitz with constant $n$, and
\begin{equation}\label{Wn-convergence}
 0\leq W_n\leq W_{n+1}\leq W_K,
 \qquad
 W_n(t,x)\uparrow W_K(t,x)
 \quad\mbox{for all }(t,x)\in J\times X.
\end{equation}
\end{lemma}

\begin{proof}
The measurable infimum theorem for normal integrands shows that $W_n$ is
product measurable; see, for example, \cite{CaVa}.  The triangle inequality
gives $|W_n(t,x)-W_n(t,x')|\leq n\|x-x'\|$, so $W_n$ is jointly measurable
and Carath\'eodory.  Taking $y=x$ proves $W_n\leq n$ and $W_n\leq W_K$.  For each
fixed $y$, both terms inside the braces in \eqref{Wn-definition} are
nondecreasing in $n$; taking infima therefore gives $W_n\leq W_{n+1}$.

Fix $(t,x)$ and let $\ell=\lim_nW_n(t,x)$.  If $\ell<+\infty$, choose
$y_n$ so that $\min\{W_K(t,y_n),n\}+n\|x-y_n\|
\leq W_n(t,x)+n^{-1}$.  Then $y_n\to x$.  For all sufficiently large $n$ the latter
minimum cannot equal $n$, and hence $W_K(t,y_n)+n\|x-y_n\|\leq W_n(t,x)+n^{-1}$.
Lower semicontinuity yields
$W_K(t,x)\leq\liminf_n W_K(t,y_n)\leq\ell$.
Together with $W_n\leq W_K$, this proves
$\ell=W_K(t,x)$.  If $W_K(t,x)=+\infty$, the same
argument rules out a finite value of $\ell$.
\end{proof}

\begin{theorem}[Regular-time necessity and epigraph sufficiency]
\label{lyapunov-characterization}
In the fixed reduced subspace described above, assume that $V_K$ is
product measurable, that $V$ satisfies
\eqref{left-lsc-V}, and that the extension $W_K$ in
\eqref{extended-W} is a nonnegative normal integrand.
\begin{enumerate}
\item[(i)] If $(V,W)$ is a Lyapunov pair, then there is a null set
$N_V\subset[0,a)$ such that
\begin{equation}\label{lyapunov-necessary-derivative}
 D_A^-V(t,x;f(t,x))\leq-W(t,x)
\end{equation}
for every $t\in[0,a)\setminus N_V$ and every
$x\in\operatorname{dom}V(t)$.
The inequality is understood in the extended-real sense; in particular,
$W(t,x)=+\infty$ requires $D_A^-V(t,x;f(t,x))=-\infty$.

\item[(ii)] Conversely, suppose there are a null set
$N\subset[0,a)$ and a finite-valued $q\in L^1_+(J)$ such that
\eqref{lyapunov-necessary-derivative} holds for every
$t\in[0,a)\setminus N$ and every $x\in\operatorname{dom}V(t)$.  Assume that the epigraph
tube satisfies
\begin{equation}\label{epigraph-exceptional-condition}
 \liminf_{h\downarrow0}\frac1h
 \left(
 d_{\mathcal X}\bigl((S(h)x,\mu),C(t+h)\bigr)
 -\int_t^{t+h}q(s)\,ds
 \right)^+=0
\end{equation}
for every $t\in N$ and every $(x,\mu)\in C(t)$.  It is sufficient, in
particular, to impose the stronger pointwise estimate
\begin{equation}\label{strong-epigraph-exceptional-condition}
 \liminf_{h\downarrow0}\frac1h
 d_{\mathcal X}\bigl((S(h)x,\mu),C(t+h)\bigr)
 \leq q(t)
 \qquad\mbox{for all }t\in N,\, (x,\mu)\in C(t),
\end{equation}
because Proposition~\ref{c-hat} converts it into
\eqref{epigraph-exceptional-condition} with
$\widehat q^{\,\circ}$ in place of $q$.  

Then $(V,W)$ is a Lyapunov pair.
\end{enumerate}
\end{theorem}

\begin{proof}
For part (i), fix $n\geq1$ and define on $\Gr(C)$
\begin{equation}\label{product-forcing-n}
 G_n(t,x,\mu):=(f(t,x),-W_n(t,x)).
\end{equation}
Joint measurability of $f$ and $W_n$, together with product measurability of
$\Gr(C)$, makes $G_n$ jointly measurable on $\Gr(C)$; it is also
Carath\'eodory.  On every state ball on which $f$ has
local Lipschitz coefficient $\omega\in L^1_+(J)$, Lemma~\ref{dissipation-approximation}
gives
\[
 \|G_n(t,x,\mu)-G_n(t,y,\nu)\|_{\mathcal X}
 \leq(\omega(t)+n)\|x-y\|
 \leq(\omega(t)+n)\|(x,\mu)-(y,\nu)\|_{\mathcal X},
\]
so $G_n$ is locally Lipschitz in $(x,\mu)$.  Moreover,
\[
 \|G_n(t,x,\mu)\|_{\mathcal X}\leq c(t)+n
 \qquad\text{for all }(t,x,\mu)\in\Gr(C),
\]
so it satisfies the required integrable growth bound.  Since
$\overline{D(\mathcal A)}=\overline{D(A)}\times\mathbb R$ and
$C(t)\subset\overline{D(A)}\times\mathbb R$, 
the effective
constraint associated with $C$ and $\mathcal A$ is simply
$C_{\mathcal A}(t)=C(t)$.

Let $t_0<a$ and $(x_0,\mu_0)\in C(t_0)$.  By the Lyapunov inequality,
the unique viable solution $u=u(\cdot;t_0,x_0)$ satisfies
\[
 V(t,u(t))+\int_{t_0}^tW_n(s,u(s))\,ds
 \leq V(t_0,x_0)\leq\mu_0.\vspace{-0.05in}
\]
Thus\vspace{-0.05in}
\begin{equation}\label{product-trajectory-n}
 z_n(t):=
 \Big( u(t),\mu_0-\int_{t_0}^tW_n(s,u(s))\,ds \Big)
\end{equation}
remains in $C(t)$ and is the mild solution of
$z_n'+\mathcal A z_n\ni G_n(t,z_n)$.  Consequently $C$ is viable for this product problem.  Apply
Proposition~\ref{necessary-regular-tangency} in the separable space
$\mathcal X$.  There is a null set $N_n$, independent of the initial
state, such that
\[
 (f(t,x),-W_n(t,x))
 \in T_C^{\mathcal A}(t;(x,V(t,x)))
\]
for all $t\in[0,a)\setminus N_n$ and
$x\in\operatorname{dom}V(t)$.  Hence $D_A^-V(t,x;f(t,x))\leq-W_n(t,x)$
by Proposition~\ref{epi-derivative-identity}.
Take $N_V=\bigcup_{n\geq1}N_n$ and let $n\to\infty$ in view of
\eqref{Wn-convergence}.  This proves (i).

For part (ii), fix $n$.  Since $W_n\leq W$,
\eqref{lyapunov-necessary-derivative} and
Proposition~\ref{epi-derivative-identity} imply
\begin{equation}\label{product-tangency-n}
 G_n(t,x,\mu)
 \in T_C^{\mathcal A}(t;(x,\mu))
 \qquad
 \mbox{for all }t\in [0,a)\setminus N\mbox{ and }(x,\mu)\in C(t).
\end{equation}
Condition \eqref{epigraph-exceptional-condition} is precisely the
integral exceptional-set condition for the unforced semigroup of
$-\mathcal A$.  Under \eqref{strong-epigraph-exceptional-condition}, replace
$q$ by $\widehat q^{\,\circ}$ as in the statement and denote the resulting
bound again by $q$.  Then $\gamma_n:=c+n+q$ is a common growth and
exceptional bound.  Since $\Gr(C)$ is product measurable and $G_n$ is a
jointly measurable Carath\'eodory map, these estimates together with
\eqref{product-tangency-n} verify the hypotheses of
Corollary~\ref{general-existence-theory} for the product problem.

Let $t_0<a$ and $x_0\in\operatorname{dom}V(t_0)$.  The same properties hold
after restriction to any $[s,a]$, $t_0\leq s<a$.  Together with the local
Lipschitz property of $G_n$, Corollary~\ref{general-existence-theory}(c)
therefore gives a unique viable product solution starting from
$(x_0,V(t_0,x_0))$.  Its components are $u(\cdot;t_0,x_0)$ and
$r_n(t)=V(t_0,x_0)-\int_{t_0}^tW_n(s,u(s;t_0,x_0))\,ds$.
Viability in $C$ yields
\begin{equation}\label{lyapunov-Wn}
 V(t,u(t;t_0,x_0))+
 \int_{t_0}^tW_n(s,u(s;t_0,x_0))\,ds
 \leq V(t_0,x_0).
\end{equation}
Because $W_K$ is product measurable and $u$ is continuous, the composition
$W(\cdot,u(\cdot))$ is measurable.  Since $W_n\uparrow W$ and all functions
are nonnegative, the monotone convergence theorem permits passage to the limit
in \eqref{lyapunov-Wn}.  The result is \eqref{lyapunov-inequality}.
\end{proof}

\begin{remark}[Relation to autonomous Lyapunov characterizations]
Lyapunov methods for nonlinear contraction semigroups in Banach spaces
go back to Pazy \cite{Pazy81}; for continuous convex functionals, his
directional-derivative and resolvent criteria, together with the systematic
treatment of Lyapunov couples in \cite[Chapter~19]{BCP}, provide classical
autonomous criteria; see also \cite{Barbu2010}.

In Hilbert space, Kocan and Soravia \cite{KoSo} obtained a characterization
for autonomous maximal-monotone systems using viscosity-solution methods for
associated nonlinear, unbounded Hamilton--Jacobi inequalities.  Later,
C\^arj\u{a} and Motreanu \cite{CaMot} established, for possibly multivalued
$m$-accretive evolutions in arbitrary Banach spaces, an autonomous
characterization of Lyapunov pairs by means of an operator-adapted contingent
derivative, using tangency and flow-invariance arguments.
\end{remark}

In the autonomous case, the exceptional-time condition disappears.
Theorem~\ref{lyapunov-characterization} thus yields the following characterization.

\begin{corollary}[Autonomous case]\label{autonomous-lyapunov-corollary}
Assume the hypotheses stated before parts~(i) and~(ii) of
Theorem~\ref{lyapunov-characterization}, with
$K$, $V$, $W$, and $f$ all time-independent.
Then no exceptional condition is needed, and
\[
 (V,W)\text{ is a Lyapunov pair}
 \quad\Longleftrightarrow\quad
 D_A^-V(x;f(x))\leq-W(x)
 \quad\text{for all }x\in\operatorname{dom}V.
\]
\end{corollary}

Indeed, the inequality furnished by the necessity part is independent of
time, so its validity at every regular time implies the displayed autonomous
condition; conversely, one applies
Theorem~\ref{lyapunov-characterization}(ii) with an empty exceptional set.

\begin{remark}[Lyapunov sequences and iterated dissipation]
Following \cite[Chapter~19]{BCP}, let
$m\geq1$ and let $V_0,V_1,\ldots,V_m:\Gr(K_A)\to[0,+\infty]$
be such that $(V_{j-1},V_j)$ is a Lyapunov pair for every
$j=1,\ldots,m$.  Then $(V_0,\ldots,V_m)$ is called a Lyapunov sequence.  If
$u=u(\cdot;t_0,x_0)$ and $V_0(t_0,x_0)<+\infty$, retaining the nonnegative
endpoint terms in the usual induction gives, for every $t\in(t_0,a]$,
\begin{equation}\label{lyapunov-sequence-combined}
 V_0(t,u(t))
 +\sum_{j=1}^{m-1}\frac{(t-t_0)^j}{j!}V_j(t,u(t)) +\int_{t_0}^t\frac{(s-t_0)^{m-1}}{(m-1)!}V_m(s,u(s))\,ds
 \leq V_0(t_0,x_0).
\end{equation}
For the extended-valued case, put $\mathcal V_j(s):=V_j(s,u(s))$,
$q_j(s):=(s-t_0)^{j-1}/(j-1)!$, and
$L_j(t):=\int_{t_0}^t q_j(s)\mathcal V_j(s)\,ds$.  The pair $(V_0,V_1)$ gives
$V_0(t,u(t))+L_1(t)\leq V_0(t_0,x_0)$.  If $L_j(t)<+\infty$, then
$\mathcal V_j(r)<+\infty$ for almost every $r\in(t_0,t)$; for such $r$,
uniqueness and the pair $(V_j,V_{j+1})$ give
$\mathcal V_j(t)+\int_r^t\mathcal V_{j+1}(s)\,ds\leq\mathcal V_j(r)$.
Multiplying by $q_j(r)$, integrating, and using Tonelli's theorem yields
$q_{j+1}(t)\mathcal V_j(t)+L_{j+1}(t)\leq L_j(t)$; admissible restart times
$r\downarrow t_0$ also give the required measurability.  Iteration for
$j=1,\ldots,m-1$ proves \eqref{lyapunov-sequence-combined}.  Dropping
nonnegative terms recovers the bounds
$V_j(t,u(t))\leq j!(t-t_0)^{-j}V_0(t_0,x_0)$ for $j<m$ and the weighted
integrability of $V_m$ from \cite[Chapter~19, Proposition~19.6]{BCP}.

Theorem~\ref{lyapunov-characterization} may therefore be applied to each
consecutive pair $(V_{j-1},V_j)$, provided its hypotheses for the direction
under consideration are satisfied for that pair. In particular, the
regular-time derivative inequalities of Theorem~\ref{lyapunov-characterization}(i)
are necessary for each pair, while the corresponding inequalities together
with the exceptional-time epigraph condition in
Theorem~\ref{lyapunov-characterization}(ii) are sufficient.

Besides regularizing effects, \eqref{lyapunov-sequence-combined} may also be
applied to the product evolution of two solutions. An additional coercive or
observability estimate for the weighted $V_m$-term may then provide a route to
proving that an iterate of a Poincar\'e operator is a strict contraction. This
suggests a possible compactness-free periodicity mechanism complementary to Section~\ref{s:periodic}, which we
leave for future work.
\end{remark}

\section{Abstract reaction--diffusion systems}\label{s:react}
The preceding viability theory applies to systems in which the diffusion acts
componentwise and the reaction couples the components.  This formulation
permits nonlinear and multivalued diffusion operators.  The assumptions below
cover, for example, nonlinear diffusion operators arising from
$-\Delta\phi_i(u_i)$ and from the $p_i$-Laplacian, with boundary conditions
for which the corresponding resolvents are order preserving.

Throughout this section $(\Omega,\Sigma,\mu)$ is a finite measure space
such that $L^p(\Omega)$ is separable (as is the case, for example, for a
bounded Lebesgue domain), $m\geq1$, $1\leq p<\infty$, and $T>0$.  We set
$X:=\prod_{i=1}^mL^p(\Omega)$ with
$\|u\|:=(\sum_{i=1}^m\|u_i\|_{L^p}^p)^{1/p}$.  Equivalently, $X=L^p(\Omega;\ell_m^p)$, where
$|z|_p=(\sum_i|z_i|^p)^{1/p}$ on $\mathbb R^m$.  We use the componentwise
order.  With this choice the product of scalar contractions is again a
contraction, and $X$ is a separable Banach lattice.  For
$i=1,\ldots,m$, let $A_i$ be an $m$-$T$-accretive operator in
$L^p(\Omega)$ with dense domain, and define the product operator
\begin{equation}\label{RD-product-operator}
 D(A):=\prod_{i=1}^mD(A_i),\qquad
 Au:=A_1u_1\times\dots\times A_m u_m.
\end{equation}

The product resolvent
$J_\lambda u=(J_\lambda^1u_1,\ldots,J_\lambda^m u_m)$ is everywhere defined,
order preserving, and contractive in the product norm.  Hence $A$ is
$m$-$T$-accretive in $X$; its semigroup acts componentwise by the exponential
formula, and $\overline{D(A)}=X$ by density of the coordinate domains.

We impose the following preservation of constant order intervals:
\begin{equation}\label{RD-constant-supersolutions}
 0\in A_i0,
 \qquad
 J^{i}_{\lambda}\alpha\leq\alpha
 \quad\mbox{for all }\lambda>0\mbox{ and }\alpha\geq0,
 \qquad\text{for all }i=1,\ldots,m,
\end{equation}
where a scalar is identified with the corresponding constant function on
$\Omega$.  Since every $J^i_\lambda$ is order preserving,
\eqref{RD-constant-supersolutions} implies
\begin{equation}\label{RD-static-interval-invariance}
 J_\lambda\{u\in X:0\leq u\leq b\}
 \subset\{u\in X:0\leq u\leq b\}
 \qquad\mbox{for all }\lambda>0\mbox{ and }b\in\mathbb R_+^m.
\end{equation}
The same inclusion holds with $S(h)$ in place of $J_\lambda$ for every
$h\geq0$ and every $b\in\mathbb R_+^m$.

The reaction rate is a map $g:J\times\mathbb R_+^m\to\mathbb R^m$,
$J=[0,T]$, satisfying:
\begin{enumerate}
\item[(R1)] $g$ is Carath\'eodory;
\item[(R2)] for every $R>0$ there is
$c_R\in L^1_+(J)$ such that
\begin{equation}\label{RD-local-integrable-bound}
 |g(t,z)|_p\leq c_R(t)
 \quad\text{for a.e. }t\in J\text{ and all }z\in[0,R]^m;
\end{equation}
\item[(R3)] $g$ is quasipositive, that is,\vspace{-0.1in}
\begin{equation}\label{RD-quasipositivity}
 \begin{gathered}
 g_i(t,z)\geq0\\[2pt]
 \text{for almost every }t\in J,
 \text{ every }z\in\mathbb R_+^m\text{ with }z_i=0,
 \text{ and every }i=1,\ldots,m.
 \end{gathered}
\end{equation}
\end{enumerate}
Changing $g$ on one null set of times does not alter the induced evolution
problem.  Apply \emph{(R1)}--\emph{(R3)} first on the integer boxes
$[0,n]^m$, $n\in\mathbb N$, and take the union of the resulting null sets.
After setting $g=0$ on that union and choosing finite representatives of the
majorants, we may assume that $g(t,\cdot)$ is continuous,
\eqref{RD-quasipositivity} holds, and
\eqref{RD-local-integrable-bound} holds for every $t\in J$ and every
integer $R\geq1$.  For arbitrary $R>0$, we henceforth take
$c_R:=c_{\max\{1,\lceil R\rceil\}}$; then
\eqref{RD-local-integrable-bound} holds for every $t\in J$ and every
$R>0$.
The associated Nemytskii map is
\begin{equation}\label{RD-Nemytskii}
 f(t,u)(\xi):=g(t,u(\xi))
 \qquad\text{for a.e. }\xi\in\Omega.
\end{equation}
\subsection{Nemytskii maps on moving rectangles}

Let $y:J\to\mathbb R_+^m$ be continuous and define
\begin{equation}\label{RD-tube}
 K^y(t):=
 \{u\in X:0\leq u_i\leq y_i(t)
 \text{ a.e.\ on }\Omega,\ i=1,\ldots,m\}.
\end{equation}
All inequalities in this section are componentwise.
Finite-dimensional distances and norms use $|\cdot|_p$, except where
$|\cdot|_\infty$ is displayed explicitly.

\begin{lemma}[Nemytskii regularity]\label{RD-Nemytskii-lemma}
Assume \emph{(R1)}--\emph{(R2)} and let $y$ be continuous.  Let $R\geq1$
be such that $y(t)\in[0,R]^m$ for every $t\in J$.  Then $K^y(t)$ is
nonempty, closed, bounded, and convex for every $t$, and $\Gr(K^y)$ is
closed in $J\times X$.  The map $f$ in \eqref{RD-Nemytskii}, restricted
to $\Gr(K^y)$, is a jointly measurable Carath\'eodory map in the sense
of Section~\ref{s:prel}, and
\begin{equation}\label{RD-Nemytskii-bound}
 \|f(t,u)\|\leq\mu(\Omega)^{1/p}c_R(t)
 \quad\mbox{for every }t\in J
 \mbox{ and all }u\in K^y(t).
\end{equation}

If, in addition, for every $\rho>0$ there is
$\ell_\rho\in L^1_+(J)$ such that
\begin{equation}\label{RD-finite-Lipschitz}
 |g(t,z)-g(t,\widetilde z)|_p
 \leq\ell_\rho(t)|z-\widetilde z|_p
 \quad\text{for a.e. }t\in J
 \text{ and all }z,\widetilde z\in[0,\rho]^m,\vspace{-0.05in}
\end{equation}
then\vspace{-0.05in}
\begin{equation}\label{RD-Nemytskii-Lipschitz}
 \|f(t,u)-f(t,v)\|
 \leq\ell_R(t)\|u-v\|
 \quad\text{for a.e. }t\in J
 \text{ and all }u,v\in K^y(t).
\end{equation}
\end{lemma}

\begin{proof}
The sections are closed bounded order intervals, and closedness of the graph
follows by taking an almost-everywhere convergent subsequence from
$t_n\to t$, $u_n\to u$ and using continuity of $y$.  
Let $P_R$ be the componentwise clipping onto $[0,R]^m$,
and put $g_R(t,z):=g(t,P_Rz)$.  For fixed $t$, the induced Nemytskii map is
continuous on $X$: from any sequence converging in $X$, extract an almost
everywhere convergent subsequence and use continuity of $g_R(t,\cdot)$ and
bounded convergence; the subsequence principle gives convergence of the full
sequence.  For a simple function $u$, the map
$t\mapsto g_R(t,u(\cdot))$ is strongly measurable.  Approximation of a
general $u$ by simple functions and the preceding pointwise
continuity therefore give strong measurability for every $u\in X$.
Since the preliminary null-set modification makes state continuity hold for
every $t$, the last assertion of Proposition~\ref{cross-meas} yields joint
measurability on the cylinder, and the restriction agrees with $f$ on
$\Gr(K^y)$.
  Integrating \eqref{RD-local-integrable-bound} and, when
assumed, \eqref{RD-finite-Lipschitz} gives
\eqref{RD-Nemytskii-bound} and \eqref{RD-Nemytskii-Lipschitz}.
\end{proof}

The next lemma transfers finite-dimensional quasimonotonicity to $X$.

\begin{lemma}[Lifting quasimonotonicity]\label{RD-quasimonotone-lifting}
Suppose that, for almost every $t$, $g(t,\cdot)$ is quasimonotone on
$\mathbb R_+^m$ in the Kamke sense:
\begin{equation}\label{RD-Kamke}
 z\leq\widetilde z,\ z_i=\widetilde z_i
 \quad\Rightarrow\quad
 g_i(t,z)\leq g_i(t,\widetilde z)
 \quad\text{for all }z,\widetilde z\in\mathbb R_+^m
 \text{ and }i=1,\ldots,m.
\end{equation}
Then, for almost every $t\in J$, the map $f(t,\cdot)$ is quasimonotone
on $K^y(t)$ in the sense of \eqref{quasimonotonicity-definition}.
\end{lemma}

\begin{proof}
Let $u\leq v$ in $K^y(t)$.  For almost every $\xi$, the vector
$g(t,v(\xi))-g(t,u(\xi))$ belongs to the contingent cone of
$\mathbb R_+^m$ at $v(\xi)-u(\xi)$ by \eqref{RD-Kamke}.  Hence
\[
 \frac1h d_{\mathbb R^m}\bigl(
 v(\xi)-u(\xi)+h(g(t,v(\xi))-g(t,u(\xi))),
 \mathbb R_+^m\bigr)\rightarrow0.
\]
The quotients are dominated by
$|g(t,v(\xi))-g(t,u(\xi))|_p$, which belongs to $L^p(\Omega)$ on the
bounded rectangle.  Dominated convergence and the pointwise formula for
the distance to the positive cone prove
\eqref{quasimonotonicity-definition} in $X$.
\end{proof}

\subsection{Finite-dimensional upper envelopes}

The reaction field need not be quasimonotone.  To construct a rectangular
upper bound, define for $y\in\mathbb R_+^m$
\begin{equation}\label{RD-upper-envelope}
 \widehat g_i(t,y):=
 \max\{g_i(t,z):0\leq z\leq y,\ z_i=y_i\},
 \qquad\text{for all }i=1,\ldots,m.
\end{equation}
Thus each component is maximized over the corresponding upper face of the
order box.

\begin{lemma}[Properties of the upper envelope]\label{RD-envelope-lemma}
Under \emph{(R1)}--\emph{(R3)}, the map
$\widehat g:J\times\mathbb R_+^m\to\mathbb R^m$ is Carath\'eodory,
locally integrably bounded, quasipositive, and quasimonotone.  More
precisely,
\begin{equation}\label{RD-envelope-bound}
 |\widehat g_i(t,y)|\leq c_R(t)
 \quad\text{for all }t\in J,\ y\in[0,R]^m,
 \text{ and }i=1,\ldots,m,
\end{equation}
and consequently $|\widehat g(t,y)|_p\leq m^{1/p}c_R(t)$ for all
$t\in J$ and $y\in[0,R]^m$.  Moreover,
\begin{equation}\label{RD-envelope-quasimonotone}
 y\leq\widetilde y,\ y_i=\widetilde y_i
 \quad\Rightarrow\quad
 \widehat g_i(t,y)\leq\widehat g_i(t,\widetilde y)
 \quad\text{for all }t\in J,\, y,\widetilde y\in\mathbb R_+^m,\, i=1,\ldots,m.
\end{equation}
\end{lemma}

\begin{proof}
Fix $i$.  We first verify the Carath\'eodory property directly from
\eqref{RD-upper-envelope}.  Fix $t\in J$ and let $y_n\to y$ in
$\mathbb R_+^m$.  Choose a maximizer $z_n$ in
\eqref{RD-upper-envelope} for $y_n$.  The sequence $(z_n)$ is bounded.
Along any subsequence realizing the upper limit, pass to a further
subsequence with $z_n\to z$.  Then $0\leq z\leq y$ and $z_i=y_i$, so
continuity of $g_i(t,\cdot)$ gives
\[
 \limsup_{n\to\infty}\widehat g_i(t,y_n)
 =g_i(t,z)\leq\widehat g_i(t,y).
\]
Conversely, let $z$ be a maximizer for $y$ and define
$(z_n)_i=(y_n)_i$ and
$(z_n)_j=\min\{z_j,(y_n)_j\}$ for $j\neq i$.  Then
$z_n\to z$, $0\leq z_n\leq y_n$, and $(z_n)_i=(y_n)_i$.  Hence
\[
 \liminf_{n\to\infty}\widehat g_i(t,y_n)
 \geq\lim_{n\to\infty}g_i(t,z_n)
 =g_i(t,z)=\widehat g_i(t,y).
\]
Thus $\widehat g_i(t,\cdot)$ is continuous.  For fixed $y$, choose a
countable dense subset $(z_k)$ of the compact upper face in
\eqref{RD-upper-envelope}.  Then
$\widehat g_i(t,y)=\sup_k g_i(t,z_k)$, so
$t\mapsto\widehat g_i(t,y)$ is measurable.  This proves the
Carath\'eodory property.

Estimate \eqref{RD-envelope-bound} follows from
\eqref{RD-local-integrable-bound}, since the maximizing set is contained
in $[0,R]^m$.  If $y_i=0$, every vector in the maximizing set has $i$th
component zero, so quasipositivity follows from
\eqref{RD-quasipositivity}.  Finally, if $y\leq\widetilde y$ and
$y_i=\widetilde y_i$, the admissible set in
\eqref{RD-upper-envelope} for $y$ is contained in the admissible set for
$\widetilde y$, which proves \eqref{RD-envelope-quasimonotone}.
\end{proof}

The finite-dimensional Carath\'eodory existence and viability theorem,
applied to the closed cone $\mathbb R_+^m$, gives a local solution from every
$y^0\in\mathbb R_+^m$.  Quasipositivity is the Nagumo face condition for
that cone \cite{Nagumo}.  Thus one obtains a nonnegative local
solution of
\begin{equation}\label{RD-envelope-ODE}
 y'(t)=\widehat g(t,y(t)),\qquad y(0)=y^0.
\end{equation}
No uniqueness assumption is imposed on this auxiliary problem.  More
generally, one may work with any absolutely continuous upper function satisfying
\begin{equation}\label{RD-upper-solution}
 y_i'(t)\geq\widehat g_i(t,y(t))
 \quad\text{for a.e. }t\in J\text{ and all }i=1,\ldots,m,
 \quad y(t)\in\mathbb R_+^m\text{ for all }t\in J.
\end{equation}

\subsection{Subtangentiality of moving order intervals}

\begin{lemma}[Rectangle subtangentiality]\label{RD-rectangle-tangency}
Assume \emph{(R1)}--\emph{(R3)}, let $y\in W^{1,1}(J;\mathbb R_+^m)$
satisfy \eqref{RD-upper-solution}, and let $K^y$ be given by
\eqref{RD-tube}.  There is a null set $N\subset[0,T)$ such that
\begin{equation}\label{RD-ordinary-tangency}
 f(t,u)\in T_{K^y}(t,u)
 \qquad\mbox{for all }t\in[0,T)\setminus N\mbox{ and }u\in K^y(t).
\end{equation}
Moreover, for every $t\in[0,T)$, $u\in K^y(t)$, and
$0<h\leq T-t$,
\begin{equation}\label{RD-exceptional-motion}
 d(S(h)u,K^y(t+h))
 \leq\mu(\Omega)^{1/p}
       \int_t^{t+h}|y'(s)|\,ds.
\end{equation}
Consequently, after enlarging $N$ if necessary,
\begin{align}
 &f(t,u)\in T^A_{K^y}(t,u)
 &&\begin{gathered}
 \text{for all }t\in[0,T)\setminus N\\[-2pt]
 \text{and }u\in K^y(t),
 \end{gathered}
 \label{RD-A-tangency}\\
 &\liminf_{h\downarrow0}\frac1h
 \left(
 d(S(h)u,K^y(t+h))
 -\int_t^{t+h}\psi(s)\,ds
 \right)^+=0,
 &&\begin{gathered}
 \text{for all }t\in N\\[-2pt]
 \text{and }u\in K^y(t),
 \end{gathered}
 \label{RD-exceptional-condition}
\end{align}
where $\psi(t):=\mu(\Omega)^{1/p}|y'(t)|$.
\end{lemma}

\begin{proof}
Choose $N$ so that, outside $N$, $y$ is differentiable,
\eqref{RD-upper-solution} holds, and $g(t,\cdot)$ is continuous.  Fix
such a $t$, $u\in K^y(t)$, and put $q=f(t,u)$.  For almost every
$\xi\in\Omega$, set $z=u(\xi)$.  The lower faces of the rectangle are
controlled by quasipositivity: if $z_i=0$, then $q_i(\xi)\geq0$.  At an
upper face, $z_i=y_i(t)$ and $0\leq z\leq y(t)$, whence
\[
 q_i(\xi)=g_i(t,z)
 \leq\widehat g_i(t,y(t))
 \leq y_i'(t).
\]
These inequalities imply
\begin{equation}\label{RD-pointwise-tangency}
 h^{-1} d_{\mathbb R^m}
 \bigl(z+hq(\xi),[0,y(t+h)]\bigr)
 \rightarrow0
 \qquad(h\downarrow0).
\end{equation}
Indeed, a component lying strictly between its lower and upper bounds remains
inside the interval for small $h$, while at a lower or upper face the sign inequalities
and differentiability of $y$ give the assertion.

For small $h$, the quotients in \eqref{RD-pointwise-tangency} are bounded
by $|q(\xi)|+C_t$, where $C_t$ bounds the difference quotients of $y$ near
$t$.  The first
term belongs to $L^p(\Omega)$ by Lemma~\ref{RD-Nemytskii-lemma}.  Moreover,
$(|q(\xi)|+C_t)^p\leq2^{p-1}(|q(\xi)|^p+C_t^p)$ is integrable because
$\mu(\Omega)<\infty$.  Since
the pointwise metric projection onto a box is measurable,
\begin{equation}\label{RD-distance-formula}
 d(v,K^y(s))^p
 =\int_\Omega
 d_{\mathbb R^m}(v(\xi),[0,y(s)])^p\,d\mu(\xi).
\end{equation}
Dominated convergence therefore proves
\eqref{RD-ordinary-tangency}.

By \eqref{RD-static-interval-invariance}, $S(h)u\in K^y(t)$.  If
$v=S(h)u$, clip its components from above at $y_i(t+h)$.  The resulting
function belongs to $K^y(t+h)$ and
\begin{align*}
 d(v,K^y(t+h))
 &\leq\mu(\Omega)^{1/p}|(y(t)-y(t+h))^+|\\
 &\leq\mu(\Omega)^{1/p}|y(t+h)-y(t)| \leq\mu(\Omega)^{1/p}\int_t^{t+h}|y'(s)|\,ds.
\end{align*}
This proves \eqref{RD-exceptional-motion} and hence
\eqref{RD-exceptional-condition}.  Finally,
\eqref{RD-static-interval-invariance},
\eqref{RD-ordinary-tangency}, and Proposition~\ref{separated-tangency-proposition} give
\eqref{RD-A-tangency}.
\end{proof}

\subsection{Existence in invariant rectangles}

The first compactness alternative is imposed only on the bounded order
interval generated by the finite-dimensional upper solution, rather than on
the full state space.

\begin{theorem}[Abstract reaction--diffusion system]\label{RD-main-theorem}
Let $A_i$ be $m$-$T$-accretive in $L^p(\Omega)$ with dense domain and
satisfy \eqref{RD-constant-supersolutions}, and let $g$ satisfy
\emph{(R1)}--\emph{(R3)}.  Let
$u^0\in L^p(\Omega;\mathbb R^m)\cap L^\infty(\Omega;\mathbb R_+^m)$.
Suppose that $y\in W^{1,1}(J;\mathbb R_+^m)$ satisfies
\eqref{RD-upper-solution} and
\begin{equation}\label{RD-initial-upper-bound}
 u^0(\xi)\leq y(0)\qquad\text{for a.e. }\xi\in\Omega.\vspace{-0.1in}
\end{equation}
Consider\vspace{-0.05in}
\begin{equation}\label{RD-abstract-system}
 u_i'(t)+A_i u_i(t)\ni
 g_i(t,u_1(t),\ldots,u_m(t)),
 \quad u_i(0)=u_i^0,
 \quad\text{for all }i=1,\ldots,m,
\end{equation}
where the right-hand side is interpreted through
\eqref{RD-Nemytskii}.
Assume one of the following:
\begin{enumerate}
\item[(a$_1$)] for every $h>0$ and every bounded
$B\subset\bigcup_{t\in J}K^y(t)$, the set $S(h)B$ is relatively compact
in $X$;
\item[(a$_2$)] for every $0\leq s<T$, every $x\in K^y(s)$, every
$\phi\in L^1_+([s,T])$, and every sequence
$(w_n)\subset L^1([s,T];X)$ satisfying $\|w_n(t)\|\leq\phi(t)$ for
almost every $t\in[s,T]$ and all $n\geq1$, the set
$\{u(t;s,x,w_n):n\geq1\}$ is relatively compact in $X$ for every
$t\in(s,T]$;
\item[(b)] the local Lipschitz condition
\eqref{RD-finite-Lipschitz} holds.
\end{enumerate}
Then \eqref{RD-abstract-system} has a mild solution
$u\in C(J;X)$ satisfying
\begin{equation}\label{RD-invariant-bound}
 0\leq u_i(t,\xi)\leq y_i(t)
 \quad\text{for every }t\in J,
 \text{ a.e. }\xi\in\Omega,
 \quad\text{for all }i=1,\ldots,m.
\end{equation}
Under alternative \emph{(b)} the solution is unique among all mild
solutions which remain in the rectangle $K^y$, and the solution map is
continuous from $K^y(0)$, equipped with the norm of $X$, into $C(J;X)$.
\end{theorem}

\begin{proof}
Choose $R\geq1$ such that $y(t)\in[0,R]^m$ for every $t\in J$.
By Lemma~\ref{RD-Nemytskii-lemma}, $f$ is a jointly measurable
Carath\'eodory map on $\Gr(K^y)$ and is bounded there by
$\mu(\Omega)^{1/p}c_R$.  Put
$c_*(t):=\mu(\Omega)^{1/p}\bigl(c_R(t)+|y'(t)|\bigr)$.
The graph of $K^y$ is closed in $J\times X$, hence product measurable and
closed from the left, and $K_A^y(t)=K^y(t)$ because
$\overline{D(A)}=X$.  Since $y$ is continuous on the compact interval,
$K^y(J)$ is bounded.  Lemma~\ref{RD-rectangle-tangency} supplies the
regular subtangential condition and the exceptional condition; enlarging
the forcing and exceptional bounds to $c_*$ therefore gives exactly the
bounded effective-tube setting of Definition~\ref{bounded-effective-setting}.
Hence \emph{(SR1)} holds.

Under alternatives (a$_1$) and (a$_2$), respectively,
Theorem~\ref{ex1}(a) and (b) give a viable mild solution.  Under
alternative (b), Lemma~\ref{RD-Nemytskii-lemma} gives the local Lipschitz
condition required in Theorem~\ref{ex4}, which yields existence,
uniqueness, and continuous dependence.  Viability is exactly
\eqref{RD-invariant-bound}.
\end{proof}

The canonical choice of the upper function is supplied by the envelope
ODE.

\begin{corollary}[Canonical upper solution and global continuation]
\label{RD-envelope-corollary}
Let $A_i$ be $m$-$T$-accretive in $L^p(\Omega)$ with dense domain and
satisfy \eqref{RD-constant-supersolutions}, for $i=1,\ldots,m$, and let
$u^0\in L^p(\Omega;\mathbb R^m)\cap L^\infty(\Omega;\mathbb R_+^m)$.
Let $g:\mathbb R_+\times\mathbb R_+^m\to\mathbb R^m$ be such that its
restriction to every compact time interval satisfies
\emph{(R1)}--\emph{(R3)}, and put
$y_i^0:=\|u_i^0\|_{L^\infty(\Omega)}$.  Let $y$ be a nonnegative Carath\'eodory solution of
\eqref{RD-envelope-ODE} with $y(0)=y^0$, defined on its maximal interval
$[0,T_y)$.  Assume that one fixed alternative
\emph{(a$_1$)}, \emph{(a$_2$)}, or \emph{(b)} of
Theorem~\ref{RD-main-theorem} holds on every interval $[0,T_0]$,
$T_0<T_y$, for the rectangle generated by this $y$.  Then the
reaction--diffusion system has a mild solution on $[0,T_0]$ for every
$T_0<T_y$, and
\[
 0\leq u_i(t,\xi)\leq y_i(t)
 \qquad\mbox{for all }t\in[0,T_0]
\]
for almost every $\xi\in\Omega$ and every $i$.

If $T_y=\infty$, a global mild solution can be constructed successively
on consecutive compact time intervals.  Under \emph{(a$_1$)}, the semigroup compactness condition is inherited
by every later subinterval, whereas \emph{(a$_2$)} is already formulated
for arbitrary initial times.  A
sufficient condition for $T_y=\infty$ is
\begin{equation}\label{RD-global-envelope-growth}
 |\widehat g(t,z)|_p\leq a(t)(1+|z|_p)
 \qquad\text{for a.e. }t\geq0\text{ and all }z\in\mathbb R_+^m
\end{equation}
with $a\in L^1_{\rm loc}(\mathbb R_+)$.  Under this condition the
reaction--diffusion solution is global as well.
\end{corollary}

\begin{proof}
Since $y$ solves \eqref{RD-envelope-ODE}, it satisfies
\eqref{RD-upper-solution}.  Hence, for each $T_0<T_y$,
Theorem~\ref{RD-main-theorem} applies to the restrictions of $g$ and $y$ to
$[0,T_0]$.

If $T_y=\infty$, choose consecutive compact time intervals whose union is
$[0,\infty)$ and construct the solution recursively, using the endpoint of
one segment as the initial value for the next.  Under \emph{(a$_1$)}, the
compactness condition passes to every later interval; under
\emph{(a$_2$)}, Theorem~\ref{RD-main-theorem} applies at each new initial
time; and under \emph{(b)}, the local Lipschitz hypothesis is preserved
under restriction.  The continuation identity for prescribed-forcing mild
solutions then shows that the concatenation is a global mild solution.
Finally, \eqref{RD-global-envelope-growth} and Gronwall's inequality exclude
finite-time blow-up of the finite-dimensional solution $y$.
\end{proof}

\begin{corollary}[Comparison of cooperative systems]\label{RD-comparison-corollary}
Let $A_i$ be $m$-$T$-accretive in $L^p(\Omega)$ with dense domain,
$i=1,\ldots,m$, and let $g$ satisfy \emph{(R1)}--\emph{(R2)}, the local
Lipschitz condition \eqref{RD-finite-Lipschitz}, and the Kamke condition
\eqref{RD-Kamke} for almost every $t$.  Let
$y:J\to\mathbb R_+^m$ be continuous, and let $u$ and $v$ be mild solutions
of \eqref{RD-abstract-system} such that
\[
 u(t),v(t)\in K^y(t)\qquad\text{for all }t\in J.
\]
If $u(0)\leq v(0)$, then $u(t)\leq v(t)$ for all $t\in J$.
\end{corollary}

\begin{proof}
Since the coordinate domains are dense, $\overline{D(A)}=X$ and hence the
effective sections are $K^y(t)$; these sections are closed under lattice
suprema.  Lemma~\ref{RD-Nemytskii-lemma} gives the required local state
Lipschitz estimate on $\Gr(K^y)$, while
Lemma~\ref{RD-quasimonotone-lifting} gives quasimonotonicity on each
$K^y(t)$ for almost every $t$.  The product diffusion operator is
$m$-$T$-accretive, so Theorem~\ref{direct-comparison} applies.
\end{proof}

\begin{corollary}[Periodic solution in a moving return rectangle]
\label{RD-periodic-corollary}
Let $A_i$ be $m$-$T$-accretive in $L^p(\Omega)$ with dense domain and
satisfy \eqref{RD-constant-supersolutions}, for $i=1,\ldots,m$.
Suppose that $g$ is defined on
$\mathbb R\times\mathbb R_+^m$, is $T$-periodic in time, and that its
restriction to $[0,T]\times\mathbb R_+^m$ satisfies
\emph{(R1)}--\emph{(R3)} and \eqref{RD-finite-Lipschitz}.  Let
$y\in W^{1,1}([0,T];\mathbb R_+^m)$ be an upper solution of
\eqref{RD-upper-solution} such that
\begin{equation}\label{RD-periodic-endpoint}
        y(T)\leq y(0).
\end{equation}
No periodicity of $y$ is required.  Assume also the semigroup compactness
condition \emph{(a$_1$)} of Theorem~\ref{RD-main-theorem} on $[0,T]$.
Then there exists an initial state $u_*\in K^y(0)$ such that the
$T$-periodic system on $\mathbb R$
\begin{equation}\label{RD-periodic-system}
 u_i'(t)+A_i u_i(t)\ni
 g_i(t,u_1(t),\ldots,u_m(t)),
 \quad u_i(0)=u_{*,i},\quad i=1,\ldots,m.
\end{equation}
has a nonnegative $T$-periodic mild solution.  Its restriction to $[0,T]$
satisfies \eqref{RD-invariant-bound}.
\end{corollary}

\begin{proof}
Condition \eqref{RD-periodic-endpoint} gives $K^y(T)\subset K^y(0)$.
Enlarge the common null set from the beginning of this section to include
the exceptional sets in \eqref{RD-finite-Lipschitz}, perform the resulting
modification on the half-open period $[0,T)$, and extend the representative
$T$-periodically.  This does not change the mild equation and makes the
reaction field and its Nemytskii map strictly periodic.  The set $K^y(0)$ is
nonempty, bounded, closed, and convex in $X$.

The verification in Theorem~\ref{RD-main-theorem}, together with
Lemma~\ref{RD-rectangle-tangency}, gives the regular and exceptional
subtangential conditions on the original effective tube $K^y$.  The
Nemytskii map is bounded on $\Gr(K^y)$ by one function in $L^1_+([0,T])$,
and hence satisfies the linear-growth hypothesis
\eqref{periodic-linear-growth}.  Put
$b_i:=\max_{0\leq t\leq T}y_i(t)$ and
$D_*:=\{u\in X:0\leq u\leq b\}$.  The $T$-periodic Nemytskii map
induced by $g$ satisfies the separate Carath\'eodory conditions
\eqref{fixed-product-carath-cond} and is locally Lipschitz on the fixed
state domain $D_*$, which contains $K^y(t)$ for every $t\in[0,T]$.
Together with \emph{(a$_1$)} and the return inclusion, these facts verify
all hypotheses of Theorem~\ref{periodic-theorem}, with
$K_A(T)=K^y(T)\subset K^y(0)=K_A(0)$.  Its fixed point supplies the initial
state $u_*$ in \eqref{RD-periodic-system}.  The resulting segment repeats
as a mild solution because the reaction field is $T$-periodic.
\end{proof}

\begin{remark}[Static rectangles and the role of uniqueness]
\label{RD-static-periodic-remark}
Assume that $g$ is $T$-periodic and satisfies \emph{(R1)}--\emph{(R3)},
but not necessarily \eqref{RD-finite-Lipschitz}.  If there is a
$b\in\mathbb R_+^m$ such that
\[
 \widehat g_i(t,b)\leq0
 \quad\text{for a.e. }t\in[0,T],\qquad i=1,\ldots,m,
\]
then $K_b=\{u\in X:0\leq u\leq b\}$ is a static rectangle satisfying
the corresponding lower- and upper-face conditions.
Here $K_{b,A}=K_b$ because $\overline{D(A)}=X$.  The set $K_b$ is closed,
bounded, convex, and proximinal, with metric projection given by pointwise
clipping, and its resolvent invariance follows from
\eqref{RD-static-interval-invariance}.  Lemmas~\ref{RD-Nemytskii-lemma} and~\ref{RD-rectangle-tangency} give
an integrably bounded Carath\'eodory
forcing satisfying the ordinary subtangential condition.  Setting the
reaction to zero on the common exceptional null set and extending the
representative periodically preserves this condition at every time.  Hence,
if $S(h)K_b$ is relatively compact for every $h>0$, the static periodic
theorem \cite[Theorem~3]{Bo12} yields a $T$-periodic mild solution in $K_b$
without uniqueness of the initial-value problem.

The moving result above is complementary.  It requires only the return
inequality $y(T)\leq y(0)$ and allows the bounds to vary over the period, but
uses local Lipschitz well-posedness to obtain a single-valued Poincar\'e
operator.  The tangency-preserving approximation in \cite{Bo12} relies on a
fixed convex tangent geometry and does not directly preserve either the
time-dependent upper-face inequalities or the variable graph of a moving
rectangle.
\end{remark}

\section*{Statements and Declarations}
\noindent\textbf{Competing interests.}
The author declares no competing interests.

\smallskip
\noindent\textbf{Data availability.}
No data were generated or analysed in this study.

\smallskip
\noindent\textbf{Use of generative AI.}
OpenAI's ChatGPT was used for language and structural editing and for
auxiliary consistency checks.  The author reviewed and verified the resulting
text and assumes full responsibility for its content.


\begin{thebibliography}{99}
\small
\setlength{\itemsep}{0pt}
\setlength{\parsep}{0pt}
\bibitem{AHN18} S.\ Adly, A.\ Hantoute, B.T.\ Nguyen,
Invariant sets and Lyapunov pairs for differential inclusions with
maximal monotone operators, J. Math. Anal. Appl. {\bf 457}, 1017--1037
(2018).  \url{https://doi.org/10.1016/j.jmaa.2017.04.059}.
\bibitem{AHT12} S.\ Adly, A.\ Hantoute, M.\ Th\'era,
Nonsmooth Lyapunov pairs for infinite-dimensional first-order
differential inclusions, Nonlinear Anal. {\bf 75}, 985--1008 (2012).
\url{https://doi.org/10.1016/j.na.2010.11.009}.
\bibitem{APS06} S.\ Aizicovici, N.S.\ Papageorgiou, V.\ Staicu,
Periodic solutions of nonlinear evolution inclusions in Banach spaces,
J. Nonlinear Convex Anal. {\bf 7}, 163--177 (2006).
\bibitem{AmannODE} H.\ Amann,
{\it Ordinary Differential Equations: An Introduction to Nonlinear Analysis},
De Gruyter Studies in Mathematics 13, Walter de Gruyter, Berlin--New York, 1990.
Translated from the German by G.\ Metzen; German original:
{\it Gew\"ohnliche Differentialgleichungen}, Walter de Gruyter,
Berlin--New York, 1983.
\url{https://doi.org/10.1515/9783110853698}.
\bibitem{AubinCellina} J.-P.\ Aubin, A.\ Cellina,
{\it Differential Inclusions: Set-Valued Maps and Viability Theory},
Grundlehren der mathematischen Wissenschaften 264, Springer-Verlag,
Berlin--Heidelberg, 1984.
\url{https://doi.org/10.1007/978-3-642-69512-4}.
\bibitem{Bader98} R.\ Bader,
The periodic problem for semilinear differential inclusions in Banach
spaces, Comment. Math. Univ. Carolin. {\bf 39}, 671--684 (1998).
\url{https://eudml.org/doc/22372}.
\bibitem{Bader00} R.\ Bader,
On the semilinear multi-valued flow under constraints and the periodic
problem, Comment. Math. Univ. Carolin. {\bf 41}, 719--734 (2000).
\url{https://eudml.org/doc/248592}.
\bibitem{BaderKryszewski03} R.\ Bader, W.\ Kryszewski,
On the solution sets of differential inclusions and the periodic problem
in Banach spaces, Nonlinear Anal. {\bf 54}, 707--754 (2003).
\url{https://doi.org/10.1016/S0362-546X(03)00098-1}.
\bibitem{Barbu} V.\ Barbu, {\it Nonlinear Semigroups and Differential Equations
in Banach Spaces}, Noordhoff, Leyden, 1976.
\bibitem{Barbu2010} V.\ Barbu, {\it Nonlinear Differential Equations of
Monotone Types in Banach Spaces}, Springer, New York, 2010.
\url{https://doi.org/10.1007/978-1-4419-5542-5}.
\bibitem{Becker81} R.I.\ Becker,
Periodic solutions of semilinear equations of evolution of compact type,
J. Math. Anal. Appl. {\bf 82}, 33--48 (1981).
\url{https://doi.org/10.1016/0022-247X(81)90223-7}.
\bibitem{BenCr} P.\ B\'enilan, M.G.\ Crandall, Completely accretive operators, in
{\it Semigroup Theory and Evolution Equations}, Lecture Notes in Pure and
Applied Mathematics 135, Marcel Dekker, New York, 41--75 (1991).
\bibitem{BCP} P.\ B\'enilan, M.G.\ Crandall, A.\ Pazy,
{\it Nonlinear Evolution Equations in Banach Spaces}, preprint book,
Besan\c{c}on, 1994.
\bibitem{BeSh24} J.\ Beurich, P.\ Sharma,
Interpolation results for convergence of implicit Euler schemes with
accretive operators, Nonlinear Differ. Equ. Appl. {\bf 31},
Art.\ 109 (2024).
\url{https://doi.org/10.1007/s00030-024-01003-9}.
\bibitem{BressanColombo92} A.\ Bressan, G.\ Colombo,
Selections and representations of multifunctions in paracompact spaces,
Studia Math. {\bf 102}, 209--216 (1992).
\url{https://doi.org/10.4064/sm-102-3-209-216}.
\bibitem{Bo92} D.\ Bothe, Multivalued differential equations on graphs,
Nonlinear Anal. {\bf 18}, 245--252 (1992).
\url{https://doi.org/10.1016/0362-546X(92)90062-J}.
\bibitem{Bo9} D.\ Bothe, Flow invariance for perturbed nonlinear evolution
equations, Abstr. Appl. Anal. {\bf 1}, 417--433 (1996).
\url{https://doi.org/10.1155/S1085337596000231}.
\bibitem{Bo10} D. Bothe, Reaction-diffusion systems with discontinuities - A viability approach,
Proc.\ 2$^{\rm nd}$ World Congress of Nonlinear Analysts,
Nonlinear Anal. {\bf 30}, 677--686 (1997).
\url{https://doi.org/10.1016/S0362-546X(97)00247-2}.
\bibitem{Bo-1998} D.\ Bothe, Multivalued perturbations of $m$-accretive differential inclusions, Israel J. Math. {\bf 108}, 109--138 (1998).
\url{https://doi.org/10.1007/BF02783044}.
\bibitem{Bo12} D.\ Bothe, Periodic solutions of a nonlinear evolution problem
from heterogeneous catalysis, Differential and Integral Equations {\bf 14},
641--670 (2001).
\url{https://doi.org/10.57262/die/1356123241}.
\bibitem{Bo-JEE03} D.\ Bothe, Nonlinear evolutions with Carath\'eodory forcing,
J. Evol. Equ. {\bf 3}, 375--394 (2003).
\url{https://doi.org/10.1007/s00028-003-0099-5}.
\bibitem{Bo-JEE05} D.\ Bothe, Flow invariance for nonlinear accretive evolutions under range conditions,
J. Evol. Equ. {\bf 5}, 227--252 (2005).
\url{https://doi.org/10.1007/s00028-005-0185-z}.
\bibitem{BoWi12} D.\ Bothe, P.\ Wittbold, Abstract reaction--diffusion systems with
$m$-completely accretive diffusion operators and measurable reaction rates,
Commun. Pure Appl. Anal. {\bf 11}, 2239--2260 (2012).
\url{https://doi.org/10.3934/cpaa.2012.11.2239}.
\bibitem{Browder65} F.E.\ Browder,
Existence of periodic solutions for nonlinear equations of evolution,
Proc. Natl. Acad. Sci. USA {\bf 53}, 1100--1103 (1965).
\url{https://doi.org/10.1073/pnas.53.5.1100}.
\bibitem{calvert} B.\ Calvert, Nonlinear equations of evolution,
Pacific J. Math. {\bf 39}, 293--350 (1971).
\url{https://doi.org/10.2140/pjm.1971.39.293}.
\bibitem{CalvertT} B.\ Calvert, On $T$-accretive operators, Ann. Mat.
Pura Appl. {\bf 94}, 291--314 (1972).
\url{https://doi.org/10.1007/BF02413616}.
\bibitem{CaDaFr18} P.\ Cannarsa, G.\ Da Prato, H.\ Frankowska,
Invariance for quasi-dissipative systems in Banach spaces, J. Math.
Anal. Appl. {\bf 457}, 1173--1187 (2018).
\url{https://doi.org/10.1016/j.jmaa.2016.11.087}.
\bibitem{CaDaFr20} P.\ Cannarsa, G.\ Da Prato, H.\ Frankowska,
Domain invariance for local solutions of semilinear evolution equations
in Hilbert spaces, J. Lond. Math. Soc. (2) {\bf 102}, 287--318 (2020).
\url{https://doi.org/10.1112/jlms.12320}.
\bibitem{CaMot} O.\ C\^arj\u{a}, D.\ Motreanu, Characterization of Lyapunov pairs in the nonlinear case and applications, Nonlinear Anal. {\bf 70}, 352--363 (2009).
\url{https://doi.org/10.1016/j.na.2007.12.004}.
\bibitem{CV} O.\ C\^arj\u{a}, I.I.\ Vrabie, Viable domains for differential equations
governed by Carath\'eodory perturbations of nonlinear $m$-accretive operators,
pp.\ 109--130 in Lecture Notes in Pure and Appl.\ Math. {\bf 225},
Marcel Dekker, New York (2002).
\url{https://doi.org/10.1201/9780203902189.ch8}.
\bibitem{CGMS22} C.\ Castaing, C.\ Godet-Thobie,
M.D.P.\ Monteiro Marques, A.\ Salvadori, Evolution problems with
$m$-accretive operators and perturbations, Mathematics {\bf 10}, Art.\ 317
(2022).  \url{https://doi.org/10.3390/math10030317}.
\bibitem{CaVa} C.\ Castaing, M.\ Valadier, {\it Convex Analysis and
Measurable Multifunctions}, Lecture Notes in Mathematics 580, Springer, Berlin, 1977.
\bibitem{CrLi} M.G.\ Crandall, T.M.\ Liggett, Generation of semi-groups of
nonlinear transformations on general Banach spaces, Amer. J. Math. {\bf 93},
265--298 (1971).
\url{https://doi.org/10.2307/2373376}.
\bibitem{CGK22} A.\ \v{C}wiszewski, G.\ Gabor, W.\ Kryszewski,
Invariance and strict invariance for nonlinear evolution problems with
applications, Nonlinear Anal. {\bf 218}, Art.\ 112756 (2022).
\url{https://doi.org/10.1016/j.na.2021.112756}.
\bibitem{MDE} K.\ Deimling, {\it Multivalued Differential Equations}, de Gruyter, Berlin, 1992.
\bibitem{DeLak79} K.\ Deimling, V.\ Lakshmikantham, Existence and comparison
theorems for differential equations in Banach spaces, Nonlinear Anal.
{\bf 3}, 569--575 (1979).
\url{https://doi.org/10.1016/0362-546X(79)90085-3}.
\bibitem{Engelking} R.\ Engelking, {\it General Topology}, revised and
completed ed., Heldermann, Berlin, 1989.
\bibitem{Fremlin4} D.H.\ Fremlin, {\it Measure Theory, Vol.\ 4:
Topological Measure Spaces}, Torres Fremlin, Colchester, 2013 edition.
\url{https://www1.essex.ac.uk/maths/people/fremlin/mt4.2013/index.htm}.
\bibitem{GhavidelRuess12} S.M.\ Ghavidel, W.M.\ Ruess,
Flow invariance for nonautonomous nonlinear partial differential delay
equations, Commun. Pure Appl. Anal. {\bf 11}, 2351--2369 (2012).
\url{https://doi.org/10.3934/cpaa.2012.11.2351}.
\bibitem{Herzog01} G.\ Herzog, Quasimonotonicity, Nonlinear Anal.
{\bf 47}, 2213--2224 (2001).
\url{https://doi.org/10.1016/S0362-546X(01)00346-7}.
\bibitem{HiranoShioji04} N.\ Hirano, N.\ Shioji,
Invariant sets for nonlinear evolution equations, Cauchy problems and
periodic problems, Abstr. Appl. Anal. {\bf 2004}, 183--203 (2004).
\url{https://doi.org/10.1155/S1085337504311073}.
\bibitem{Hirsch} M.W.\ Hirsch, H.L.\ Smith, Monotone dynamical systems, in
{\it Handbook of Differential Equations: Ordinary Differential Equations},
Vol.\ II, Elsevier B.V., Amsterdam, 239--357 (2005).
\bibitem{ItoKa} K.\ Ito, F.\ Kappel, {\it Evolution Equations and Approximations}, World Scientific, Singapore, 2002.
\bibitem{JiKe25} J.\ Jiang, C.\ Keller,
Viability for locally monotone evolution inclusions and lower
semicontinuous solutions of Hamilton--Jacobi--Bellman equations in infinite
dimensions, ESAIM Control Optim. Calc. Var. {\bf 31}, Art.\ 4 (2025).
\url{https://doi.org/10.1051/cocv/2024086}.
\bibitem{Kamke} E.\ Kamke, Zur Theorie der Systeme gew\"ohnlicher Differentialgleichungen II, Acta Math. {\bf 58}, 57--85 (1932).
\bibitem{KoSo} M.\ Kocan, P.\ Soravia, Lyapunov functions for infinite-dimensional systems, J. Funct. Anal. {\bf 192}, 342--363 (2002).
\url{https://doi.org/10.1006/jfan.2001.3910}.
\bibitem{KryszewskiGaborSiemianowski18} W.\ Kryszewski, D.\ Gabor,
J.\ Siemianowski,
The Krasnosel'skii formula for parabolic differential inclusions
with state constraints, Discrete Contin. Dyn. Syst. Ser. B {\bf 23},
295--329 (2018).
\url{https://doi.org/10.3934/dcdsb.2018021}.
\bibitem{Kucia} A.\ Kucia, Scorza--Dragoni type theorems, Fund. Math. {\bf 138}, 197--203 (1991).
\url{https://doi.org/10.4064/fm-138-3-197-203}.
\bibitem{Martin} R.\ Martin, {\it Nonlinear Operators and Differential Equations in Banach Spaces}, Wiley-Interscience, New York, 1976.
\bibitem{MartinSmith90} R.H.\ Martin, Jr., H.L.\ Smith, Abstract functional
differential equations and reaction--diffusion systems, Trans. Amer. Math.
Soc. {\bf 321}, 1--44 (1990).
\url{https://doi.org/10.1090/S0002-9947-1990-0967316-X}.
\bibitem{MartinSmith91} R.H.\ Martin, Jr., H.L.\ Smith, Reaction--diffusion
systems with time delays: monotonicity, invariance, comparison and
convergence, J. Reine Angew. Math. {\bf 413}, 1--35 (1991).
\url{https://doi.org/10.1515/crll.1991.413.1}.
\bibitem{Michael56} E.\ Michael, Continuous selections. II,
Ann. Math. (2) {\bf 64}, 562--580 (1956).
\bibitem{Mi} I. Miyadera, {\it Nonlinear Semigroups}, American Mathematical Society, Providence, RI, 1992.
\bibitem{Mueller} M.\ M\"uller, \"Uber das Fundamentaltheorem in der
Theorie der gew\"ohnlichen Differentialgleichungen, Math. Z. {\bf 26},
619--645 (1927).
\bibitem{Nagumo} M.\ Nagumo, \"Uber die Lage der Integralkurven gew\"ohnlicher
Differentialgleichungen, Proc. Phys.-Math. Soc. Japan, 3rd Ser., {\bf 24},
551--559 (1942).
\url{https://doi.org/10.11429/ppmsj1919.24.0_551}.
\bibitem{Paicu08} A.\ Paicu,
Periodic solutions for a class of differential inclusions in general
Banach spaces, J. Math. Anal. Appl. {\bf 337}, 1238--1248 (2008).
\url{https://doi.org/10.1016/j.jmaa.2007.04.053}.
\bibitem{Pavel77} N.H.\ Pavel, Invariant sets for a class of semilinear equations
of evolution, Nonlinear Anal. {\bf 1}, 187--196 (1977).
\url{https://doi.org/10.1016/0362-546X(77)90009-8}.
\bibitem{PavelBook} N.H.\ Pavel,
{\it Differential Equations, Flow Invariance and Applications},
Research Notes in Mathematics 113, Pitman, Boston--London--Melbourne, 1984.
\bibitem{Pazy81} A.\ Pazy, The Lyapunov method for semigroups of nonlinear
contractions in Banach spaces, J. Anal. Math. {\bf 40}, 239--262 (1981).
\url{https://doi.org/10.1007/BF02790164}.
\bibitem{peressini} A.L.\ Peressini, {\it Ordered Topological Vector Spaces}, Harper and Row, New York, 1967.
\bibitem{Pierre78} M.\ Pierre, Invariant closed subsets for nonlinear
semigroups, Nonlinear Anal. {\bf 2}, 107--117 (1978).
\url{https://doi.org/10.1016/0362-546X(78)90046-9}.
\bibitem{Pruss79} J.\ Pr\"uss,
Periodic solutions of semilinear evolution equations,
Nonlinear Anal. {\bf 3}, 601--612 (1979).
\url{https://doi.org/10.1016/0362-546X(79)90089-0}.
\bibitem{Ruess09} W.M.\ Ruess,
Flow invariance for nonlinear partial differential delay equations,
Trans. Amer. Math. Soc. {\bf 361}, 4367--4403 (2009).
\url{https://doi.org/10.1090/S0002-9947-09-04833-8}.
\bibitem{STD26} H.\ Saoud, M.\ Th\'era, M.N.\ Dao,
Geometric stability analysis for differential inclusions governed by
maximally monotone operators, Set-Valued Var. Anal. {\bf 34}, Art.\ 22 (2026).
\url{https://doi.org/10.1007/s11228-026-00810-9}.
\bibitem{Schaefer} H.\ H.\ Schaefer, {\it Banach Lattices and Positive Operators}, Springer, Berlin, 1974.
\bibitem{Smith95} H.L.\ Smith, {\it Monotone Dynamical Systems: An Introduction to
the Theory of Competitive and Cooperative Systems}, Mathematical Surveys and
Monographs 41, American Mathematical Society, Providence, RI, 1995.
\bibitem{Tol23} A.A.\ Tolstonogov,
Comparison theorems for evolution inclusions with maximal monotone
operators. $L^2$-theory, Sb. Math. {\bf 214}, 853--877 (2023).
\url{https://doi.org/10.4213/sm9736e}.
\bibitem{Tol24} A.A.\ Tolstonogov,
Evolution inclusions with state-dependent maximal monotone operators,
Proc. Steklov Inst. Math. {\bf 327}, Suppl. 1, S226--S238 (2024).
\url{https://doi.org/10.1134/S0081543824070174}.
\bibitem{Volkmann72} P.\ Volkmann, Gew\"ohnliche Differentialungleichungen mit
quasimonoton wachsenden Funktionen in topologischen Vektorr\"aumen, Math. Z.
{\bf 127}, 157--164 (1972).
\url{https://doi.org/10.1007/BF01112607}.
\bibitem{Vra1} I.I.\ Vrabie, Compactness methods and flow-invariance
for perturbed nonlinear semigroups, Anal. Stiin. Univ. Iasi {\bf 27}, 117--124 (1981).
\bibitem{Vrabie90} I.I.\ Vrabie,
Periodic solutions for nonlinear evolution equations in a Banach space,
Proc. Amer. Math. Soc. {\bf 109}, 653--661 (1990).
\url{https://doi.org/10.1090/S0002-9939-1990-1015686-4}.
\end{thebibliography}
\end{document}